\documentclass[a4paper,12pt]{article}
\usepackage{amsmath,amsfonts,amssymb,amsthm,amscd}
\usepackage[arrow, matrix, curve]{xy}
\usepackage[colorlinks=true,citecolor=blue]{hyperref}
\newcommand\blfootnote[1]{%
  \begingroup
  \renewcommand\thefootnote{}\footnote{#1}%
  \addtocounter{footnote}{-1}%
  \endgroup
}

\font\cmssl=cmss10 at 12 pt

\usepackage{slashed}
\usepackage[normalem]{ulem}

\newtheorem{thm}{Theorem}
\newtheorem{lem}[thm]{Lemma}
\newtheorem{prop}[thm]{Proposition}
\newtheorem{defn}[thm]{Definition}
\newtheorem{cor}[thm]{Corollary}
\newtheorem{rem}[thm]{Remark}
\newtheorem{notation}[thm]{Notation}
\newtheorem{exa}[thm]{Example}
\newtheorem{ass}[thm]{Assumption}

\title{$T$-duality for transitive Courant algebroids}

\date{\today}

\author{Vicente Cort\'es, Liana David}

\begin{document}

\maketitle

\begin{abstract}
\noindent
We develop a theory of T-duality for transitive Courant algebroids. 
We show that $T$-duality between transitive Courant algebroids $E\rightarrow M$ and $\tilde{E}\rightarrow \tilde{M}$ induces a map between the spaces of sections of the corresponding canonical weighted spinor bundles $\mathbb{S}_E$ and $\mathbb{S}_{\tilde{E}}$ intertwining the canonical Dirac generating operators. The map is shown to induce an isomorphism between the spaces 
of invariant spinors,  compatible with an isomorphism between the spaces of invariant sections 
of the Courant algebroids. The notion of invariance is defined after lifting the vertical parallelisms of the underlying torus bundles $M\rightarrow B$ and $\tilde{M}\rightarrow B$
to the Courant algebroids and their spinor bundles.  We prove a general existence result for $T$-duals under assumptions generalizing the cohomological integrality conditions for $T$-duality 
in the exact case. Specializing our construction, we find that the $T$-dual of an exact or a heterotic Courant algebroid is again exact or heterotic, respectively.

\blfootnote{{\it 2000 Mathematics Subject Classification:} MSC  53D18 (primary); 53C27 (secondary).\\
{\it Keywords:}  Courant algebroids, T-duality, Dirac generating operators.}
\end{abstract}
\newpage
\tableofcontents

\section{Introduction}
The concept of $T$-duality appeared first in theoretical physics as a duality between a pair of physical theories related by compactification of a common (possibly hidden) theory along circles of reciprocal radii. Examples include the famous duality between type IIA and type IIB string theories. More generally, it refers to 
an isomorphism between certain type of structures on a pair of torus bundles over the same manifold \cite{BHM}. Already in the case of circle bundles the topology of the bundle typically changes under T-duality \cite{BEM1,BEM2}. 

Precise formulations of T-duality are available in the framework of generalized geometry (in the sense of Hitchin) \cite{T-duality-exact,baraglia}.  Recall that the basic idea of generalized geometry is to replace the tangent bundle
$TM$  of a manifold $M$ by a Courant algebroid $E$. The first examples of 
Courant algebroids considered in the literature were the exact Courant algebroids. They are obtained 
from the generalized tangent bundle $\mathbb{T}M:= T^{*}M\oplus TM$ by twisting the canonical 
Dorfman bracket with a closed $3$-form. A more general class of transitive Courant algebroids is the class  of heterotic Courant algebroids considered in \cite{baraglia}. These were introduced by Baraglia and Hekmati \cite{baraglia} motivated by T-duality in heterotic string theory. Cavalcanti and Gualtieri \cite{T-duality-exact} developed a theory of $T$-duality for exact Courant algebroids and Baraglia and Hekmati \cite{baraglia} extended it to heterotic Courant algebroids. The approach of \cite{baraglia} is based on reduction of exact Courant algebroids and uses T-duality for the latter algebroids.

In this article we develop T-duality for transitive Courant algebroids. 
Our theory applies to general 
Courant algebroids, which might not arise from reduction 
of an exact Courant algebroid. In fact, it does not use reduction. Our main focus is the systematic study 
of the interplay of T-duality with Dirac generating operators. 

Let $M$ and $\tilde{M}$ be 
principal $k$-torus bundles over a manifold $B$. We call two transitive Courant algebroids 
$E$ and $\tilde{E}$ over $M$ and $\tilde{M}$, respectively,  $T$-dual if there exists a certain type of 
isomorphism between the pullbacks of $E$ and $\tilde{E}$ to the fiber product $N=M\times_B\tilde{M}$ (see Definition \ref{def-T-duality}  for details). 
We show that $T$-duality gives rise to a map between the spaces of sections of the corresponding canonical weighted spinor bundles $\mathbb{S}_E$ and $\mathbb{S}_{\tilde{E}}$ intertwining the canonical Dirac generating operators, see Theorem \ref{thm-T-duality}. More specifically, we obtain compatible 
isomorphisms between the spaces of (appropriately defined) invariant sections of $E$ and $\tilde{E}$ 
as well as between the spaces of  invariant sections of $\mathbb{S}_E$ and $\mathbb{S}_{\tilde{E}}$. 
This implies, in particular, that any invariant geometric structure on the 
Courant algebroid $E$ gives rise to a corresponding invariant `$T$-dual' geometric structure on $\tilde{E}$.  
A structure solving a system of partial differential equations defined in terms of the Courant algebroid structure  on $E$ will be mapped to a solution of the corresponding system on $\tilde{E}$. Examples 
include integrability equations as considered in \cite{cortes-david-MMJ} and equations of motion of physical theories such as supergravity. For instance, it was shown in \cite[Section 7]{garcia} that the Hull-Strominger system is 
invariant under T-duality. 
We plan to investigate these type of applications in the future.

In Theorem \ref{statement-T-dual} we prove the existence of a $T$-dual $\tilde{E}$ for a class of transitive Courant algebroids  $E$ over a principal torus bundle $M\rightarrow B$ 
under the assumption that certain 
cohomology classes in $H^2(B,\mathbb{R})$ are integral.   The result generalizes a theorem of Bouwknegt, Hannabuss, and Mathai \cite{BHM} in the exact case, see Section \ref{exact:sec}. In the 
heterotic case we show that the `T-dual' Courant algebroids obtained from our 
construction are again heterotic, see Proposition \ref{T-dual-het-prop}.

In this paper we only consider Courant algebroids with scalar product of neutral signature.
It would be interesting to develop $T$-duality for other classes of Courant algebroids including, in particular, the `odd exact' Courant
algebroids studied in \cite{rubio}. A first step in this direction would be to develop a theory of Dirac generating operators for such Courant algebroids.\\[1.0em]
{\bf Acknowledgements.}  
Research of V.C.\  was partially funded by the Deutsche Forschungsgemeinschaft (DFG, German Research Foundation) under Germany's Excellence Strategy -- EXC 2121 Quantum Universe  
-- 390833306.
L.D. was partially supported by the UEFISCDI research grant PN-III-P4-ID-PCE-2020-0794, project title
''Spectral Methods in Hyperbolic Geometry''.\\

\section{Preliminary material}

To keep the text reasonably self-contained, we recall, 
following \cite{chen,cortes-david-MMJ},  basic facts we need on
transitive  Courant 
algebroids and their  canonical Dirac generating operator. We assume that the reader is familiar with the definition of
Courant algebroids,  Dirac generating operators, generalized connections and $E$-connections. 
Basic facts on these notions
can be found e.g.\ in \cite{cortes-david-MMJ},  the approach and notation of which we preserve along the paper.
In this paper we always assume that the Courant algebroids  have  scalar product
of neutral signature. 
For the definition of densities we keep  the  conventions from 
our previous work \cite{cortes-david-MMJ} which coincide with those from 
\cite{berline}. 
Namely,  if $V$ is a vector space of 
dimension $n$ and $s\in \mathbb{R}$, then 
the one-dimensional oriented   vector space $| \mathrm{det} V^{*} |^{s}$ of 
$s$-densities on  $V$ consists of all  maps
$\Psi : \Lambda^{n}V\setminus \{ 0\} \rightarrow \mathbb{R}$  (called $s$-densities) which satisfy   
$\omega ( \lambda \vec{v}) = | \lambda |^{s} \omega (\vec{v})$, for any
$\vec{v}\in \Lambda^{n}\setminus \{ 0\}$ and $\lambda \in \mathbb{R}\setminus \{ 0 \}.$ 
Note that, when $s$ is an integer, $| \mathrm{det} V^{*} |^{s}$ is canonically isomorphic to  
$|\mathrm{det} V^{*} |^{\otimes s}$ and $| \mathrm{det} V^{*} |^{2s}$ to $( \mathrm{det} V^{*})^{2s}$.
Any  form  $\omega \in \Lambda^{n} V^{*}$  defines an $s$-density
$| \omega |^{s} (\vec{v}) = | \omega (v_{1}, \cdots , v_{n} ) |^{s}$, 
where $\vec{v}:= v_{1}\wedge \cdots \wedge v_{n}$. 
If  $V$ is oriented  then we will identify 
$\Lambda^{n} V^{*}$ and 
$|  \mathrm{det}\ V^{*} |$  by the isomorphism which assigns to  a positively  oriented volume form $\omega \in \Lambda^{n} V^{*}$ 
the density $ | \omega |$;  the $s$-density $| \omega |^{s}$ will be denoted by 
$\omega^{s}$ when $\omega$ is positively  oriented and  $| \mathrm{det } V^{*} |^{s}$  by  $( \mathrm{det } V^{*} )^{s}$
when $V$ is oriented.
The same notation will be used when $V$ is replaced by a  vector bundle.\

\subsection{The canonical Dirac generating operator}

Let $(E, \pi ,  [\cdot , \cdot ] , \langle \cdot , \cdot \rangle )$ be a regular Courant algebroid over a  manifold $M$, with anchor $\pi : E \rightarrow TM$, Dorfmann bracket $[\cdot , \cdot ]$ and scalar product (of neutral signature) $\langle \cdot , \cdot \rangle$. 
Let  
$S$ be an irreducible $\mathrm{Cl} (E)$-bundle
(sometimes called a {\cmssl spinor bundle} of $E$).  We denote by $E\ni v\mapsto \gamma_{v}$ the Clifford action of $E$ on $S$. 
We assume that $S$ is  $\mathbb{Z}_{2}$-graded and that the gradation is compatible
with the Clifford multiplication (this always holds when $E$ is oriented).
Let $| \mathrm{det}\, S^{*}|^{1/r}$ be the line bundle of $1/r$-densities on $S$, where  $r:=\mathrm{rk}\, S$. 
An $E$-connection $D$ on $S$ induces an $E$-connection on $| \mathrm{det}\, S^{*}|^{1/r}$:
if  $\mathrm{vol}_{S} \in \Gamma (\Lambda^{r} S^{*})$ is a local volume form   on $S$ and 
$D_{e} \mathrm{vol}_{S} =\omega (e) \mathrm{vol}_{S}$ then $D_{e}| \mathrm{vol}_{S}|^{1/r}
=\frac{1}{r}\omega (e) |\mathrm{vol}_{S}|^{1/r}$, for any $e\in E.$

The canonical Dirac generating operator $\slashed{d}$ of $E$ acts on sections of the canonical weighted  spinor bundle
of $E$ determined by $S$.   The latter is defined by
\begin{equation}\label{can-spinor}
\mathbb{S} := S\otimes | \mathrm{det}\, S^{*}|^{1/r} \otimes | \mathrm{det}\, T^{*}M|^{1/2}=
\mathcal S \otimes L,
\end{equation}
where $\mathcal S : = S\otimes | \mathrm{det}\, S^{*}|^{1/r}$ is the canonical spinor bundle of $S$
and $L:= |\mathrm{det}\, T^{*}M|^{1/2}$. 
The operator $\slashed{d}:\Gamma (\mathbb{S}) \rightarrow \Gamma (\mathbb{S})$ is given by 
\begin{equation}\label{dirac-reg}
\slashed{d} = \slashed{D} +\frac{1}{4} \gamma_{T^{D}},
\end{equation}
where $\slashed{D} := \frac{1}{2} \sum_{i} \gamma_{\tilde{e}_{i}} D^{\mathbb{S}}_{e_{i}} $ is the 
Dirac operator computed
with $D^{\mathbb{S}}:= D^{\mathcal S} \otimes D^{L}$, 
$D^{\mathcal S}$ is the $E$-connection on $\mathcal S$  
induced by an arbitrary $E$-connection $D^{S}$ on $S$ compatible with a given 
generalized connection $D$ on $E$, $D^{L}$ is the $E$-connection on $L$
defined by $D$ by the rule 
\begin{equation} \label{rule:eq}
D^{L}_{v}(\mu )={\mathcal L}_{\pi (v)} \mu -\frac{1}{2} \mathrm{div}_{D}(v)\mu ,\ v\in E,\ \mu \in \Gamma (L),
\end{equation}
where $\mathrm{div}_{D}(v):= \mathrm{tr}\, (Dv)$,  
$( e_{i} )$ is a frame of $E$, $( \tilde{e}_{i} )$ the metrically dual frame
(i.e.\ $\langle e_{i}, \tilde{e}_{j}\rangle = \delta_{ij}$) and $T^{D}\in \Gamma (\Lambda^{3} E^{*})$ 
is the torsion of $D$, viewed  as a section of the Clifford bundle $\mathrm{Cl}(E)$ and acting by Clifford multiplication on $\mathbb{S}.$   
The definition of $\slashed{d}$ is independent of the choice of generalized connection $D$
and $D$-compatible $E$-connection $D^{S}$.

\subsection{Transitive Courant algebroids}\label{transitive-sect}

\subsubsection{Basic properties}\label{trans-sect-basic}
Recall that a scalar product on a Lie algebra is called {\cmssl invariant}, if the adjoint representation acts by skew-symmetric 
endomorphisms. A Lie algebra endowed with an invariant scalar product is called a {\cmssl quadratic Lie algebra}. 

Similarly,  a vector bundle $\mathcal{G} \rightarrow M$ endowed with a tensor field 
${[}\cdot , \cdot{]}\in \Gamma (\wedge^2\mathcal{G}^*\otimes \mathcal{G})$ satisfying the Jacobi identity is called a {\cmssl Lie algebra bundle}
if in a neighborhood of every point $p\in N$ the tensor field has constant coefficients with respect to some local frame. 
A {\cmssl bundle of quadratic Lie algebras} is a Lie algebra bundle $(\mathcal{G}, {[}\cdot , \cdot{]})$ endowed with an invariant metric $\langle \cdot , \cdot \rangle \in \Gamma (\mathrm{Sym}^2\mathcal{G}^*)$,  which we assume of neutral signature. 

Let  $(\mathcal G , [\cdot , \cdot ]_{\mathcal G},
\langle \cdot , \cdot \rangle_{\mathcal G})$ be a bundle of quadratic Lie algebras  over  a manifold $M$  and  $E:= T^{*}M \oplus \mathcal G \oplus TM$.
Let $\mathrm{pr} _{\mathcal G}$ and $\mathrm{pr}_{TM}$  be the natural projections from $E$ to $\mathcal G$ and $TM$ respectively.
As proved in  Theorem 2.3 of \cite{chen}, 
any Courant algebroid with underlying bundle  $E$,    anchor $\mathrm{pr}_{TM}$, 
scalar product 
$$
\langle \xi + r_{1} + X, \eta + r_{2}+ Y \rangle = \frac{1}{2} ( \eta (Y) +\xi (X)) + \langle r_{1},
r_{2} \rangle_{\mathcal G},\ \xi , \eta\in T^{*}M,\ r_{1}, r_{2}\in \mathcal G ,
$$
and whose Dorfman bracket satisfies
$$
\mathrm{pr}_{\mathcal G} [r_{1}, r_{2} ] = [r_{1}, r_{2}]_{\mathcal G}
$$
is defined by data $(\nabla , R, H)$  where $\nabla$ is  a connection on the vector bundle $\mathcal G$,
$R\in \Omega^{2}(M, \mathcal G )$ and $H\in \Omega^{3}(M)$ such that  $\nabla$ preserves $\langle \cdot , \cdot \rangle_{\mathcal G}$ and $[\cdot , \cdot ]_{\mathcal G}$,  the curvature $R^{\nabla}$ of $\nabla$  is given by 
\begin{equation}\label{cond-curv}
R^{\nabla}(X, Y) r= [ R(X, Y), r]_{\mathcal G},\ X, Y\in \mathfrak{X}(M),\  r\in \Gamma (\mathcal G),
\end{equation}
and the following relations hold:
\begin{align}
\label{dnablaR:eq}& 
d^\nabla R =0,\\
\label{cond-system}
 & d H= \langle R \wedge R\rangle_\mathcal{G}. 
 \end{align}
We recall that 
\begin{eqnarray*} (d^\nabla R)(X,Y,Z) &:=& \sum_{\mathfrak{S}(X, Y, Z)} \left( \nabla_{X} (R(Y, Z)) - R({\mathcal L}_{X}Y, Z) \right)\quad\mbox{and}\\
\langle R \wedge R\rangle_\mathcal{G} (X,Y,Z,W) &:=&  
2\sum_{\mathfrak{S}(X, Y, Z)}  \langle  R(X, Y), R(Z,W)\rangle_{\mathcal G},\end{eqnarray*}
where  $X, Y, Z, W\in {\mathfrak X}(M)$ and $\mathfrak{S}(X, Y, Z)$ denotes  cyclic permutations over $X, Y,Z$.
The Dorfman bracket of $E$ is uniquely determined by the relations 
\begin{align}
\nonumber& [X, Y]=   {\mathcal L}_{X}Y + R(X, Y) + i_{Y} i_{X} H\\
\nonumber& [X, r] = \nabla_{X} r - 2 \langle i_{X}R, r\rangle_{\mathcal G}\\
\nonumber& [r_{1}, r_{2}]= [r_{1}, r_{2} ]_{\mathcal G} + 2\langle \nabla r_{1} , r_{2} \rangle_{\mathcal G} \\
\label{expr-courant} & [X, \eta ] = {\mathcal L}_{X}\eta ,\ 
[\eta_{1}, \eta_{2}] = [r, \eta ] =0,
\end{align}
for any $X, Y\in {\mathfrak X}(M)$, $\eta_{1}, \eta_{2} , \eta \in \Omega^{1}(M)$ and $r,r_{1}, r_{2} \in \Gamma (\mathcal G)$, 
together with the condition
\begin{equation}
[ u, v]+[ v, u ] = 2 d\langle u, v\rangle , u, v\in \Gamma (E). 
\end{equation}
Such a Courant algebroid is 
called a {\cmssl standard Courant algebroid}. 
As proved in \cite{chen}, 
any transitive Courant algebroid (i.e.\ a Courant algebroid with surjective anchor)
is isomorphic to a standard Courant algebroid.
A {\cmssl dissection} of a transitive Courant algebroid $E$ is an isomorphism from  $E$ to a standard Courant algebroid.
The quadratic Lie algebra bundle  $(\mathcal G , [\cdot , \cdot ]_{\mathcal G},
\langle \cdot , \cdot \rangle_{\mathcal G})$  which is a summand of a dissection of $E$
is isomorphic to $\mathrm{Ker}\, \pi /  (\mathrm{Ker}\, \pi )^{\perp}$ (with scalar product and Lie bracket induced from $E$),
where $\pi : E \rightarrow TM$ is the anchor of $E$. 
The following  simple lemma holds.

\begin{lem}\label{c-d-k} Let $E$ be a transitive Courant algebroid with anchor $\pi : E \rightarrow TM$. Let
$(\mathcal G_{0}, [\cdot , \cdot ]_{0}, \langle \cdot , \cdot
\rangle_{0})$ be a quadratic Lie algebra bundle, isomorphic to  $\mathrm{Ker}\, \pi /  (\mathrm{Ker}\, \pi )^{\perp}$.
Then $E$ admits a dissection $I_{0}: E \rightarrow T^{*}M\oplus \mathcal G_{0}\oplus TM.$
\end{lem}

\begin{proof} Start with an arbitrary dissection $I: E \rightarrow T^{*}M\oplus \mathcal G\oplus TM$,
where the target is defined by data $(\nabla , R, H)$ and a quadratic Lie algebra bundle
$(\mathcal G, [\cdot , \cdot ]_{\mathcal G}, \langle \cdot , \cdot
\rangle_{\mathcal G})$
The new data 
$$
\tilde{\nabla}_{X} := K \nabla_{X} K^{-1},\ \tilde{R}(X, Y):= KR(X, Y),\ \tilde{H}(X, Y, Z) := H(X, Y, Z),
$$
where $K : \mathcal G \rightarrow \mathcal G_{0}$ is an isomorphism of quadratic Lie algebra bundles,
together with   $(\mathcal G_{0}, [\cdot , \cdot ]_{0}, \langle \cdot , \cdot
\rangle_{0})$,  define a standard Courant algebroid isomorphic to $ T^{*}M\oplus \mathcal G\oplus TM$
(use relations (\ref{iso-chen})  below with $\Phi :=0$ and $\beta :=0$).  By composing this isomorphism with $I$ we obtain the required dissection of $E$.
\end{proof}

Let $E_{i} := T^{*}M \oplus \mathcal G_{i} \oplus TM$  ($i=1,2$) be two  standard  Courant algebroids
over a manifold $M$, defined by quadratic Lie algebra bundles $(\mathcal G_{i}, [\cdot , \cdot]_{\mathcal G_{i}},
\langle \cdot ,\cdot \rangle_{\mathcal G_{i}})$ and data $(\nabla^{(i)}, R_{i}, H_{i}).$ As proved in Proposition 2.7 of  \cite{chen},
any fiber preserving Courant algebroid  isomorphism $I_{E}: E_{1} \rightarrow E_{2}$  is of the form
\begin{equation}\label{concrete-iso}
 I_{E}(\eta ) = \eta ,\  I_{E}(r) = -2 \Phi^{*} K(r) + K(r),\ I_{E}(X) = i_{X}\beta -\Phi^{*}\Phi (X) + \Phi (X) +X,
\end{equation}
for any $X\in {\mathfrak X}(M)$, $r\in \Gamma (\mathcal  G_{1})$ and $\eta \in \Omega^{1}(M)$.
Above $\beta \in \Omega^{2}(M)$, $K\in \mathrm{Isom} (\mathcal G_{1}, \mathcal G_{2})$ 
is an isomorphism of quadratic Lie algebra bundles, 
$\Phi \in \Omega^{1}(M,  \mathcal G_{2})$, 
\begin{align*}
\Phi^{*} \Phi : TM \rightarrow T^{*}M,\  ( \Phi^{*}\Phi ) (X) (Y):= \langle \Phi (X), \Phi (Y)\rangle_{\mathcal G_{2}},\\
\Phi^{*} K : \mathcal G_{1} \rightarrow T^{*}M,\  (\Phi^{*} K )(r) (X):= \langle K(r), \Phi (X)\rangle_{\mathcal G_{2}},
\end{align*}
for any $X, Y\in {\mathfrak X}(M)$ and $r\in \Gamma (\mathcal G_{1})$, 
and the next relations are satisfied:
\begin{align}
\nonumber \nabla^{(2)}_{X} r&= K \nabla^{(1)}_{X}( K^{-1} r) + [r, \Phi (X)]_{\mathcal G_{2}},\\
\nonumber  K R_{1} (X, Y) - R_{2}(X, Y) &= (d^{{\nabla}^{(2)}} \Phi )(X, Y)  + [\Phi (X), \Phi (Y) ]_{\mathcal G_{2}},\\ 
\label{iso-chen} H_{1} - H_{2}& =  d\beta  +
\langle (KR_{1} + R_{2} )\wedge \Phi\rangle_{\mathcal G_{2}}  - c_{3},
\end{align}
where $c_{3}(X, Y, Z) := \langle \Phi (X), [ \Phi (Y) , \Phi (Z)]_{\mathcal G_{2}} \rangle_{\mathcal G_{2}}$, for any
$X, Y, Z\in {\mathfrak X}(M)$.  
The second and third relations (\ref{iso-chen})  are equivalent  with relations
(46) and (47) of \cite{chen} (easy check) but are written in a simpler form. 
(We decomposed  $\mathrm{pr}_{T^{*}M} (I_{E}\vert_{TM})$, which in the notation of \cite{chen} is denoted by $\beta$,
into its symmetric part $- \langle \Phi ( \cdot ), \Phi (\cdot )\rangle_{\mathcal G_{2}}$ and skew-symmetric part 
$\beta$,  see relation  (44) of \cite{chen}). Under an additional condition on the Courant algebroids
$E_{i}$ relations (\ref{iso-chen}) can be further simplified,  see Lemma \ref{simplificare-second} below.
In the following we denote for simplicity by $\mathrm{Der}\, \mathcal{G}$ 
the bundle of skew-symmetric derivations of a bundle of quadratic Lie algebras $(\mathcal{G}, {[}\cdot , \cdot{]}_{\mathcal G}, \langle \cdot , \cdot \rangle_{\mathcal G})$. 
\begin{lem}\label{simplificare-second} Assume that the adjoint actions  $\mathrm{ad}_{\mathcal G_{i}} : \mathcal G_{i} \rightarrow \mathrm{Der} (\mathcal G_{i})$ of  the Lie algebra bundles $(\mathcal G_{i} , [\cdot , \cdot ]_{\mathcal G_{i}})$   of the 
standard Courant algebroids $E_{i}$  are  isomorphisms.
Then  the second relation (\ref{iso-chen}) follows from the first. 
\end{lem}

\begin{proof}
From the injectivity of $\mathrm{ad}_{\mathcal G_{2}}$,  the  second relation (\ref{iso-chen}) holds if and only if 
\begin{equation}\label{ce1} 
 [K R_{1} (X, Y) - R_{2}(X, Y),r]_{\mathcal G_{2}} = [ (d^{{\nabla}^{(2)}} \Phi )(X, Y)   +  [\Phi (X), \Phi (Y) ]_{\mathcal G_{2}}, r]_{\mathcal G_{2}} 
\end{equation}
for any $r\in\Gamma ( \mathcal G_{2} ).$   Taking the covariant derivative of the first relation (\ref{iso-chen}) we obtain
\begin{align}
\nonumber[\nabla_{Y}^{(2)} (\Phi (X)), r]_{\mathcal G_{2}} =& [ R_{2} (X, Y), r]_{\mathcal G_{2}} +\nabla^{(2)}_{\mathcal L_{X} Y} r
+ \nabla^{(2)}_{Y} ( K \nabla^{(1)}_{X} ( K^{-1} r))\\
 -  \label{cov-chen}&  K \nabla^{(1)}_{X} (K^{-1} \nabla^{(2)}_{Y}r). 
\end{align}
Now, a straightforward computation which uses the first relation
(\ref{iso-chen}), relation (\ref{cov-chen}),  and 
$$
(d^{ \nabla^{(2)}} \Phi )(X, Y) =\nabla^{(2)}_{X} ( \Phi (Y)) -  \nabla^{(2)}_{Y} ( \Phi (X)) - \Phi ({\mathcal L}_{X}Y)  
$$
shows that
\begin{align}
\nonumber [(d^{\nabla^{(2)}} \Phi )(X, Y), r]_{\mathcal G_{2}}& = [ K R_{1} (X, Y) - R_{2} (X, Y) ,r]_{\mathcal G_{2}}\\
\label{ce2}&  
+\left(  (\nabla_{X} K) (\nabla_{Y} K^{-1}) - (\nabla_{Y} K)(\nabla_{X} K^{-1})\right) ( r),
\end{align}
where $\nabla$ denotes the connection on $\mathrm{End} (\mathcal G_{1}, \mathcal G_{2})$ 
induced by $\nabla^{(1)}$ and $\nabla^{(2)}.$
On the other hand, using the Jacobi identity for $[\cdot , \cdot ]_{\mathcal G_{2}}$ and
the first relation (\ref{iso-chen}), we can compute
\begin{equation}\label{c3}
[ [\Phi (X), \Phi (Y) ]_{\mathcal G_{2}}, r]_{\mathcal G_{2}} = 
\left(  (\nabla_{Y} K) (\nabla_{X} K^{-1}) - (\nabla_{X} K)(\nabla_{Y} K^{-1})\right) ( r).
\end{equation}
Relations  (\ref{ce2}) and  (\ref{c3}) imply (\ref{ce1}). 
\end{proof}

We say that two dissections $I_{i}: E \rightarrow T^{*} M \oplus \mathcal G_{i} \oplus TM$ of a transitive Courant
algebroid $E$ are related by $(\beta , K, \Phi )$, 
where $\beta \in \Omega^{2}(M)$, $K\in \mathrm{Isom}(\mathcal G_{1}, \mathcal G_{2})$ and
$\Phi :\in \Omega^{1}(M,   \mathcal G_{2})$,  if the isomorphism $I_{2}\circ I_{1}^{-1}$ is given by  (\ref{concrete-iso}).\

The proof of the following proposition is straightforward.
\begin{prop} 
 If $I_{1} : E_{1} \rightarrow E_{2}$ and $I_{2} : E_{2} \rightarrow E_{3}$ are isomorphisms between 
standard Courant algebroids
$E_{i} = T^{*}M \oplus \mathcal G_{i} \oplus TM$, 
defined, according to (\ref{concrete-iso}), 
by  $(\beta_{1}, K_{1}, \Phi_{1})$ and $(\beta_{2}, K_{2}, \Phi _{2})$ respectively, 
then $I_{2}\circ I_{1} : E_{1} \rightarrow E_{3}$ is defined by $(\beta_{3}, K_{3}, \Phi_{3})$
where 
\begin{equation}\label{iso-standard-1}
K_{3} := K_{2} K_{1},\ \Phi_{3}: = \Phi_{2} + K_{2} \Phi_{1}
\end{equation}
and, for any $X, Y\in TM$, 
\begin{equation}\label{iso-standard-2} 
\beta_{3}(X, Y):=( \beta_{1}+\beta_{2})(X, Y) +\langle \Phi_{2}(X), K_{2} \Phi_{1}(Y)
\rangle_{{\mathcal  G}_{2}} -   \langle \Phi_{2}(Y), K_{2} \Phi_{1}(X)\rangle_{{\mathcal G}_{2}}.
\end{equation}
 In particular,
\begin{align}
\nonumber (\beta_{3} - \Phi^{*}_{3}\Phi_{3})(X, Y) &= (\beta_{1} - \Phi^{*}_{1}\Phi_{1})(X, Y) 
+  (\beta_{2} - \Phi^{*}_{2}\Phi_{2})(X, Y)\\
\label{iso-standard-3} &-2 \langle K_{2}\Phi_{1} (X), \Phi_{2}(Y)\rangle_{\mathcal G_{3}} .
\end{align} 
\end{prop}

\subsubsection{The canonical Dirac generating operator of a standard Courant algebroid}

Let $E= T^{*} M\oplus \mathcal G \oplus TM$ be a standard Courant algebroid as above and  
 $S_{\mathcal G}$  an  irreducible  $\mathrm{Cl} (\mathcal G)$-bundle, with canonical  spinor bundle
$\mathcal S_{\mathcal G}.$  
Then $S:= \Lambda (T^{*}M)\hat{\otimes} S_{\mathcal G}$ is an irreducible  spinor  bundle
of $E$, with Clifford action
\begin{equation}
\gamma_{\xi + r + X} (\omega \otimes s)=( i_{X}\omega + \xi \wedge \omega ) \otimes s + (-1)^{| \omega |} \omega 
\otimes (r\cdot s),
\end{equation}
for any $\xi\in T^{*}M$,  $ r\in \mathcal G$, $X\in TM$,  $\omega\in  \Lambda (T^{*}M)$ and  $s\in S_{\mathcal G}.$ 
The canonical weighted spinor bundle  of $E$
determined by $S$, as
defined in (\ref{can-spinor}),  
is canonically isomorphic to   
\begin{equation}\label{S-dissection}
\mathbb{S} = 
 \Lambda  (T^{*}M)\hat{\otimes} \mathcal S_{\mathcal G}
\end{equation}
owing to the canonical isomorphism
\begin{equation}\label{iso-standards}
| \mathrm{det}\, ( \Lambda (TM) \otimes S^{*}_{\mathcal G} |^{\frac{1}{Nr}} \otimes |\mathrm{det}\, T^{*}M|^{1/2}
\cong | \mathrm{det}\, S^{*}_{\mathcal G}|^{1/r}
\end{equation}
given by
\begin{equation}
|  ( Z_{1}\otimes s_{1}^{*})\wedge \cdots\wedge  (Z_{N}\otimes  s_{r}^{*}) |^{\frac{1}{Nr }}\otimes
|\alpha_{1}\wedge\cdots \wedge \alpha_{m}|^{1/2} \mapsto | s_{1}^{*}\wedge \cdots \wedge s_{r}^{*}|^{1/r},
\end{equation}
where $N:= \mathrm{rk}\, \Lambda (TM)$, 
$r:= \mathrm{rk}\, S_{\mathcal G}$, 
$ (s_{i}^{*})$ is a  local frame of $\mathcal{S}_{\mathcal G}^{*}$,  
$(  \alpha_{i})$ is a  local frame of $T^{*}M$, and 
$( Z_{i})$ is the   local frame  of $\Lambda  (TM)$ determined by the dual frame 
$( X_{i})$ of  $( \alpha_{i})$.\

As shown in Theorem 67 of  \cite{cortes-david-MMJ},  the canonical Dirac generating operator $\slashed{d} : \Gamma (\mathbb{S}) \rightarrow \Gamma (\mathbb{S})$ takes the form 
\begin{align}
\nonumber& \slashed{d} ( \omega \otimes s) = (d\omega - H\wedge \omega ) \otimes s + \nabla^{\mathcal S_{\mathcal G}}(s) \wedge \omega \\
\label{can-dirac} &+\frac{1}{4} (-1)^{| \omega | +1} \omega 
\otimes (C_{\mathcal G} s) + (-1) ^{ |\omega | +1} \bar{R}^{E}(\omega \otimes s),
\end{align}
where $\omega \in \Omega (M)$ and $s\in \Gamma ({\mathcal S}_{\mathcal G})$.
Above $C_{\mathcal G}\in \Gamma (\Lambda^{3} \mathcal G^{*})\subset \Gamma (\mathrm{Cl} (\mathcal G))$ is the Cartan form 
$C_{\mathcal G}(u, v, w) := \langle [u, v]_{\mathcal G}, w\rangle_{\mathcal G}$ which acts on $s$ by 
Clifford  multiplication,
$\nabla^{\mathcal S_{\mathcal G}}$ is a connection on $\mathcal S_{\mathcal G}$ 
induced by (any) connection $\nabla^{S_{\mathcal G}}$  on $S_{\mathcal G}$ compatible with $\nabla $, 
$$
\nabla^{{\mathcal S}_{\mathcal G}} (s)\wedge \omega = \sum_{i} \alpha_{i}\wedge \omega \otimes
(\nabla^{{\mathcal S}_{\mathcal G}}_{X_{i}} s)
$$
and
$$ 
\bar{R}^{E} (\omega \otimes s) =\frac{1}{2} \sum_{i,j,k} \langle R(X_{i}, X_{j}), r_{k}\rangle_{\mathcal G} 
(\alpha_{i}\wedge \alpha_{j} \wedge \omega )\otimes
(\tilde{r}_{k} s) ,
$$
where  $( r_{k})$ is  a local  frame of $\mathcal G$, $(\tilde{r}_{k})$ the metrically dual frame 
(i.e.\ $\langle r_{i}, \tilde{r}_{j} \rangle_{\mathcal G} = \delta_{ij}$ for any $i, j$)
and $\tilde{r}_{k}s$ is the Clifford action of $\tilde{r}_{k}$ on $s$.
Sometimes it will be convenient to write the canonical Dirac generating operator in the form 
\begin{align}
\nonumber& \slashed{d} ( \omega \otimes s) = (d\omega )\otimes s  - H\cdot ( \omega \otimes s )
+\sum_{i} \alpha_{i}\cdot (\omega \otimes \nabla_{X_{i}}^{\mathcal S_{\mathcal G}}s)\\
 \label{can-dirac-sometimes} &-\frac{1}{4} C_{\mathcal G} \cdot (\omega \otimes s) 
 -\frac{1}{2} \sum_{i,j,k} \langle R(X_{i}, X_{j}) ,r_{k} \rangle_{\mathcal G} \tilde{r}_{k}\cdot \alpha_{i}\cdot \alpha_{j}\cdot 
 (\omega \otimes s),
\end{align}
where the dots denote the Clifford action of $\mathrm{Cl}(E)\cong \Lambda E$ on $\mathbb{S}.$

\section{The bilinear pairing on spinors}\label{pairing-section}

In general one cannot define  the determinant of  a pairing on a vector space. However, when the vector space is of
the form 
$\mathcal V := V\otimes  |\det V^{*} |^{1/n}$, where $V$ is a vector space of dimension $n$, the determinant  of a 
pairing $\langle \cdot , \cdot \rangle$ on $\mathcal V$ is defined as follows.
Consider $\langle\cdot ,  \cdot\rangle$ as 
a map $\mathcal{V}\rightarrow \mathcal{V}^*$, $v\mapsto \langle v ,  \cdot\rangle$.
We define the determinant $\det \langle\cdot ,  \cdot\rangle $
as the induced map
$\det \mathcal{V}\rightarrow \det \mathcal{V}^*$, seen as a vector from the tensor product 
$(\det \mathcal{V}^*)^{\otimes 2}$. 
Since $\mathcal{V} = V \otimes | \det V^*|^{1/n}$, 
$\det  \mathcal{V} = \det V \otimes |\det V^* |$
and $(\det  \mathcal{V} )^{2} \cong (\det V)^2 \otimes |\det V^* |^{2}\cong (\det V)^2 \otimes (\det V^* )^{2}$
is  canonically isomorphic to $\mathbb{R}$. 
This  means that 
$\det \langle\cdot ,  \cdot\rangle$ is  a real number. It   can be computed as follows: 
let $( v_{i})$ be a 
basis of  $V$  and  $l:= | v_{1}\wedge \cdots \wedge v_{n}|^{-1/n}$.  The determinant 
$\mathrm{det}\langle\cdot , \cdot \rangle$ of $\langle \cdot , \cdot \rangle$
coincides  with the determinant of the matrix $A= (a_{ij})$ where 
$a_{ij} := \langle v_{i}\otimes l,  v_{j}\otimes l\rangle$. 
Note that
$\mathrm{det}\, \lambda  \langle\cdot , \cdot \rangle = \lambda^{n} \mathrm{det}  \langle\cdot , \cdot \rangle $,
for any $\lambda \in \mathbb{R}^{*}.$ 
The above considerations can  be extended to pairings on vector bundles in the obvious way.

Let $(E, \pi , [\cdot , \cdot ], \langle \cdot , \cdot \rangle_{E} )$ be a  rank $2n\ge 2$ regular Courant algebroid
over a  manifold $M$, $S$ an irreducible  spinor bundle of $E$  of  rank $r$ and  $\mathcal S
= S\otimes |\det S^{*} |^{1/r}$ the canonical
spinor bundle of $S$.

\begin{prop}\label{prop-pairing} i) 
For any $U\subset M$ open and sufficiently small, 
there  is a   pairing 
\begin{equation}\label{clifford-mult}
\langle \cdot , \cdot\rangle_{\mathcal{S}\vert_{U}}:\Gamma ( \mathcal{S}\vert_{U}) \times
\Gamma ( \mathcal{S}\vert_{U}) \rightarrow C^{\infty}(U)
\end{equation}
which is $C^{\infty}(U)$-linear, satisfies 
\begin{equation}\label{cliff-rel-spinors}
\langle u\cdot s, u\cdot \tilde{s}\rangle_{\mathcal{S}\vert_{U}} = \langle u, u\rangle_{E} \langle s, \tilde{s}\rangle_{\mathcal{S}\vert_{U}},\
u\in \Gamma (E\vert_{U}), s, \tilde{s}\in \Gamma (\mathcal{S}\vert_{U}) 
\end{equation}
and has determinant  one  if $n>1$ and $-1$ if $n=1$. 
Any two such pairings differ by multiplication by $\pm 1.$\ 

ii) If $n$ is even then  the even and odd parts $\mathcal{S}^{0}\vert_{U}$ and $\mathcal{S}^{1}\vert_{U}$  of $\mathcal{S}\vert_{U}$ are orthogonal with respect to
$\langle \cdot , \cdot\rangle_{\mathcal{S}\vert_{U}}$. If $n$ is odd then $\mathcal{S}^{0}\vert_{U}$ and 
$\mathcal{S}^{1}\vert_{U} $ are isotropic with respect to
$\langle \cdot , \cdot\rangle_{\mathcal{S}\vert_{U}}$.\ 

iii)
The pairing is symmetric if $n \equiv 0, 1\pmod{4}$ and skew-symmetric if $n \equiv 2, 3\pmod{4}$.\

iv) Let $D$ be a generalized connection on $E$. The pairing   $\langle \cdot ,\cdot\rangle_{\mathcal S\vert_{U} }$ is preserved by the $E$-connection
$D^{\mathcal S}$   induced by (any) $E$-connection $D^{S}$ on $S$, compatible with $D$.
\end{prop}

The remaining part of this section is devoted to the proof of Proposition \ref{prop-pairing} and to various corollaries. 
Let $V$ be an $n$-dimensional vector space.
We begin by considering the irreducible  $\mathrm{Cl}(V\oplus V^{*})$-module $\Lambda V^{*}$ 
where $V\oplus V^{*}$ is endowed with its standard metric of neutral signature $\langle X +\xi  , Y+\eta \rangle
=\frac{1}{2} ( \xi (Y) +\eta (X))$ and the Clifford action is given by 
$$
(X+\xi ) \omega := i_{X} \omega + \xi \wedge \omega ,\ X\in V,\  \xi \in V^{*},\ \omega \in\Lambda V^{*}. 
$$ 
It is well-known that the  vector valued bilinear pairing 
\begin{equation}\label{scalar-standard}
\langle \cdot , \cdot \rangle : \Lambda V^{*}\otimes \Lambda V^{*}\rightarrow \Lambda^{n}V^{*},\ 
\langle \omega , \tilde{\omega } \rangle := \left( \omega^{t} \wedge \tilde{\omega}\right)_{\mathrm{top}},
\end{equation}
where ${ }^{t}: \Lambda V^{*}\rightarrow \Lambda V^{*}$ is defined on decomposable forms by $(\alpha_{1}\wedge \cdots
\wedge 
\alpha_{k})^{t}:=\alpha_{k}\wedge \cdots \wedge \alpha_{1}$ and, for a form $\omega \in \Lambda V^{*}$,
$\omega_{\mathrm{top}}\in \Lambda^{n}V^{*}$ denotes its component of maximal degree, satisfies (\ref{cliff-rel-spinors})
(see e.g.\  \cite{gualtieri-thesis}). Since the metric of $V\oplus V^{*}$ has neutral signature, we obtain that
(\ref{scalar-standard})   is determined  
(up to  multiplication by a non-zero constant),  
by this property. Note that $\Lambda^{\mathrm{even}} V^{*}$ and $\Lambda^{\mathrm{odd}} V^{*}$ are orthogonal with respect
to the pairing (\ref{scalar-standard}) when $n$ is even and are isotropic when $n$ is odd.  Also it is easy to check that 
the above pairing  is non-degenerate,  symmetric  if $n \equiv 0, 1\pmod{4}$ and skew-symmetric if $n \equiv 2, 3\pmod{4}$. 
By choosing a volume form on $V$, we obtain an $\mathbb{R}$-valued pairing with the same properties.
A  pairing with such properties can be constructed on any 
irreducible Clifford  module in neutral signature.

\begin{lem}\label{pas1} 
Let $W$ be an irreducible $\mathrm{Cl}^{n,n}$-module. 
There is a canonical  (determined up to multiplication by $\pm1$)  $\mathbb{R}$-valued pairing 
$\langle \cdot , \cdot\rangle_{\mathcal W}$
on $\mathcal W :=W\otimes | \mathrm{det}\, W^{*} |^{1/r}$ (where $r:= \mathrm{rank}\,  W$)  
which satisfies (\ref{cliff-rel-spinors}) and $\mathrm{det} \langle \cdot , \cdot\rangle_{\mathcal W} =1$ 
if $n>1$,  respectively  $\mathrm{det} \langle \cdot , \cdot\rangle_{\mathcal W} =-1$ if $n=1$.  The pairing is symmetric if $n \equiv 0, 1\pmod{4}$ and skew-symmetric if $n \equiv 2, 3\pmod{4}$.
Moreover, the even and odd parts $W^{0}$ and $W^{1}$ are orthogonal with respect to 
$ \langle \cdot , \cdot\rangle_{\mathcal W}$ when $n$ is even and are isotropic 
when $n$ is odd.
\end{lem}

\begin{proof} It remains to prove 
that we can rescale $\langle \cdot , \cdot\rangle_{\mathcal W}$  such that $\mathrm{det}\langle \cdot ,\cdot\rangle_{\mathcal W} =1$ or $-1.$  Assume that $n>1.$  
Using $\mathrm{det}(\lambda \langle \cdot ,\cdot\rangle_{\mathcal W}  )
=\lambda^{r}\mathrm{det}( \langle \cdot ,\cdot\rangle_{\mathcal W}  )$
(when $r:= \mathrm{dim}\, W= 2^{n}$)
this reduces to showing that
$\mathrm{det}\, \langle \cdot ,\cdot\rangle_{\mathcal W} >0$ for any bilinear pairing $\langle \cdot , \cdot\rangle_{\mathcal W}$
which satisfies  (\ref{cliff-rel-spinors}).
For this, it is sufficient to compute $\mathrm{det}\, \langle \cdot ,\cdot\rangle_{\mathcal W}$ using a  
basis of $W$ of the form  $( w_{1}, \cdots , w_{r/2}, vw_{1}, \cdots , v w_{r/2})$ where 
$v\in \mathbb{R}^{n,n}$ satisfies $\langle v, v\rangle=1.$ 
We obtain $\mathrm{det}\,  \langle \cdot ,\cdot\rangle_{\mathcal W} =
(\mathrm{det}\, A)^{2}$, where $A = (A_{ij})\in M_{ r/2\times r/2}(\mathbb{R})$ with $ A_{ij} =\langle w_{i}\otimes l, w_{j}\otimes l\rangle_{\mathcal W} $ 
when $n$ is even, $ A_{ij} =\langle w_{i}\otimes l,  v w_{j}\otimes l\rangle_{\mathcal  W}$ when $n>1$ is odd and 
$l:= | w_{1}\wedge \cdots \wedge  w_{r/2}\wedge  v w_{1}\wedge \cdots \wedge  v w_{r/2}|^{-1/r}$ in both cases. For $n=1$ 
we obtain instead $\mathrm{det}( \langle \cdot ,\cdot\rangle_{\mathcal W} ) = - (\mathrm{det}\, A)^{2}$.
\end{proof}

\begin{rem}{\rm A canonical pairing 
$\langle \cdot , \cdot\rangle_{\mathcal W}$
can be obtained by starting with any (non-trivial) pairing $\langle \cdot , \cdot\rangle_{W}$
on $W$, which satisfies (\ref{cliff-rel-spinors}): take a  basis $( w_{i})$
of $W$ and let $l:= | w_{1}\wedge \cdots \wedge w_{r}|^{-1/r}.$ Then 
\begin{equation}
\langle s \otimes l, \tilde{s} \otimes l\rangle_{\mathcal W} = | \mathrm{det}\, C|^{-1/r}
\langle s, \tilde{s}\rangle_{W},\ C:= ( \langle w_{i}, w_{j}\rangle_{W})_{i,j}.
\end{equation}}
\end{rem}

The next lemma concludes the proof of Proposition \ref{prop-pairing}.

\begin{lem}\label{pas2}
Let $U$ be a sufficiently small open subset of $M$. The section of $(\mathcal S^{*}\otimes \mathcal S^{*})/ \pm 1$ defined by
the pairings $\langle \cdot , \cdot \rangle_{\mathcal S_{p}}$
from Lemma \ref{pas1}   lifts  to a smooth section 
$\langle \cdot , \cdot \rangle_{\mathcal S\vert_{U}}$ 
of $\mathcal S^{*}\vert_{U}\otimes \mathcal S^{*}\vert_{U}$,   which is
preserved by the $E$-connection  $D^{\mathcal S}$ on $\mathcal S$  induced by any  generalized connection $D$ on $E$.
\end{lem}

\begin{proof}    Assume that  $E\vert_{U}$ admits a local frame $(e_{i})$ with $\langle e_{i}, e_{j}\rangle_{E}= \epsilon_{i} \delta_{ij}$,
where $\epsilon_{i} =1$ for $i\leq n$ and $-1$ for $i\geq n+1$.
On $\mathbb{R}^{2n}$ 
we consider the standard basis $( v_{i}) $ and metric $\langle\cdot , \cdot \rangle_{\mathbb{R}^{2n}}$
defined by
$\langle v_{i}, v_{j}\rangle_{\mathbb{R}^{2n}}=\epsilon_{i} \delta_{ij}$. Let $V$ be an irreducible 
$\mathrm{Cl} (\mathbb{R}^{2n}, \langle\cdot , \cdot \rangle_{\mathbb{R}^{2n}})$-module and $\Sigma := U\times V$, which is an irreducible   $\mathrm{Cl}(E\vert_{U})$-bundle
with Clifford action $\gamma_{e_{i}} (p, w):= (p, v_{i}\cdot w)$, for any $(p, w)\in U\times V.$ 
Since $E$ has neutral signature,    $S\vert_{U} = \Sigma \otimes L$ where $L$ is a  line bundle and 
$$
{\mathcal S}\vert_{U} = \Sigma \otimes  | \mathrm{det}\, \Sigma^{*}|^{1/r} \otimes L \otimes  | L^{*} |,
$$
where $r:= \mathrm{dim}\, V$.
The bilinear pairing $\langle \cdot , \cdot \rangle_{\mathcal S\vert_{U}}$ we are looking for  is given by
$$
\langle  s\otimes l,\tilde{s} \otimes l\rangle_{\mathcal S}= \langle s, \tilde{s}\rangle_{\Sigma
\otimes | \mathrm{det}\,\Sigma^{*}|^{1/r} } l^{2},\ s, \tilde{s}\in
 \Sigma\otimes |\mathrm{det}\, \Sigma^{*}|^{1/r},\  
l\in L\otimes  | L^{*}|,
$$ 
where  $\langle s, \tilde{s}\rangle_{\Sigma
\otimes | \mathrm{det}\,\Sigma^{*}|^{1/r} }$ is the constant pairing on $\Sigma$ induced by a canonical 
$\mathbb{R}$-valued  pairing  on $V\otimes  | \mathrm{det}\,V^{*}|^{1/r} $ 
 (according to Lemma \ref{pas1})
and $l^{2}\in C^{\infty}(U)$ 
under  the canonical isomorphism $(L\otimes | L^{*} |)^{2} = U\times \mathbb{R}.$ 
 If 
$$
D_{u} ( e_{k}) = 2  \sum_{j< p} \omega_{pj} (u)  (e_{p}\wedge e_{j}) (e_{k}),\ \forall u\in\Gamma ( E\vert_{U}),
$$   
where $\omega_{pj}\in \Gamma ( E^{*}\vert_{U})$, then the   $E$-connection $D^{\Sigma}$ on $\Sigma$ 
defined by 
$$
D^{\Sigma}_{u}(\sigma_{\alpha }) := \frac{1}{2} \sum_{i<j} \omega_{ji} (u) e_{j} e_{i}\cdot \sigma_{\alpha},\ 
1\leq \alpha \leq r, 
$$
where $(\sigma_{\alpha})$ is a constant  frame of $\Sigma$, 
is compatible with $D$ (see e.g.\ \cite{cortes-david-MMJ}). 
Since $\mathrm{trace}(\, e_{i}e_{j}\cdot ) =0$, $D^{\Sigma}(\sigma_{1}\wedge \cdots \wedge \sigma_{r})=0$ 
and the $E$-connection induced by $D^{\Sigma}$ on $\Sigma \otimes |\mathrm{det}\,  \Sigma^{*} |^{1/r}$, also
denoted by $D^{\Sigma}$, 
satisfies
$$
D^{\Sigma}_{u}(\sigma_{\alpha }\otimes l_{\Sigma}) = \frac{1}{2} \sum_{i<j} \omega_{ji} (u) (e_{j} e_{i}\cdot \sigma_{\alpha})
\otimes l_{\Sigma},
$$
where $l_{\Sigma}:= | \sigma_{1}\wedge \cdots \wedge \sigma_{r}|^{-1/r}$.
Since  
$\langle\cdot , \cdot\rangle_{\Sigma\otimes | \mathrm{det}\,\Sigma^{*}|^{1/r}}$ is constant in the frame
$(\sigma_{\alpha}\otimes l_{\Sigma})$ and the Clifford action of  $e_{i}e_{j}$ is skew-symmetric with respect to 
$\langle\cdot , \cdot\rangle_{\Sigma\otimes | \mathrm{det}\Sigma^{*}|^{1/r}}$
(from the property (\ref{cliff-rel-spinors}) of $\langle\cdot , \cdot\rangle_{\Sigma\otimes |\mathrm{det}\,\Sigma^{*}|^{1/r}}$),
we obtain that $D^{\Sigma}$ preserves 
$\langle\cdot , \cdot\rangle_{\Sigma\otimes | \mathrm{det}\,\Sigma^{*}|^{1/r}}$.
Let  $D^{L}$ be an  $E$-connection on $L$.  Then $D^{\Sigma}\otimes D^{L}$ is an $E$-connection on $S\vert_{  U }$, 
compatible with $D$,  with the property that the induced connection on $\mathcal S\vert_{  U}$ preserves $\langle \cdot , \cdot \rangle_{\mathcal S}$ (easy check).
The latter coincides with $D^{\mathcal S}\vert_{  U}.$ 
\end{proof}

As a consequence of Proposition \ref{prop-pairing} we obtain,  for any $U\subset M$ open and sufficiently small,  a
canonical (unique modulo $\pm 1$)  $C^{\infty}(U)$-bilinear  
pairing 
\begin{equation}\label{ss-pairing}
\langle \cdot , \cdot \rangle_{\mathbb{S}\vert_{U} }:\Gamma ( \mathbb{S}\vert_{U}) \times\Gamma
( \mathbb{S}\vert_{U}) \rightarrow | \mathrm{det}\, T^{*}U|,\ 
\langle s\otimes l , \tilde{s}\otimes l\rangle_{\mathbb{S}\vert_{U}}: = \langle s, \tilde{s}\rangle_{\mathcal{S}\vert_{U} }
l^{2},
\end{equation}
where  $s, \tilde{s}\in\Gamma ( \mathcal S\vert_{U})$ and $l\in |  \mathrm{det}\, T^{*}U|^{1/2}.$
It satisfies
\begin{equation}\label{symmetry-pairing}
\langle u\cdot ( s\otimes l), u\cdot (\tilde{s}\otimes \tilde{l})\rangle_{\mathbb{S}\vert_{U}} =\langle u,u\rangle_E \langle s\otimes l, \tilde{s}\otimes \tilde{l} \rangle_{\mathbb{S}\vert_{U} },\
u\in \Gamma (E\vert_{U}) ,\ s\otimes l, \tilde{s}\otimes \tilde{l}\in \Gamma (\mathbb{S}\vert_{U}) .
\end{equation}
When $M$ is oriented, 
$\langle \cdot , \cdot\rangle_{\mathbb{S}\vert_{U}}$ 
takes  values in the bundle $\mathrm{det}\, T^{*}U$ of forms of top degree on $U$.
A pairing with similar properties (but with values in $(\mathrm{det}\, T^{*} U)\otimes \mathbb{C}$) was constructed in Proposition 3.14 of \cite{weyl-quant}. 

Given a standard Courant algebroid $T^{*}M\oplus \mathcal G\oplus TM$, we will often consider,
as in the next lemma, an irreducible $\mathrm{Cl}(\mathcal G)$-bundle $S_{\mathcal G}$  of $\mathcal G .$ 
This will always be assumed to be $\mathbb{Z}_{2}$-graded, with gradation compatible with the Clifford action.  In the next lemma by a canonical bilinear pairing of 
the weighted spinor bundle 
$\mathcal S_{\mathcal G}\vert_{U} $ of $S_{\mathcal G}\vert_{U}$ 
we mean a smooth $C^{\infty}(U)$-bilinear pairing 
\begin{equation}\label{pairing-lifted}
\langle \cdot , \cdot \rangle_{\mathcal S_{\mathcal G}\vert_{U}}: \Gamma (\mathcal S_{\mathcal G}\vert_{U})
\times  \Gamma (\mathcal S_{\mathcal G}\vert_{U})\rightarrow C^{\infty}(U)
\end{equation}
of normalized determinant, which satisfies 
\begin{equation}
\langle u\cdot s, u\cdot \tilde{s}\rangle_{\mathcal S_{\mathcal G}\vert_{U} }  = 
\langle u, u\rangle_{\mathcal G} \langle s, \tilde{s}\rangle_{\mathcal S_{\mathcal G}\vert_{U} } ,
\end{equation}
 for any $ s, \tilde{s}\in \Gamma ({\mathcal S}_{\mathcal G}\vert_{U} )$ and  $u\in \Gamma (\mathcal G\vert_{U}) $.
With the same argument as in Proposition \ref{prop-pairing}, such a pairing
exists when $U\subset M$ is a sufficiently small open set   
and  is unique up to  multiplication by $\pm 1$.  It 
 is preserved by the connection $\nabla^{\mathcal S_{\mathcal G}}$ induced by any  connection
$\nabla^{S_{\mathcal G}}$ on $S_{\mathcal G}$ compatible with $\nabla .$

\begin{lem}\label{cor-can-properties}   Assume that $E = T^{*} M\oplus \mathcal G \oplus TM$ is 
a standard  Courant algebroid over an oriented manifold $M$,  
defined by 
a quadratic Lie algebra bundle $(\mathcal G , [\cdot , \cdot ]_{\mathcal G}, \langle\cdot , \cdot\rangle_{\mathcal G})$
and data $(\nabla , R, H)$.  Let $S_{\mathcal G}$ be an irreducible   
$\mathrm{Cl}(\mathcal G)$-bundle, $\mathcal S_{\mathcal G}= S_{\mathcal G}\otimes | \mathrm{det } S^{*} |^{1/r}$ the canonical spinor bundle of $S_{\mathcal G}$ 
and $\mathbb{S} = \Lambda (T^{*}M)\hat{\otimes} \mathcal S_{\mathcal G}$
the corresponding canonical  weighted spinor bundle of $E$. 
For any  $U\subset M$ open and sufficiently  small, 
the  canonical bilinear pairing $\langle \cdot , \cdot \rangle_{\mathbb{S}\vert_{U} }$ is given 
(up to multiplication by $\pm 1$) 
by
\begin{equation}\label{can-pairing}
\langle \omega \otimes s, \tilde{\omega}\otimes \tilde{s}\rangle_{\mathbb{S}\vert_{U} } =
(-1)^{|s|  ( |\omega |+| \tilde{\omega } | )} ( \omega^{t} \wedge \tilde{\omega})_{\mathrm{top}} 
\langle s, \tilde{s}\rangle_{\mathcal S_{\mathcal G}\vert_{U} },
\end{equation}
where $\langle \cdot , \cdot \rangle_{\mathcal S_{\mathcal G}\vert_{U} }$ is the canonical bilinear pairing of ${\mathcal S}_{\mathcal G}\vert_{U}.$\
\end{lem}

\begin{proof} The claim is a consequence of the following general statement:
if $(V_{i}, \langle \cdot , \cdot \rangle_{i})$
are Euclidian vector spaces with metrics of neutral signature and $S_{i}$ are 
irreducible  $\mathrm{Cl}(V_{i})$-modules of ranks $r_{i}$, with canonical bilinear pairings $\langle\cdot , \cdot \rangle_{\mathcal S_{i}}$
on ${\mathcal S}_{i} = S_{i}\otimes | \mathrm{det}\, S^{*}_{i}|^{1/r_{i}}$, then
$S:=S_{1}\hat{\otimes} S_{2}$ is an irreducible  $\mathrm{Cl} (V_{1}\oplus V_{2})$-module, with
canonical   spinor module $\mathcal S = \mathcal S_{1}\hat{\otimes}\mathcal S_{2}$ and the
canonical bilinear pairing on $\mathcal S$ is  given by
\begin{equation}\label{tensors-pair}
\langle s_{1}\otimes s_{2}, \tilde{s}_{1}\otimes \tilde{s}_{2}\rangle_{\mathcal S} 
= (-1)^{| s_{2}| ( |s_{1} | + | \tilde{s}_{1}| )} \langle s_{1},
\tilde{s}_{1}\rangle_{\mathcal S_{1}} \langle s_{2}, \tilde{s}_{2}\rangle_{\mathcal S_{2}}.
\end{equation}
Indeed,  the scalar product (\ref{tensors-pair}) satisfies  (\ref{cliff-rel-spinors}) (easy check). 
In order to show that it has normalized determinant, we remark that
\begin{equation}\label{details}
\mathrm{det}\, \langle \cdot , \cdot \rangle_{\mathcal S} = \mathrm{det}\, \langle \cdot , \cdot \rangle_{\mathcal S}^{\prime}
=( \mathrm{det} \, \langle\cdot , \cdot \rangle_{\mathcal S_{1}})^{r_{2}} ( \mathrm{det} \, \langle\cdot , \cdot \rangle_{\mathcal S_{2}})^{r_{1}}=1,  
\end{equation}
where 
$$
\langle s_{1}\otimes s_{2}, \tilde{s}_{1}\otimes \tilde{s}_{2}\rangle_{\mathcal S}^{\prime} := 
\langle s_{1}, \tilde{s}_{1}\rangle_{\mathcal S_{1}} \langle s_{2}, \tilde{s}_{2}\rangle_{\mathcal S_{2}}.
$$
The first relation in (\ref{details}) can be checked using that the scalar products $\langle \cdot , \cdot \rangle_{\mathcal S}$ and  
$\langle \cdot , \cdot \rangle_{\mathcal S}^{\prime}$ differ only by a sign (dependent on degrees) when restricted to tensor products of homogeneous elements. (Recall that the even and odd parts of $\mathcal S_{i}$ are orthogonal or isotropic with respect to $\langle \cdot , \cdot \rangle_{\mathcal S_{i}}$.)
\end{proof}

The next corollary will be used to show that the pushforward commutes with the canonical Dirac generating operators
(see Section \ref{sect-push}).

\begin{cor}\label{de-adaugat:cor}
In the setting of Lemma  \ref{cor-can-properties}, 
let $\nabla^{\mathcal S_{\mathcal G}}$ be the connection on $\mathcal S_{\mathcal G}$ induced by an arbitrary connection
$\nabla^{S_{\mathcal G}}$ on 
$S_{\mathcal G}$,
compatible with $\nabla $. Define ${\mathcal E}\in \mathrm{End}\,  \Gamma ( \mathbb{S}\vert_{U}  )$ by
\begin{equation}
{\mathcal E} (\omega \otimes s) := (d\omega )\otimes s + \sum_{i} (\alpha_{i}\wedge \omega ) \otimes \nabla^{\mathcal S_{\mathcal G}}_{X_{i}} s,\ \omega \in \Omega (U ),\
s\in \Gamma ({\mathcal S}_{\mathcal G}\vert_{U} )
\end{equation}
where  $(X_{i})$ is a local frame of $TU$, with dual frame $( \alpha_{i})$. 
Then,   for any
$U\subset M$ open and sufficiently small and 
products $\omega \otimes s, \tilde{\omega}\otimes \tilde{s}\in
\Gamma (\mathbb{S}\vert_{U})$ of homogeneous elements, 
\begin{eqnarray}\label{symm-part} &
\langle {\mathcal E} (\omega \otimes s), \tilde{\omega}\otimes \tilde{s}\rangle_{\mathbb{S}\vert_{U}} + \langle \omega \otimes s, {\mathcal E} (\tilde{\omega} \otimes \tilde{s}) \rangle_{\mathbb{S}\vert_{U} }=\nonumber\\ 
& (-1)^{ | s | ( | \omega | + | \tilde{\omega} | + 1) + | \omega | }d \left(  \langle s, \tilde{s}\rangle_{{\mathcal S}_{\mathcal G}\vert_{U} } (\omega^{t}\wedge \tilde{\omega })_{m-1}\right) . 
\end{eqnarray}
Here $m$ is the dimension of $M$ and $\omega_{m-1}$ denotes  the  degree $(m-1)$-component of a form $\omega \in \Omega ( U).$
\end{cor}

\begin{proof} We use the expression (\ref{can-pairing}) of the canonical bilinear pairing $\langle \cdot , \cdot \rangle_{\mathbb{S}\vert_{U} }$ of
$\mathbb{S}\vert_{U}  = \Lambda (T^*U) \hat{\otimes} {\mathcal S}_{\mathcal G}\vert_{U} .$  Since $d ( \omega^{t})  = (-1)^{ |\omega |} ( d\omega )^{t}$ 
we obtain
\begin{align}
\nonumber& \langle (d \omega )\otimes s , \tilde{\omega}\otimes \tilde{s}\rangle_{\mathbb{S}\vert_{U} } = \\
\label{rel1} &(-1)^{ | s | ( | \omega | + | \tilde{\omega} | + 1) + | \omega | }
\langle s, \tilde{s}\rangle_{\mathcal S_{\mathcal G}\vert_{U} } \left(  d ((\omega^{t}  \wedge \tilde{\omega})_{m-1})  + (-1)^{ | \omega | +1}
(\omega^{t}  \wedge d \tilde{\omega})_{\mathrm{top}}\right) . 
\end{align}
Similarly, since $ (\alpha_{i}\wedge \omega )^{t}  = \omega^{t} \wedge \alpha_{i}$ and using that $\nabla^{{\mathcal S}_{\mathcal G}}$ preserves
$\langle \cdot , \cdot \rangle_{\mathcal S_{\mathcal G}\vert_{U} }$ we obtain
\begin{align}
\nonumber& \langle  (\alpha_{i}\wedge \omega )\otimes \nabla_{X_{i}}^{\mathcal S_{\mathcal G}} s, \tilde{\omega}\otimes \tilde{s}\rangle_{\mathbb{S}\vert_{U} } =\\
\label{rel2}& 
(-1)^{ | s | ( | \omega | + | \tilde{\omega } |+ 1) } (\omega^{t}  \wedge \alpha_{i}\wedge \tilde{\omega })_{\mathrm{top}} \left( 
X_{i} \langle s, \tilde{s}\rangle_{\mathcal S_{\mathcal G}\vert_{U}} - \langle s, \nabla^{\mathcal S_{\mathcal G}\vert_{U}}_{X_{i}} \tilde{s} \rangle_{\mathcal S_{\mathcal G}\vert_{U}}\right) .
\end{align}
From  (\ref{rel1}) and  (\ref{rel2}) we obtain  (\ref{symm-part}).
\end{proof}

\section{Dirac generating operator and operations on spinors}

\subsection{Behaviour of canonical Dirac generating operators under isomorphisms}

The next lemma and proposition  are stated for transitive Courant algebroids but the same arguments  hold in
the larger setting of regular Courant algebroids.

\begin{lem}\label{iso-bdle} Let $I_{E} : E_{1} \rightarrow E_{2}$ be an isomorphism of transitive Courant algebroids over 
a manifold $M$ and $S_{i}$ irreducible  $\mathrm{Cl}(E_{i})$-bundles ($i=1,2$). 
Then,  for any $U\subset M$ open and sufficiently small,  there is a unique
(up to  multiplication by a smooth non-vanishing  function)   isomorphism $I_{S\vert_{U} } : S_{1}\vert_{U} 
 \rightarrow S_{2}\vert_{U}$  such that
\begin{equation}\label{i-s}
I_{S\vert_{U} }\circ \gamma_{u} = \gamma_{I_{E}(u)} \circ I_{S\vert_{U}},\ \forall u\in E_{1}\vert_{U}.
\end{equation}
The map $I_{S\vert_{U} }$ is homogeneous (i.e.\ even or odd).
If $\slashed{d}_{1} \in \mathrm{End}\, \Gamma (S_{1}\vert_{U} )$ is a Dirac generating operator  of 
$E_{1}\vert_{U}$ then 
$\slashed{d}_{2} := I_{S\vert_{U}}\circ \slashed{d}_{1}\circ I_{S\vert_{U}}^{-1}\in \mathrm{End}\,  \Gamma (S_{2}\vert_{U})$ is a Dirac generating operator of $E_{2}\vert_{U}.$
\end{lem}

\begin{proof} 
Assume that    
$E_{1}\vert_{U}$ admits an orthonormal frame  $(e_{i})$ and let   $(\tilde{e}_{i}):=( I_{E}(e_{i}))$ be 
the corresponding orthonormal frame of $E_{2}\vert_{U}.$ Like in the proof of  Lemma \ref{pas2}, 
$S_{i}\vert_{U} = \Sigma_{i}\otimes L_{i}$
where $\Sigma_{i}:= U\times V$ are $\mathrm{Cl} (E_{i})$-bundles, constructed using an irreducible      
$\mathrm{Cl} (\mathbb{R}^{2n}, \langle \cdot , \cdot \rangle_{\mathbb{R}^{2n}})$-module $V$ and  the orthonormal frames
$(e_{i})$ and $(\tilde{e}_{i})$ respectively,  and $L_{i}$ are line bundles over $U$.  
Restricting $U$ if necessary, we may assume that $L_{i}$ are isomorphic.  Let $I_{L}: L_{1}\rightarrow L_{2}$
be an isomorphism.
Then $I_{S\vert_{U} }: S_{1}\vert_{U}\rightarrow S_ {2}\vert_{U}$ defined by
$I_{S\vert_{U} }( \sigma \otimes l) := \sigma \otimes I_{L}(l)$  satisfies  (\ref{i-s}).
The even and odd parts of $S_{1}$  are given by $S_{1}^{0} = \frac{1}{2} ( 1 +\epsilon \gamma_{\omega})S$
and $S_{1}^{1} =\frac{1}{2} ( 1 -\epsilon \gamma_{\omega}) S$, where $\epsilon \in \{ \pm 1\}$ and 
 $\omega = e_{1}\cdots e_{2n}$,  and similarly for the even and odd parts of $S_{2}$ (using $\tilde{\omega}=
\tilde{e}_{1} \cdots  \tilde{e}_{n}$). 
Therefore the statement
that $I_{S\vert_{U} }$ is homogeneous follows from (\ref{i-s}), which implies that $I_{S\vert_{U} }\circ \gamma_{\omega}
= \gamma_{\tilde{\omega } }\circ I_{S\vert_{U} }$.   
Since  $I_{S\vert_{U} }$ is homogeneous  and  $\slashed{d}_{2}$ is odd,  we obtain that also $\slashed{d}_{1}$
is odd.   
The statement that $\slashed{d}_{2}$ satisfies the remaing conditions  from the definition of a Dirac generating operator 
can be checked using  (\ref{i-s}), which 
implies 
\[
[ \slashed{d}_{2}, \gamma_{I_{E}(u)}] = I_{S\vert_{U} }\circ [ \slashed{d}_{1}, \gamma_{u} ] \circ I_{S\vert_{U} }^{-1}. \qedhere
\]
\end{proof}

\begin{rem}\label{global-remark}{\rm 
i) In general, the isomorphisms $I_{S\vert_{U}}$ do not glue together to give an isomorphism
$I_{S} : S_{1}\rightarrow S_{2}$ compatible with $I_{E}$. However, 
assume  that 
$E_{1} = E_{2} = E$ and let
$S_{1} = S_{2} = S$  be an irreducible spinor bundle over $\mathrm{Cl}(E).$ 
If  $I_{E}\in  \mathrm{Aut} (E)$ is of the form 
$I_{E}(u) =\alpha \cdot u\cdot \alpha^{-1}$, where $\alpha \in \Gamma ( \mathrm{Pin}(E))$,
then $I_{S} \in \mathrm{Aut}(S)$ defined by  
$I_{S} (s) := \alpha \cdot s$, $s\in \Gamma (S)$,  satisfies (\ref{i-s}).\

ii) For example, if $E= T^{*}M \oplus \mathcal G \oplus TM$ is in the standard form
and $\beta \in \Omega^{2}(M)$, then 
$$
I_{E} (\xi + r + X)= \xi + r + X +i_{X}\beta
$$
can be written as $I_{E}(u) = \alpha \cdot u\cdot \alpha^{-1}$ for $\alpha := e^{-\beta}$ and the induced action on 
the spinor bundle $S:= \Lambda (T^{*}M)\hat{\otimes} S_{\mathcal G}$ 
is given by  $I_{S} (\omega \otimes r) := (e^{-\beta }\wedge \omega )\otimes r$ and  is globally defined. 
If, moreover, $d\beta =0$ then $I_{E}$ is a Courant algebroid automorphism (see relations  (\ref{iso-chen})).
Similarly, any automorphism $K\in \mathrm{Aut}(\mathcal G)$ of the bundle of quadratic Lie algebras
$\mathcal G$,  which belongs to the connected component $\mathrm{Aut}(\mathcal G , \langle \cdot , \cdot \rangle_{\mathcal G})_{0}$ and is parallel with respect to the connection $\nabla$ from the data which defines $E$,  defines a Courant algebroid automorphism of $E$ whose action on spinors is globally defined. In the case of exact odd Courant algebroids \cite{rubio} one has an additional class of such
automorphisms, defined by $1$-forms (and called in \cite{rubio} $A$-fields). However,  the class of exact odd Courant algebroids is 
a very special class of Courant algebroids, with $\mathcal G$ the trivial  rank one vector bundle and 
$\nabla$ the standard flat connection. 
  }
\end{rem}

The isomorphism $I_{S\vert_{U}} : S_{1}\vert_{U} \rightarrow S_{2}\vert_{U}$  from Lemma \ref{iso-bdle} 
induces an   isomorphism 
$I_{\mathcal S\vert_{U} } : \mathcal S_{1}\vert_{U} \rightarrow \mathcal S_{2}\vert_{U}$  
between the canonical  spinor bundles of $S_{1}$ and $S_2$, given by
\begin{equation}
I_{\mathcal S} (s\otimes | s_{1}\wedge \cdots \wedge s_{r}|^{-1/r}):=
(I_{S}s) \otimes  | I_{S}s_{1}\wedge \cdots \wedge I_{S} s_{r}|^{-1/r},
\end{equation} 
where $s_{1}\wedge \cdots \wedge s_{r}\in\Gamma ( \Lambda^{r}(S_{1}\vert_{U})) $ is non-vanishing. 
Since $I_{S\vert_{U}}$ is unique up to a multiplicative factor, 
$I_{\mathcal S\vert_{U}}$  is independent of the choice of $I_{S\vert_{U}}$, modulo multiplication by
$\pm 1$   (see  also Remark 55 of  \cite{cortes-david-MMJ}).

\begin{lem}\label{pairing-iso} For any $U\subset M$ open and sufficiently small, the isomorphism $I_{\mathcal S\vert_{U}}$ preserves   the canonical bilinear pairings 
$\langle \cdot , \cdot \rangle_{\mathcal S_{i}\vert_{U} }$ of $\mathcal S_{i}\vert_{U}$, i.e.\ 
$\langle I_{\mathcal S} s,I_{\mathcal S} \tilde{s} \rangle_{\mathcal S_{2}\vert_{U} } = \epsilon \langle s, \tilde{s}\rangle_{\mathcal S_{1}\vert_{U}}$,
for all $s, \tilde{s}\in \Gamma ({\mathcal S}_{1}\vert_{U})$, where $\epsilon \in \{\pm 1\}$  is independent of $s, \tilde{s}$. 
\end{lem}

\begin{proof}From relation (\ref{i-s}) and 
the fact that $I_{E}$ is an isometry, we obtain 
 that bilinear pairing
$\langle s, \tilde{s}\rangle^{\prime}_{\mathcal S_{1}\vert_{U} } := \langle I_{\mathcal S}(s), I_{\mathcal S} (\tilde{s} )\rangle_{\mathcal S_{2}\vert_{U} }$
on $\mathcal S_{1}$  satisfies (\ref{cliff-rel-spinors}). 
Also, $\mathrm{det}\, \langle\cdot , \cdot\rangle_{\mathcal S_{1}\vert_{U} }^{\prime} =\mathrm{det}\, \langle\cdot , \cdot \rangle_{\mathcal S_{2}\vert_{U}} =1$, if $\mathrm{rk}\, E_1 >2$ (and $=-1$ if $\mathrm{rk}\, E_1 =2$). 
\end{proof}

From the above lemma,  
$I_{\mathbb{S}\vert_{U}} := I_{\mathcal S\vert_{U} }\otimes \mathrm{Id}_{ |\mathrm{det}\, T^{*}U|^{1/2}}
:\mathbb{S}_{1}\vert_{U} \rightarrow \mathbb{S}_{2}\vert_{U}$ satisfies
\begin{equation}\label{pairing-iso-relation}
\langle I_{\mathbb{S}\vert_{U} } (s), I_{\mathbb{S}\vert_{U} } (\tilde{s})\rangle_{\mathbb{S}_{2}\vert_{U} } =
\epsilon \langle s, \tilde{s}\rangle_{\mathbb{S}_{1}\vert_{U} },\ s, \tilde{s}\in \Gamma (\mathbb{S}_{1}\vert_{U} )
\end{equation} 
where 
$\langle\cdot , \cdot\rangle_{\mathbb{S}_{i}\vert_{U} }$ are the  canonical
$|\mathrm{det}\, T^{*}U|$-valued bilinear pairings of 
the canonical weighted spinor bundles $\mathbb{S}_{i}\vert_{U} $ of $E_{i}\vert_{U}$
determined by $S_{i}\vert_{U}$ and  $\epsilon \in \{ \pm 1\}$ is independent on $s$ and $\tilde{s}.$

\begin{notation}{\rm 
The isomorphisms $I_{\mathcal S\vert_{U}}$ and  $I_{\mathbb{S}\vert_{U}}$
are   determined only up to multiplication by $\pm 1.$ 
In our computations we will often choose (without repeating it each time) 
one  $I_{\mathcal S\vert_{U}}$,  $I_{S\vert_{U}}$ or
$I_{\mathbb{S}\vert_{U}}$,  and refer to it as {\cmssl the}  isomorphism induced by $I$  
(or {\cmssl the} isomorphism compatible with $I$)  
on the spinor bundle, canonical spinor bundle and canonical weighted spinor bundle on $U$, respectively.  The arguments will be independent on this choice. A similar convention will be used for the 
various canonical bilinear pairings like  $\langle \cdot , \cdot \rangle_{\mathbb{S}\vert_{U}}$  or $\langle \cdot , \cdot \rangle_{\mathcal S_{\mathcal G}\vert_{U}}$ and for  the pullback and pushforward
on spinors (which will also be uniquely defined only  up to multiplication by $\pm 1$, see the next sections)}. 
\end{notation}

\begin{rem}\label{computation-iso} {\rm 
In the setting of Lemma \ref{iso-bdle}, assume that
$E_{i} = T^{*}M \oplus \mathcal G_{i} \oplus TM$  ($i=1,2$) are standard Courant algebroids and that $I_{S} : S_{1} 
\rightarrow S_{2}$ is defined globally.  
Let  $S_{\mathcal G_{i}}$ be   irreducible  $\mathrm{Cl} (\mathcal G_{i})$-bundles of rank $r$
and $S_{i}:= \Lambda (T^{*}M)\hat{\otimes} S_{\mathcal{G}_i}$. 
Using (\ref{S-dissection}) and (\ref{iso-standards}), one can show that the  isomorphism 
$$
I_{\mathbb{S}} : {\mathbb S}_{1} = \Lambda\, (T^{*}M)
\hat{\otimes} {\mathcal S}_{\mathcal G_{1}}\rightarrow\mathbb{S}_{2} =   \Lambda\, (T^{*}M)
\hat{\otimes} {\mathcal S}_{\mathcal G_{2}}
$$
induced by  $I: E_{1} \rightarrow E_{2}$ 
on the canonical weighted spinor bundles  determined by $S_{i}$  is given in terms of $I_{S} : S_{1} \rightarrow S_{2}$ by
\begin{equation}\label{computation-induced-iso}
I_{\mathbb{S}} ( \omega\otimes s \otimes  |s_{1}^{*}\wedge \cdots \wedge s_{r}^{*}|^{\frac{1}{r}} )
=| \mathrm{det}\, (I_{S})|^{-\frac{1}{Nr}}\, I_{S} (\omega \otimes s) \otimes   | \tilde{s}_{1}^{*}\wedge \cdots \wedge \tilde{s}_{r}^{*}|^{\frac{1}{r}},
\end{equation}
where   $\omega \otimes s\in \Gamma (S_{1})$,  
$( s_{i}^{*})$
and $(\tilde{s}^{*}_{i}) $ are  local frames  of $S^{*}_{\mathcal G_{1}}$ and $S^{*}_{\mathcal G_{2}}$ respectively, $N:= \mathrm{rk}\, \Lambda (TM)$, $r:= \mathrm{rk}\, S_{\mathcal G_{i}}$,  
and  $\mathrm{det}\,  (I_{S})$ is the determinant of the representation matrix 
of $I_{S}$ in the local frames $( \Omega_{i}\otimes s_{j})$ and 
$( \Omega_{i}\otimes\tilde{s}_{j})$ respectively, where  
$( \Omega_{i})$ is the  local frame  of $\Lambda (T^{*}M)$ induced by a  local frame
of $TM$ 
and $(s_{i})$, $( \tilde{s}_{i})$ are the frames  dual to $( s_{i}^{*})$ and $( \tilde{s}_{i}^{*})$ respectively.
(We shall refer to $\mathrm{det}\,  (I_{S})$  as {\cmssl the determinant of $I_{S}$ with respect to
the local frames  $( s_{i})$ and $(\tilde{s}_{i})$}).  } 
\end{rem}

\begin{prop}\label{iso-dirac} In the setting of Lemma   \ref{iso-bdle}, if $\slashed{d}_{1}$ is the canonical Dirac generating operator of $E_{1}\vert_{U}$ then 
\begin{equation}\label{dirac-conj}
\slashed{d}_{2} = I_{\mathbb{S}\vert_{U}} \circ \slashed{d}_{1} \circ I_{\mathbb{S}\vert_{U} }^{-1}
\end{equation} 
is the canonical Dirac generating operator of $E_{2}\vert_{U}.$
\end{prop}

\begin{proof}
Let $\nabla^{(1)}$ be a metric connection on
$E_{1}\vert_{U}$ and  $\nabla^{S_{1}}$ a compatible connection on $S_{1}\vert_{U}$. Let  $D^{(1)}$ 
and $D^{S_{1}}$ be
the generalized connection on $E_{1}\vert_{U}$ and the $E_{1}\vert_{U}$-connection on $S_{1}\vert_{U}$
defined by $\nabla^{(1)}$ and $\nabla^{S_{1}}$ respectively. 
Let $\nabla^{(2)}:= I_{E}\circ \nabla^{(1)}\circ I_{E}^{-1}$ and $\nabla^{S_{2}}:= I_{S\vert_{U} }\circ \nabla^{S_{1}}
\circ( I_{S\vert_{U} })^{-1}$. Then $\nabla^{(2)}$ is a metric connection on $E_{2}\vert_{U}$ and $\nabla^{S_{2}}$
is compatible with $\nabla^{(2)}.$ Let $D^{(2)}$ and $D^{S_{2}}$ be the generalized connection on $E_{2}\vert_{U}$
and the (compatible) $E_{2}\vert_{U}$-connection on $S_{2}\vert_{U}$,
defined by $\nabla^{(2)}$ and $\nabla^{S_{2}}.$ 
As formula (\ref{dirac-reg}) for  
the canonical Dirac generating operator is independent of the choice of generalized connection 
(and compatible $E$-connection),  we can (and will) choose to  compute 
$\slashed{d}_{1}$ and $\slashed{d}_{2}$ using $(D^{(1)}, D^{S_{1}})$ and $( D^{(2)}, D^{S_{2}})$ respectively.

A straightforward computation using (\ref{rule:eq}) shows that
\begin{equation}
 ( D^{(2)})^{L}_{I_{E}(u)} \mu = (D^{(1)})^{L}_{u}\mu,\ \forall u\in E_{1}\vert_{U},\ \mu\in\Gamma (L\vert_{U})
\end{equation}
and 
\begin{equation}\label{rel-dirac}
 (D^{\mathcal S_{2}}\otimes (D^{(2)})^{L} )_{u}  =
I_{\mathbb{S}\vert_{U}} \circ ( D^{\mathcal S_{1}} \otimes (D^{(1)})^{L})_{I^{-1}_{E}(u)}\circ  
(I_{\mathbb{S}\vert_{U}})^{-1}.
\end{equation}
Relation (\ref{rel-dirac}) implies 
that the Dirac  operators $\slashed{D}^{(2)}$ 
on $\mathbb{S}_{2}\vert_{U}$ and $\slashed{D}^{(1)}$ 
on $\mathbb{S}_{1}\vert_{U}$ computed with 
$D^{\mathcal S_{2}}\otimes (D^{(2)})^{L}$ and $D^{\mathcal S_{1}}\otimes ( D^{(1)})^{L}$ respectively, are related by 
\begin{equation}\label{conj-D}
\slashed{D}^{(2)} = I_{\mathbb{S}\vert_{U}} \circ \slashed{D}^{(1)} \circ (I_{\mathbb{S}\vert_{U}})^{-1}.
\end{equation}
On the other hand, it is easy to see that
$$
T^{D^{(2)}} (u, v, w) = T^{D^{(1)}} ( I_{E}^{-1} u, I_{E}^{-1} v , I_{E}^{-1} w),\  \forall u, v, w\in E_{2}\vert_{U}
$$
which implies that 
\begin{equation}\label{t-d-2}
T^{D^{(2)}} = I_{E } (T^{D^{(1)}})
\end{equation}
where 
$T^{D^{(i)}} \in \Gamma (\Lambda^{3} E_{i}\vert_{U} ) \subset \Gamma\,  \mathrm{Cl} (E_{i}\vert_{U})$
and  $I_{E}$ denotes the 
action induced by  the isometry $I_{E}$ on Clifford
algebras.
Relations (\ref{i-s}) and (\ref{t-d-2})  
imply that $\gamma_{T^{ D^{(2)}}} = I_{\mathbb{S}\vert_{U}} \circ \gamma_{T^{ D^{(1)}}} \circ
(I_{\mathbb{S}\vert_{U}})^{-1}$ which, together with (\ref{dirac-reg} and 
(\ref{conj-D}), implies our claim.
\end{proof}

\subsection{Pullback of spinors}\label{pull-back-subsection}

Let $f: M \rightarrow N$ be a submersion and $E$ a transitive Courant algebroid over $N$. 
Following \cite{li-bland}, we  recall  the definition of  the pullback Courant algebroid $f^{!}E$.  
Let $\mathbb{T}M := T^{*}M\oplus TM$ be the generalized tangent bundle with its standard Courant algebroid structure, given by the  Dorfmann
bracket
\begin{equation}
[\xi + X, \eta + Y]:=  {\mathcal L}_{X}(Y+\eta )- i_{Y} d\xi , 
\end{equation}
for any $X, Y\in {\mathfrak X}(M)$, $\xi , \eta\in \Omega^{1}(M)$, 
scalar product  $\langle \xi + X, \eta + Y \rangle := \frac{1}{2} ( \xi (Y) +\eta (X))$
and anchor the natural projection from $\mathbb{T}M$ to $TM.$
Consider the direct product Courant algebroid
$E\times \mathbb{T}M$ and let   $a: E\times \mathbb{T}M\rightarrow T(N\times M)$  be its anchor. 
Then $C:= a^{-1} (TM_f)$ is a coisotropic subbundle of $E\times \mathbb{T}M$ over the 
graph $M_f\subset N\times M$ of $f$, which we identify with $M$.
Its fiber over $p\in M$ is given by
\begin{equation}
C_{p}:= \{ (u, \eta + X )\in  E_{f(p)}\times\mathbb{T}_{p}M  \mid  \pi (u) = (d_{p}f)(X) \}
\end{equation}
and 
\begin{equation}
C^{\perp}_{p}:= \{ (\frac{1}{2}\pi^{*} \gamma ,  - (d_{p}f)^{*}\gamma ) \mid  \gamma \in T^{*}_{f(p)} N\} \subset C_{p},
\end{equation}
where $\pi^{*} : T^{*}N\rightarrow E$ is the dual  of the anchor $\pi : E \rightarrow TN$ composed with the
natural identification $E^*\stackrel{\sim}{\rightarrow} E$ induced  by the scalar  product $\langle \cdot , \cdot \rangle$ of $E$. 
The quotient  $C/ C^{\perp}$ is a Courant algebroid  over $M  \cong M_f$ with anchor,
scalar product and Courant bracket induced from $E\times \mathbb{T}M.$   The Courant algebroid
$C/ C^{\perp}$ was  called in \cite{li-bland} the {\cmssl pullback of $E$ by the map $f$}.

\begin{lem} \label{pull-back-bdle} 
i)  Let $E = T^{*}N \oplus \mathcal G \oplus TN$ be a standard Courant algebroid,
defined by a bundle of
quadratic Lie algebras $(\mathcal G , [\cdot , \cdot ]_{\mathcal G}, \langle \cdot , \cdot \rangle_{\mathcal G} )$ and  data $(\nabla , R, H).$ Then $f^{!}E$ is isomorphic to the standard Courant algebroid defined by the bundle of quadratic Lie algebras 
$$
(f^{*}\mathcal G ,\   [\cdot , \cdot ]_{f^{*}\mathcal G} :=  f^{*}[\cdot , \cdot ]_{\mathcal G}, \
\langle \cdot , \cdot \rangle_{f^{*} \mathcal G} :=  f^{*}\langle \cdot , \cdot 
\rangle_{\mathcal G} )
$$ 
together with $(f^{*}\nabla , f^{*} R, f^{*}H )$.\

 ii) Let $I:E_{1} \rightarrow E_{2}$ be an isomorphism between two transitive Courant algebroids
over $N$ and $a_{i}:   E_{i}\times \mathbb{T}M \rightarrow T(N\times M)$ the anchors of the direct product
Courant algebroids $E_{i}\times \mathbb{T}M$ ($i=1,2$).  Then $I$ induces a Courant algebroid  isomorphism $I^{f} : f^{!}E_{1} \rightarrow f^{!}E_{2}$
defined by
\begin{equation}\label{I^f:eq}
I^{f} [ (u,  \eta + X )]:= [( I(u),   \eta + X ) ],\ \forall (u, \eta + X) \in (C_1)_p,
\end{equation} 
where $C_i=(a_i)^{-1}(TM_f)$ and  $[( I(u),  \eta + X )]$ denotes the class of $(I(u),  \eta + X )\in  (C_2)_{p}$ modulo $(C_2)_{p}^{\perp}.$\

iii) Let $E$ be a transitive Courant algebroid over $N$. Any dissection of $E$ induces a dissection of $f^{!}E.$\
Moreover,    if  $I_{i} : E \rightarrow T^{*}N \oplus \mathcal G_{i} \oplus TN$ are two dissections of $E$, 
related by  $(\beta, K, \Phi )$,  then  the induced dissections 
of  $ f^{!} E $ are related by  $(f^{*}\beta,f^{*} K, f^{*}\Phi )$.
\end{lem}

\begin{proof}
i) We  claim that the quadratic Lie algebra bundle  $(f^{*}\mathcal G , [\cdot , \cdot ]_{f^{*}\mathcal G}, \langle\cdot , \cdot\rangle_{f^{*}\mathcal G})$
together with    $(f^{*}\nabla , f^{*}R, f^{*} H)$ define  a standard Courant algebroid. 
The proof reduces to the verification of the conditions stated in Section \ref{trans-sect-basic}. 
 The form $f^{*} R$ is defined  by 
$(f^{*}R) (X, Y) = R(dfX, dfY)\in \mathcal{G}_{f(p)}=(f^*\mathcal{G})_p$, for any $X, Y\in T_pM$, $p\in M$.  
To show, for instance, that 
\begin{equation} \label{dnablaRXYZ:eq} [d^{f^{*}\nabla} (f^{*} R)](X,Y,Z)=0\quad\mbox{for all}\quad X,Y,Z\in \mathfrak{X}(M),\end{equation} 
cf.\ equation (\ref{dnablaR:eq}), we notice that it holds for any projectable vector fields $X, Y , Z\in {\mathfrak X}(M)$,
since
\begin{align*}
\nonumber& (f^{*}\nabla ) _{X} [ ( f^{*}R )(Y, Z) ] = f^{*} [ \nabla_{f_{*}X} R(f_{*}Y, f_{*}Z) ],\\
\nonumber& (f^{*}R) ({\mathcal L}_{X} Y, Z) = f^{*} [ R ({\mathcal L}_{f_{*}X} f_{*}Y, f_{*}Z)]
\end{align*}
and that it is $C^{\infty}(M)$-linear in all arguments $X, Y, Z$. Here $f_{*}X$ denotes the vector field 
on $N$ obtained by projection of a projectable vector field $X\in \mathfrak{X}(M)$. Recall that 
for projectable vector fields we have $dfX = (f_*X) \circ f$. 
Relation (\ref{dnablaRXYZ:eq}) follows. In a similar way we prove
that $(f^{*}\mathcal G , [\cdot , \cdot ]_{f^{*}\mathcal G}, \langle\cdot , \cdot\rangle_{f^{*}\mathcal G})$
together with    $(f^{*}\nabla , f^{*}R, f^{*} H)$
satisfy   the remaining conditions for standard Courant algebroids. 

One  can show that the 
map
\begin{equation}\label{iso-pull-exact}
F :  T^{*}M\oplus f^*\mathcal G \oplus TM \rightarrow  f^{!}E,\  F( \eta + r + X) := [ (r + df \, ( X), \eta + X ) ] 
\end{equation}
where  $\eta\in T_{p}^*M$, $r\in \mathcal G_{f(p)}$, $X\in T_{p} M$ and  $p\in M$ is arbitrary, 
is a Courant algebroid isomorphism between  the standard Courant algebroid defined by  the quadratic Lie algebra bundle 
$(f^{*}\mathcal G , [\cdot , \cdot ]_{f^{*}\mathcal G}, \langle\cdot , \cdot\rangle_{\mathcal G})$
together with   $(f^{*}\nabla , f^{*}R, f^{*} H)$,   and
$f^{!}E$.\

ii), iii)    Claim ii) is an easy check and  claim iii)  follows by combining claims ii) and iii) and using
(\ref{concrete-iso}). 
\end{proof}

Our next  aim is to define a pullback from spinors of $E$ to spinors of $f^{!}E.$ 
At first, we assume that $E$ is a standard Courant algebroid.

\begin{rem}\label{pull-back-standard}{\rm   i) Assume that $E= T^{*} N \oplus \mathcal G \oplus TN$ is a standard Courant algebroid,
defined by a bundle of quadratic
Lie algebras $(\mathcal G , [\cdot , \cdot ]_{\mathcal G},\langle \cdot , \cdot \rangle_{\mathcal G})$
and data $(\nabla , R, H)$. 
Using the isomorphism
(\ref{iso-pull-exact}), 
we often  identify (without repeating  it  each time)   $f^{!}E$ with the standard Courant algebroid $T^{*}M\oplus f^{*} \mathcal G \oplus TM$ 
defined by the quadratic Lie algebra bundle  $(f^{*}\mathcal G , f^{*}[\cdot , \cdot ]_{\mathcal G},f^{*}\langle \cdot , \cdot \rangle_{\mathcal G})$
and  data $(f^{*}\nabla ,f^{*}R,f^{*} H)$.
We fix
an irreducible  $\mathrm{Cl}(\mathcal G)$-bundle $S_{\mathcal G}$. Then $S_{f^{*}\mathcal G}:= f^{*}S_{\mathcal G}$ is an 
irreducible $\mathrm{Cl}( f^{*} \mathcal G )$-bundle 
and $f^{*}\mathcal S_{\mathcal G}  = {\mathcal S}_{f^{*} \mathcal G}$. 
The natural map 
\begin{align}
\nonumber f^{*} : \Gamma (\mathbb{S}_{N})= &  \Omega (N,   \mathcal S_{\mathcal G})
\rightarrow 
\Gamma (\mathbb{S}_{M}) = \Omega (M,  f^{*} \mathcal S_{\mathcal G} )  ,\\
\label{pull-back-forms}&  \omega \otimes s\rightarrow   f^{*}(\omega ) \otimes f^{*}(s)
\end{align} 
 preserves the $\mathbb{Z}_{2}$-degrees of  $\mathbb{S}_{N}$ and
$\mathbb{S}_{M}$.  It is called the {\cmssl pullback on spinors}.\

ii)  Assume in addition that $f:M \rightarrow N$ is endowed with a horizontal distribution.  
For any $X\in {\mathfrak X}(N)$, we denote by $\widehat{X}\in \mathfrak{X}(M)$ the horizontal lift of $X$
and  we  define a map
\begin{equation} \label{fstar:eq}
f^{*} : \Gamma (E) \rightarrow \Gamma (f^{!} E),\ f^{*} (\xi + r + X) := f^{*}(\xi + r) + \widehat{X}.
\end{equation}
Let  $\langle \cdot , \cdot\rangle_{E}$ and 
$\langle \cdot , \cdot \rangle_{f^{!}E}$ be  the scalar products  of $E$ and $f^{!}E$. 
As 
$$
\langle f^{*}(u), f^{*}( v) \rangle_{f^{!} E} = \langle u, v\rangle_{E}\circ f,\ u, v\in \Gamma (E),
$$
we obtain an induced map $f^{*} :\Gamma\, \mathrm{Cl}(E) \rightarrow \Gamma\, \mathrm{Cl} (f^{!} E)$, which satisfies
\begin{equation}\label{pull-back-clifford}
f^{*}(u\cdot v) = f^{*}(u) \cdot f^{*}(v),\ \forall u, v\in \Gamma\, \mathrm{Cl} (E)
\end{equation}
and 
\begin{equation}\label{pull-compat-cliff}
f^{*}(u\cdot s) = f^{*}(u)\cdot f^{*}(s),\ u\in \Gamma \mathrm{Cl}(E),\ s\in \Gamma (\mathbb{S}_{N}).
\end{equation}}
\end{rem}

Assume now that $E$ is a transitive, but not necessarily standard,  Courant algebroid 
and let $\mathbb{S}_{E}$ and $\mathbb{S}_{f^{!}E}$ 
be canonical weighted spinor bundles of $E$ and $f^{!}E$, 
determined by irreducible spinor bundles $S_{E}$ and $S_{f^{!}E}$ respectively. 
In order to be able  to construct a pullback map from  $\Gamma (\mathbb{S}_{E})$ to  
$\Gamma (\mathbb{S}_{f^{!}E})$,  
we assume that
the following condition is satisfied:
there is a dissection
$I:E\rightarrow  E_{N}= T^{*}N \oplus \mathcal G \oplus TN$ of $E$ 
and an   irreducible  
$\mathrm{Cl}(\mathcal G )$-bundle
$S_{\mathcal G}$, such that $I$ and 
the dissection  $I^{f}:f^{!}E\rightarrow  E_{M}= T^{*}M\oplus f^{*}\mathcal G \oplus TM$  of $f^{!}E$  induce {\cmssl  global} isomorphisms $ I_{S}: S_{E}\rightarrow \Lambda (T^{*} N)\hat{\otimes} S_{\mathcal G}$
and $I^{f}_{S}: S_{f^{!}E}\rightarrow \Lambda (T^{*} M)\hat{\otimes} f^{*} S_{\mathcal G}$ between spinor bundles.  
We shall often refer to $(I, S_{\mathcal G })$ as an {\cmssl admissible pair} for $\mathbb{S}_{E}$
and $\mathbb{S}_{f^{!}E}.$  
Let 
\begin{align}
\nonumber& I_{\mathbb{S}} :\mathbb{S}_{E}\rightarrow \mathbb{S}_{N}= \Lambda (T^{*}N)\hat{\otimes} \mathcal S_{\mathcal G}\\
 \label{maps-spinors}& I^{f}_{\mathbb{S}} :\mathbb{S}_{f^{!}E}\rightarrow \mathbb{S}_{M}= \Lambda (T^{*}M)\hat{\otimes}
f^{*} \mathcal S_{\mathcal G}
\end{align}
be the induced (global)  isomorphisms between the canonical weighted spinor bundles determined by $S_{E}$, 
$\Lambda ( T^{*}N)\hat{\otimes} S_{\mathcal G}$, $S_{f^{!}E}$ and $\Lambda ( T^{*}M)\hat{\otimes}f^{*} S_{\mathcal G}.$

\begin{lem}\label{well-def-lem} The map    
\begin{equation}\label{pull-back-relation}
f^{!}:\Gamma ( \mathbb{S}_{E}) \rightarrow \Gamma (\mathbb{S}_{f^{!}E}),\ f^{!}:= ( I^{f}_{\mathbb{S}})^{-1} \circ  f^{*} \circ  I_{\mathbb{S}}
\end{equation}
is well defined, up to multiplication by $\pm 1.$ It is called the {\cmssl pullback on spinors}.
\end{lem}

\begin{proof} Let 
\begin{equation}\label{I}
I:E_{1}=  T^{*}N \oplus \mathcal G_{1} \oplus TN \rightarrow E_{2} = T^{*}N\oplus \mathcal G_{2} \oplus TN
\end{equation}
be  an isomorphism between standard Courant algebroids and 
$$
I^{f}:  f^{!}E_{1}= T^{*}M \oplus f^{*} \mathcal G_{1} \oplus TM \rightarrow
f^{!}E_{2}=T^{*}M \oplus f^{*}\mathcal G_{2} \oplus T^{*}M
$$
the induced isomorphism between their  pullbacks. Let $S_{\mathcal G_{i}}$  be   irreducible 
$\mathrm{Cl}(\mathcal G_{i})$-bundles, such that 
$I$ and $I^{f}$ induce global isomorphisms 
$$
I_{S} : \Lambda (TN)\hat{\otimes} S_{{\mathcal G}_{1}}
\rightarrow  \Lambda (TN)\hat{\otimes} S_{{\mathcal G}_{2}},\ 
I^{f}_{S} : \Lambda (TM)\hat{\otimes} f^{*}S_{{\mathcal G}_{1}}
\rightarrow  \Lambda (TM)\hat{\otimes}f^{*} S_{{\mathcal G}_{2}}.
$$
By considering two admissible pairs for $\mathbb{S}_{E}$ and $\mathbb{S}_{f^{!}E}$  
the claim reduces to showing that
\begin{equation}\label{required-pullback}
I^{f}_{ \mathbb{S_{M}}} \circ f^{*} = \epsilon  f^{*} \circ I_{\mathbb{S}_{N}}
\end{equation}
where  $\epsilon \in \{ \pm 1\}$,     
\begin{equation}\label{canonical-weighted-spinors}
I_{\mathbb{S}_{N}}:   \mathbb{S}^{1}_{N}\rightarrow \mathbb{S}^{2}_{N},\ I^{f}_{ \mathbb{S}_{M}}: \mathbb{S}^{1}_{M}\rightarrow \mathbb{S}^{2}_{M}
\end{equation}
are the  isomorphisms  induced by $I$ and $I^{f}$ 
on  the canonical weighted spinor bundles 
$$
\mathbb{S}_{N}^{i}= \Lambda (T^{*}N)\hat{\otimes}{\mathcal S}_{\mathcal G_{i}},\  \mathbb{S}_{M}^{i}=\Lambda (T^{*}M)\hat{\otimes}f^{*}{\mathcal S}_{\mathcal G_{i}}
$$
determined by the spinor bundles
$$
 S^{i}_{N}:= \Lambda (T^{*}N)\hat{\otimes} S_{\mathcal G_{i}},\ 
S_{M}^{i}:= \Lambda (T^{*}M)\hat{\otimes} f^{*}S_{\mathcal G_{i}}
$$
where ${\mathcal S}_{\mathcal G_{i}} = S_{\mathcal G_{i}}\otimes  | \mathrm{det}\,  S^{*}_{\mathcal G_{i}}|^{1/r}$, 
and  $f^{*}: \mathbb{S}_{N}^{i}\rightarrow \mathbb{S}_{M}^{i}$ 
are defined by  (\ref{pull-back-forms}).  
In order to prove (\ref{required-pullback})
we fix 
a distribution  $\mathcal D\subset TM$ complementary to $\mathrm{Ker}\, df$ and we 
decompose orthogonally   $f^{!}E_{i}= V^{+}_{i} \oplus V^{-}$, where  $V^{+}_{i}$ and $V_{-}$  are given  by
\begin{align*}
\nonumber& (V^{+}_{i})_{p}= \mathcal D_{p}^{*}\oplus ({\mathcal G_{i}})_{f(p)}\oplus {\mathcal D}_{p}\\
\nonumber& (V^{-})_{p}= (\mathrm{Ker}\, d_{p}f)^{*}\oplus  \mathrm{Ker}\, d_{p}f,
\end{align*}
for any $p\in M.$ 
Assume that $I$ is defined by $(\beta , K, \Phi )$ as in Section \ref{trans-sect-basic}. Then, from Lemma
\ref{pull-back-bdle} iii), 
$I^{f}$ is defined by
$(f^{*}\beta , f^{*}K, f^{*}\Phi )$ and  acts as the identity on 
$V^{-}$ while its restiction 
$ I^{f_{+}}:=I^{f}\vert_{V_{1}^{+}}: V_{1}^{+} \rightarrow V_{2}^{+}$ 
satisfies
\begin{equation}\label{i-f-plus}
(I^{f_{+}} )_{p}(f^{*} u) = f^{*} ( I_{f(p)}(u)),\ \forall u\in  (E_{1})_{f(p)},\ p\in N,
\end{equation}  
where $f^{*} : (E_{i})_{f(p)}\rightarrow (V_{i}^{+})_{p}$
are given by    (\ref{fstar:eq}), constructed  using the distribution $\mathcal D$. Consider the spinor bundles
\begin{equation}
S_{i}^{+}:= \Lambda\, \mathcal D^{*} \hat{\otimes} f^{*}S_{\mathcal G_{i}},\ S^{-} = \Lambda\, (\mathrm{Ker}\, df )^{*}
\end{equation}
of $V_{i}^{+}$ and  $V^{-}$. 
Then  $\bar{S}^{i}_{M}:= S_{i}^{+}\hat{\otimes} S^{-}$ is a spinor bundle of $f^{!}E_{i}$, isomorphic to the spinor bundle
$S^{i}_{M}$ via the $\mathrm{Cl} ( f^{!} E_{i})$- bundle  isomorphism
\begin{equation}\label{T}
T_{i}:\bar{S}^{i}_{M}\rightarrow S^{i}_{M},\ T( (\omega\otimes  s)\otimes \eta )
=(-1)^{ | s| | \eta | } (\omega \wedge \eta )\otimes  s
\end{equation}
where  $\omega \in \Lambda \mathcal D^{*}$ and  $s\in f^{*}S_{\mathcal G_{i}}$, $\eta \in S^{-} $ are homogeneous.  
Let 
\begin{equation}\label{det-isn}
 I^{f_{+}}_{ S^{+}}:= f^{*} \circ  I_{S_{N}}    \circ ( f^{*})^{-1} : S^{+}_{1}\rightarrow S^{+}_{2}, 
\end{equation}
where   $f^{*} : S_{N}^{2} \rightarrow S_{2}^{+}$ 
and $(f^{*})^{-1} : S_{1}^{+} \rightarrow  S_{N}^{1}$ are induced by the pullback (and its inverse). 
From (\ref{i-f-plus}), we obtain that  $I^{f_{+}}_{ S^{+}}$  is compatible with $I^{f_{+}}.$  
Since $I^{f}\vert_{V^{-}}=\mathrm{Id}_{V^{-}}$, 
\begin{equation}\label{ind-sp}
I^{f}_{\bar{S}_{M}} ( s{\otimes} \eta ) := (-1)^{ | \eta | |I^{f_{+}}_{S^{+}}| } I^{f_{+}}_{S^{+}} (s) \otimes \eta,
\end{equation} 
for any $s\in S_{1}^{+}$ and  $\eta \in S^{-}$ homogeneous, 
where $|\eta |$ and $ |I^{f_{+}}_{ S^{+}}| $ denote the degrees of $\eta$ and  $I^{f_{+}}_{ S^{+}}$,
is compatible with $I^{f}$.  
The isomorphism $I^{f}_{\bar{S}_{M}}$ induces, via the  isomorphisms (\ref{T}),  
an isomorphism 
$I^{f}_{ S_{M}}: S_{M}^{1} \rightarrow S_{M}^{2}$ compatible with $I^{f}$,
which maps   $S^{+}_{1}\subset S^{1}_{M}$ onto  $S^{+}_{2}\subset S^{2}_{M}$  
and whose restriction to $S_{1}^{+}$ coincides with $I^{f_{+}}_{S^{+}}$.  As we already know, any isomorphism
compatible with $I^{f}$  and acting between  $S_{M}^{i}$ is uniquely determined up to a multiplicative  factor
and the isomorphism it induces on the canonical weighted spinor bundles $\mathbb{S}_{M}^{i}$
is independent of this factor, up to multiplication by $\pm 1.$
It remains to  show that the isomorphisms
$I^{f}_{ \mathbb{S}_{M}}: \mathbb{S}^{1}_{M}\rightarrow \mathbb{S}_{M}^{2} $ and 
$I_{\mathbb{S}_{N}}:\mathbb{S}_{N}^{1} \rightarrow \mathbb{S}_{N}^{2}$ 
induced by $I^{f}_{ S_{M}}$ 
(defined as above) 
and $I_{S_{N}}$  are related by
(\ref{required-pullback}).  For this,
we use Remark \ref{computation-iso}.
Let  $( s_{i})$, $( \tilde{s}_{i})$ 
be   local frames
of  $\mathcal S_{\mathcal G_{1}}$, 
$\mathcal S_{\mathcal G_{2}}$  and 
$( s_{i}^{*} )$,  $(  \tilde{s}_{i}^{*})$ the dual frames. 
From  Remark (\ref{computation-iso}), 
\begin{align}
\nonumber& I^{f}_{ \mathbb{S}_{M}} ( (\omega \otimes s) \otimes | f^{*} {s}_{1}^{*}\wedge \cdots \wedge 
f^{*}s_{r}^{*}|^{1/2})\\
\label{pb-1}& =  I^{f}_{ {S}_{M}}(\omega \otimes s) \otimes  | f^{*} \tilde{s}_{1}^{*}\wedge \cdots \wedge 
f^{*}\tilde{s}_{r}^{*}|^{1/2} | \mathrm{det} (I^{f}_{ S_{M}}) |^{-\frac{1}{r N_{h} N_{v}} }
\end{align}
where $\omega \in\Lambda (T^{*}M)$,  $s\in f^{*} S_{\mathcal G_{i}}$, 
$N_{h} := \mathrm{rk}\, ( \Lambda \mathcal D )$, $N_{v}:= \mathrm{rk}\, ( \Lambda\,  \mathrm{Ker} df)$,
$r:= \mathrm{rk}\, S_{\mathcal G_{i}}$
and  $\mathrm{det} (I^{f}_{S_{M}}) $ denotes the  determinant of 
 $I^{f}_{ S_{M}}$ with respect to the local frames $( f^{*} s_{i})$ and $( f^{*}\tilde{s}_{i}) .$ 
Similarly, 
\begin{equation}\label{pb-2} I_{\mathbb{S}_{N}} ( (\omega \otimes s) \otimes |  {s}_{1}^{*}\wedge \cdots \wedge 
s_{r}^{*}|^{1/2})=  I_{ {S}_{N}}(\omega \otimes s) \otimes  | \tilde{s}_{1}^{*}\wedge \cdots \wedge 
\tilde{s}_{r}^{*}|^{1/2} | \mathrm{det} (I_{S_{N}}) |^{-\frac{1}{r N_{h}} }
\end{equation}
where $\omega \in \Lambda (T^{*}N)$, $s\in S_{\mathcal G_{1}}$
and  $\mathrm{det} (I_{S_{N}})$ is the  determinant of $I_{S_{N}}$ with respect to
the local frames  $(  s_{i})$ and $(\tilde{s}_{i}).$ 
Using  
(\ref{pb-1}), (\ref{pb-2}) together with  
\begin{equation}\label{computation-determinants}
\mathrm{det} (I^{f}_{ S_{M}}) = 
\mathrm{det} (I^{f_{+}}_{ S_{+}})^{N_{v}} =   \mathrm{det} (I_{ S_{N}})^{N_{v}}\circ f
\end{equation}
we obtain  (\ref{required-pullback}). 
(In the first relation (\ref{computation-determinants})   we used the definition of $I^{f}_{S_{M}}$ and (\ref{ind-sp})
while in the second relation  (\ref{computation-determinants})  we used the definition 
(\ref{det-isn})  of $I_{S_{N}}$). 
\end{proof}

\begin{prop}\label{lem-pull} 
Let $f: M \rightarrow N$ be a submersion,  $E$ a transitive Courant algebroid over $N$
and $\mathbb{S}_{E}$, $\mathbb{S}_{f^{!}E}$ canonical weighted spinor bundles
of $E$ and $f^{!}E$ such that the pullback
$f^{!} : \Gamma (\mathbb{S}_{E}) \rightarrow \Gamma (\mathbb{S}_{f^{!}E})$ is defined.
Let  $\slashed{d}_{E}\in \mathrm{End}\, \Gamma (\mathbb{S}_{E})$ and 
$\slashed{d}_{f^{!}E}\in  \mathrm{End}\, \Gamma (\mathbb{S}_{f^{!}E})$ 
be  the canonical Dirac generating operators of $E$ 
and $f^{!}E$.   
Then
\begin{equation} \label{pull-back-dirac}
f^{!}\circ \slashed{d}_{E} = \slashed{d}_{f^{!}E}\circ f^{!}.
\end{equation}
\end{prop}

\begin{proof} 
From the invariance of the canonical Dirac generating operators under isomorphisms (see  Proposition \ref{iso-dirac})  and the definition of $f^{!}$
(see (\ref{pull-back-relation})) it is sufficient to prove
(\ref{pull-back-dirac}) when  
$E= T^{*}N \oplus  \mathcal G \oplus TN$ is a standard Courant algebroid,  
as in Remark \ref{pull-back-standard}. 
With the notation of that remark, we need to show that  
\begin{equation}\label{pull-back-simple}
f^{*}\circ \slashed{d}_{N} = \slashed{d}_{M}\circ  f^{*} : \Gamma ( \mathbb{S}_{N}) \rightarrow
\Gamma (\mathbb{S}_{M}),
\end{equation}
where $\slashed{d}_{N}$ and $\slashed{d}_{M}$ are the canonical Dirac generating operators 
of $E$ and $f^{!}E = T^{*}M\oplus f^{*}\mathcal G \oplus TM$, which can be computed  using (\ref{can-dirac}).
Let $m$ and $n$ be the dimensions of $M$ and $N$ respectively.
Let $( X_{i})_{1\leq i\leq m}$ be a local frame of $TM$ such that 
$( X_{i})_{1\leq i\leq n}$ are projectable and their projections 
$(f_{*} X_{i })_{i\leq n}$
form  a local frame of $TN$ and $( X_{i})_{n+1\leq i\leq m}$ are vertical.  
Let $(\alpha_{i})_{1\leq i\leq n}$ be the dual frame of $( f_{*}X_{i} )_{1\leq i\leq n}.$  Then, using $f_{*}X_{i}=0$ for any $i\geq n+1$, 
\begin{align*}
\nonumber &\bar{R}^{f^{!}E} (f^{*}(\omega \otimes s))\\
\nonumber&  =\frac{1}{2} \sum_{i, j\leq n} \langle 
f^{*} ( R(f_{*}X_{i}, f_{*}X_{j})), f^{*} r_{k} \rangle_{f^{*}\mathcal G} f^{*} (\alpha_{i}\wedge
\alpha_{j} \wedge \omega ) \otimes  (f^{*} (r_{k}) f^{*}(s))
\end{align*}
that is, 
\begin{equation}\label{2}
(\bar{R}^{f^{!}E} \circ f^{*})(\omega \otimes s) = (f^{*}\circ \bar{R}^{E}) (\omega \otimes s).
\end{equation}
On the other hand, if $\nabla^{S_{\mathcal G}}$ is compatible with $\nabla$ then 
$\nabla^{S_{f^{*}\mathcal G}}:=  f^{*}\nabla^{S_{\mathcal G}} $ is compatible 
with $f^{*}\nabla$ and 
\begin{equation}\label{3}
\nabla^{ \mathcal S_{f^{*}\mathcal G}} = f^{*} \nabla^{\mathcal S_{\mathcal G}}.
\end{equation} 
Relations  (\ref{2}), (\ref{3}), $C_{f^{*}\mathcal G}  = f^{*} C_{\mathcal G}$ 
and the expression of the canonical Dirac generating operator 
(\ref{can-dirac})  imply 
(\ref{pull-back-simple}).
\end{proof}

\subsection{Pushforward on spinors}\label{sect-push}
Let  $f: M \rightarrow  N$ be  a fiber bundle with compact fibers and $M$,   $N$ oriented. Let  $E$ a transitive  Courant algebroid over $N$. 
In this section we define  a pushforward from spinors of $f^{!}E$ to spinors of $E$.
As for the pullback, we assume first that  $E= T^{*}N\oplus \mathcal G \oplus TN$ is a standard 
Courant algebroid, as in Remark
\ref{pull-back-standard}.  We choose an  irreducible  $\mathrm{Cl} ({\mathcal G})$-bundle
$S_{\mathcal G}$, with canonical spinor bundle $\mathcal S_{\mathcal G}$.
Consider an open cover $\mathcal U = \{ U_{i}\}  $ of $N$ and, for any $U_{i}\in \mathcal U$, a 
canonical bilinear pairing  $\langle \cdot , \cdot \rangle_{\mathcal S_{\mathcal G}\vert_{U_{i}}}$ 
on $\Gamma (\mathcal S_{\mathcal G}\vert_{U_{i}}) $.
We define 
 $\langle \cdot , \cdot \rangle_{f^{*}\mathcal S_{\mathcal G}\vert_{f^{-1} (U_{i})} }:= f^{*} \langle \cdot , \cdot \rangle_{\mathcal S_{\mathcal G}\vert_{U_{i}}}$, which is a canonical bilinear
pairing on 
$\Gamma ( {\mathcal S}_{f^{*}\mathcal G}\vert_{ f^{-1} (U_{i})})$, where 
$ {\mathcal S}_{f^{*}\mathcal G}= f^{*} {\mathcal S}_{\mathcal G}$ is the canonical  spinor
bundle of $f^{*}S_{\mathcal G}.$   
We denote by $\langle \cdot , \cdot \rangle_{\mathbb{S}_{U_{i}}}$ and $\langle \cdot , \cdot \rangle_{\mathbb{S}_{f^{-1} (U_{i}) } }$  the corresponding $ \mathrm{det}\, (T^{*}U_{i})$ and 
 $\mathrm{det}\,  (T^{*} f^{-1} (U_{i}) )$-valued canonical bilinear pairings 
 on 
$\Gamma (\mathbb{S}_{N}\vert_{U_{i}})$ and $\Gamma (\mathbb{S}_{M}\vert_{ f^{-1} (U_{i})})$, where  
$\mathbb{S}_{N}=\Lambda (T^{*}N)\hat{\otimes} \mathcal S_{\mathcal G}$ and $\mathbb{S}_{M}= \Lambda (T^{*}M)\hat{\otimes} f^{*} \mathcal S_{\mathcal G} $, see relation (\ref{can-pairing}). 
For any $U_{i}\in \mathcal U$, let
 \begin{equation}\label{push-forward-spinors-ui}
f^{U_{i}}_{*} : \Gamma ( \mathbb{S}_{M}\vert_{f^{-1} (U_{i})} ) = \Omega ( f^{-1} (U_{i})
, f^{*}{\mathcal S}_{\mathcal G}) \rightarrow \Gamma (\mathbb{S}_{N}\vert_{U_{i}}) =
 \Omega (U_{i}, {\mathcal S}_{\mathcal G})
\end{equation}
be defined by 
\begin{equation}\label{defining-pull-back}
\int_{U_{i}} \langle f_{*}^{U_{i}} s_{1}, s_{2}\rangle_{\mathbb{S}_{U_{i} } }= \int_{f^{-1} (U_{i})} \langle s_{1}, f^{*}s_{2}\rangle_{\mathbb{S}_{f^{-1}( U_{i})}},
\end{equation}
for all $s_{1}\in \Gamma  (\mathbb{S}_{M}\vert_{f^{-1} (U_{i})})$ and 
$s_{2}\in \Gamma_{c}  ( \mathbb{S}_{N }\vert_{U_{i}})$,  
where  $\Gamma_{c}(V)$ denotes the space of compactly supported sections of a vector bundle $V$.
Using the maps $f_{*}^{U_{i}}$ 
and a partition of unity $\{ \lambda_{i}\}$ of  $\mathcal U$ 
 we obtain a map 
 \begin{equation}\label{push-forward-spinors}
f_{*} : \Gamma ( \mathbb{S}_{M} ) = \Omega ( M,  f^{*}{\mathcal S}_{\mathcal G}) \rightarrow \Gamma (\mathbb{S}_{N}) = \Omega ( N, {\mathcal S}_{\mathcal G})
\end{equation}
defined by 
\begin{equation}
(f_{*}  s) = \sum_{i} \lambda_{i} f_{*}^{U_{i}} (  s\vert_{f^{-1} (U_{i} )}), \forall \ s\in \Gamma ( \mathbb{S}_{M}).  
\end{equation}
The map  (\ref{push-forward-spinors})  is called the {\cmssl pushforward } on spinors.

\begin{rem}\label{rem-push-forward}{\rm  Recall that the pushforward on forms $f_{*} : \Omega (M) \rightarrow \Omega (N)$
has the properties 
\begin{equation}\label{prop-push-forward}
f_{*} \circ d = d \circ f_{*},\ f_{*} ( ( f^{*} \alpha ) \wedge \beta ) = \alpha \wedge f_{*}\beta ,\ \int_{M}  (f^{*} \alpha )\wedge \beta = \int_{N} \alpha \wedge f_{*}\beta . 
\end{equation}
 Let $U$ be a local chart over which the fiber bundle $f: M \rightarrow N$ is trivial. 
Then we can identify $f^{-1}(U)$ with $U \times F$, where $F$ is the compact  fiber. The decomposition $U\times F$ induces a bigrading on $\Lambda T_p^*M=\Lambda T_x^*U\otimes \Lambda T_t^*F =\bigoplus_{k,\ell} \Lambda^k T_x^*U \otimes \Lambda^\ell T_t^*F$ for all $p=(x,t)\in U\times F$. 
Then $f_*\omega=0$ for every differential form $\omega$ on $U\times F$ of type $(k,\ell )$, $\ell \neq r=\dim F$. 
Choosing a positively oriented volume form $\mathrm{vol}_F$ on the fiber $F$, we can write every differential form of type $(k,r)$ as 
$\omega = h \omega_U \wedge \mathrm{vol}_F$, where $h$ is a function on $U\times F$ and $\omega_U$ is a differential form on $U$. 
Then
\begin{equation}\label{expr-push-forms}
f_{*} \omega =  \omega_U\int_{F} h(x, t) \mathrm{vol}_F(t).
\end{equation}
So $f_*$ is simply integration over the fibers.}
\end{rem}

The next lemma provides a concrete formulation for  the pushforward on spinors  in terms of the pushforward on forms.

\begin{lem} \label{push:lem}For any ${\omega}\otimes f^{*}s\in \Gamma ( \mathbb{S}_{M} )$ such that $s$ is homogenous, 
\begin{equation}\label{concrete-expr-push}
f_{*} (  {\omega}\otimes f^{*}s) = (-1)^{ r | s| + nr + \frac{r(r-1)}{2} }
(  f_{*} {\omega} ) \otimes s,
\end{equation}
where $n$ and $r$ are the dimensions of  $N$ and the fibers of $f$, respectively. 
 In particular, the pushforward is well-defined (i.e. independent on the choice of 
$\mathcal U$,  partition of unity $\{ \lambda_{i}\}$  and canonical bilinear pairings $\langle \cdot , \cdot \rangle_{\mathcal S_{\mathcal G}\vert_{U_{i}}}$).

\end{lem}

\begin{proof}
We  show that for any  $\omega\otimes f^{*} s \in \Gamma (\mathbb{S}_{M} \vert_{ f^{-1} (U_{i})})$ 
with $s$ homogeneous
and
$\tilde{\omega}\otimes \tilde{s}\in \Gamma_{c} ( \mathbb{S}_{N}\vert_{U_{i}})$, 
\begin{equation}\label{integrals}
\int_{U_{i}}\langle ( f_{*}\omega ) \otimes s, \tilde{\omega}\otimes \tilde{s}\rangle_{\mathbb{S}_{U_{i}}} = 
(-1)^{ r | s| + nr + \frac{r(r-1)}{2}}
\int_{ f^{-1} (U_{i}) }\langle \omega \otimes f^{*}s, f^{*}(\tilde{\omega}\otimes \tilde{s})\rangle_{\mathbb{S}_{ f^{-1} (U_{i})}} .
\end{equation}
In order to prove (\ref{integrals}), we  assume, without loss of generality, that  $\omega $, $\tilde{\omega}$ and $\tilde{s}$ are also 
homogeneous. If $| \omega | + | \tilde{\omega } | \neq m$ (where $m:= n + r$)  both terms in (\ref{integrals}) vanish. Assume now
that $| \omega | + | \tilde{\omega } | =m.$  Then 
\begin{eqnarray*}
 \int_{ f^{-1} (U_{i} )}\langle \omega \otimes f^{*}s, f^{*}(\tilde{\omega}\otimes \tilde{s}) \rangle_{\mathbb{S}_{ f^{-1} (U_{i})}} 
&=&  (-1)^{ | s| m + | \omega | | \tilde{\omega } | }\int_{ f^{-1} (U_{i})} f^{*}( \langle s, \tilde{s}\rangle_{\mathcal S_{\mathcal G}} 
\tilde{\omega } ) \wedge \omega^{t}\\
&=&  (-1)^{ |s| m + | \omega | | \tilde{\omega} |}\int_{ U_{i}} \langle s, \tilde{s}\rangle_{\mathcal S_{\mathcal G}}
\tilde{\omega}\wedge f_{*} (\omega^{t})\\  
&=&  (-1)^{r( m - \frac{r+1}{2} - | s| )}\int_{ U_{i} } \langle  (f_{*}\omega ) \otimes s, \tilde{\omega}\otimes \tilde{s}\rangle_{\mathbb{S}_{U_{i}}}, 
\end{eqnarray*}
where  we used $f_{*} (\omega^{t}) = ( f_{*}\omega )^{t} (-1)^{ \frac{r (r-1)}{2} + r ( | \omega | - r) }$, which can be checked using (\ref{expr-push-forms})  and the third property  (\ref{prop-push-forward}) of the pushforward on  forms.
Relation (\ref{integrals}) is proved and implies (\ref{concrete-expr-push}).
\end{proof}

\begin{rem}\label{later}{\rm 
In the above setting, assume that $f$ is endowed with an horizontal distribution, like in 
Remark \ref{pull-back-standard} ii). Then 
\begin{equation}\label{clifford-push-1}
f_{*} (f^{*} (u)\cdot s) = u\cdot f_{*} s,\ \forall u\in \Gamma \mathrm{Cl}(E),\ s\in \Gamma (\mathbb{S}_{f^{!}E}). 
\end{equation} 
where  $f^{*}: \Gamma \mathrm{Cl}(E) \rightarrow \Gamma \mathrm{Cl}( f^{!}E)$ 
is the map  (\ref{fstar:eq}). }
\end{rem}

Assume now that    $E$ is a transitive, but not necessarily standard,  Courant algebroid. Then 
we can define the pushforward 
$ f_{!}: \Gamma (\mathbb{S}_{f^{!}E}) \rightarrow \Gamma (\mathbb{S}_{E})$ 
for any canonical weighted spinor bundles  $\mathbb{S}_{E}$ and $\mathbb{S}_{f^{!}E}$,  
for which the  pullback $f^{!} :\Gamma (\mathbb{S}_{E})\rightarrow
\Gamma (\mathbb{S}_{f^{!}E})$ is defined. 
Namely, we consider an admissible pair $(I : E \rightarrow T^* N \oplus\mathcal G\oplus TN, 
S_{\mathcal G})$ for $\mathbb{S}_{E}$ and $\mathbb{S}_{f^{!}E}$ 
and we define the {\cmssl pushforward on spinors}
\begin{equation}\label{P-F}
f_{!} :\Gamma ( \mathbb{S}_{f^{!}E} )\rightarrow\Gamma ( \mathbb{S}_{E}),\ f_{!}:=  (I_{\mathbb{S}} )^{-1}\circ f_{*}\circ I_{\mathbb{S}}^{ f}
\end{equation}
where  $f_{*}:\Gamma (\mathbb{S}_{M})\rightarrow \Gamma (\mathbb{S}_{N})$ is  the map  (\ref{defining-pull-back}). 
Remark that if $\langle \cdot , \cdot \rangle_{\mathbb{S}_{E}{\vert_{U_{i}}}}: = (I_{\mathbb{S}})^{*} \langle \cdot , \cdot \rangle_{\mathbb{S}_{U_{i}}}$  and $\langle \cdot , \cdot \rangle_{\mathbb{S}_{f^{!}E}\vert_{f^{-1}(U_{i})} }:=
(I^{f}_{\mathbb{S}})^{*}  \langle \cdot , \cdot \rangle_{\mathbb{S}_{f^{-1}(U_{i})}}$, then
\begin{equation}
\int_{U_{i}} \langle f_{!} s_{1}, s_{2}\rangle_{\mathbb{S}_{E}\vert_{U_{i}} }= \int_{f^{-1}(U_{i})} \langle s_{1}, f^{!}s_{2}\rangle_{\mathbb{S}_{f^{!}E}\vert_{f^{-1}(U_{i})}},  
\end{equation} 
for any $s_{1}\in \Gamma (\mathbb{S}_{f^{!}E}\vert_{f^{-1}(U_{i} )})$ and  $s_{2}\in \Gamma_{c} ( \mathbb{S}_{E}\vert_{U_{i}} )$,
where $f^{!}=( I^{f}_{\mathbb{S}})^{-1} \circ f^{*}\circ I_{\mathbb{S}}$, cf.\ Lemma \ref{well-def-lem}. 
In particular, (\ref{P-F}) is well defined, up to  multiplication by $\pm 1$.

\begin{prop}\label{lem-push}
The pushforward  $f_{!}:\Gamma (f^{!}E) \rightarrow \Gamma ( E)$ 
commutes with the canonical Dirac generating operators,  i.e.\  
 $f_{!}\circ \slashed{d}_{f^{!}E} = \slashed{d}_{E}\circ f_{!}.$
\end{prop}

\begin{proof}  
Like in the proof of Proposition \ref{lem-pull}, it is sufficient to show that
\begin{equation}\label{98}
f_{*}\circ \slashed{d}_{M} =\slashed{d}_{N}\circ f_{*},
\end{equation}
where we preserve the notation from the proof of that proposition.
From (\ref{clifford-push-1})  
we know
\begin{equation}\label{clifford-push}
f_{*} (f^{*} (u)\cdot s) = u\cdot f_{*} s,\ \forall u\in \Gamma (T^{*}N\oplus {\mathcal G} ),\  s\in \Gamma (\mathbb{S}_{ M}).
\end{equation}
Iterating  relation   (\ref{clifford-push})
 and using (\ref{pull-back-clifford}),  we see that (\ref{clifford-push})  holds  for any $u\in \Gamma  \Lambda (T^{*}N\oplus \mathcal G)$.  
Using the expression (\ref{can-dirac-sometimes}) of the canonical Dirac generating operator  and property (\ref{clifford-push}) of $f_{*}$ (with 
$u:= H$,  $\alpha_{i}$, $C_{\mathcal G}$), 
we obtain that 
\begin{equation}
\slashed{d}_{N}  f_{*}(\tilde{\omega } \otimes f^{*}s) = f_{*} \slashed{d}_{M} (\tilde{\omega} \otimes f^{*}s),\ \forall \tilde{\omega}\in \Omega (M),\ s\in \Gamma ({\mathcal S}_{\mathcal G})
\end{equation}
reduces  to 
\begin{equation}\label{without-int}
f_{*} {\mathcal E}_{M} (\tilde{\omega }\otimes f^{*} s) ={\mathcal  E}_{N} f_{*} (\tilde{\omega}\otimes s),
\end{equation}
where
\begin{align}
 \nonumber& {\mathcal E}_{N}(\omega \otimes s): = (d\omega )\otimes s +\sum_{i}  (\alpha_{i}\wedge \omega ) \otimes \nabla_{X_{i}}^{{\mathcal S}_{\mathcal G}}s\\
  \nonumber& {\mathcal E}_{M}(\tilde{\omega} \otimes f^{*}s) := (d\tilde{\omega} )\otimes f^{*}s + \sum_{i}( (f^{*}\alpha_{i})\wedge \tilde{\omega} ) \otimes  f^{*}(\nabla_{X_{i}}^{{\mathcal S}_{\mathcal G}}s),
\end{align}
for any $\omega \in \Omega (N)$, $\tilde{\omega}\in \Omega (M)$ and $s\in \Gamma ({\mathcal S}_{\mathcal G})$. 
In order to show (\ref{without-int}) it is sufficient to show that for any $U\subset N$ open and sufficiently small
and  $\beta \in \Gamma_{c} (\mathbb{S}_{N}\vert_{U})$, 
\begin{equation}\label{needed}
\int_{U} \langle {\mathcal E}_{N} f_{*} (\tilde{\omega} \otimes f^{*} s), \beta \rangle_{\mathbb{S}_{U}} 
= \int_{f^{-1}(U)} \langle {\mathcal E}_{M} (\tilde{\omega} \otimes f^{*}s), f^{*}\beta\rangle_{\mathbb{S}_{f^{-1}(U)}}.
\end{equation}
From Corollary  \ref{de-adaugat:cor} and $f^{*} {\mathcal E}_{N} ={\mathcal  E}_{M} f^{*}$ we have 
\begin{align*}
\nonumber&\int_{U} \langle {\mathcal E}_{N} f_{*}( \tilde{\omega}\otimes f^{*}s), \beta\rangle_{\mathbb{S}_{U}} = - \int_{U} \langle f_{*} (\tilde{\omega} \otimes f^{*}s), {\mathcal E}_{N} \beta \rangle_{\mathbb{S}_{U}}\\ 
\nonumber&= - \int_{f^{-1} (U) } \langle \tilde{\omega}\otimes f^{*}s, f^{*} {\mathcal E}_{N} \beta \rangle_{\mathbb{S}_{f^{-1}(U)}}
=  -\int_{f^{-1} (U)} \langle \tilde{\omega}\otimes f^{*}s,  {\mathcal E}_{M} f^{*}\beta \rangle_{\mathbb{S}_{f^{-1} (U)}}\\  
\nonumber&= \int_{ f{-1} (U)} \langle {\mathcal E}_{M} (\tilde{\omega}\otimes f^{*}s), f^{*}\beta
\rangle_{\mathbb{S}_{f^{-1}(U)}},
\end{align*}
which  proves (\ref{needed}).
\end{proof}

\section{Actions  on transitive Courant algebroids}

\subsection{Basic properties}

In this section we consider a class of actions on a transitive Courant algebroid which  generalizes 
torus actions on exact and, more generally,  on heterotic Courant  algebroids.  For the latter types of Courant algebroids, a notion of $T$-duality has been developed  in   \cite{T-duality-exact} and \cite{baraglia} respectively.

Let $E$ be a transitive Courant algebroid over a manifold $M$, with anchor $\pi : E\rightarrow TM$,
Dorfmann bracket  $[\cdot , \cdot ]$ and scalar product  $\langle \cdot , \cdot \rangle$.  
Recall that the automorphism group $\mathrm{Aut}(E)$ of $E$ is the group of orthogonal  automorphisms
$F: E \rightarrow E$ which cover a diffeomorphism $f: M \rightarrow M$,  such that 
$$\pi ( F(u) )=
(d_{p} f) \pi (u),\ \forall u\in E_{p},\ p\in M
$$  
and the natural map induced by 
$F$ on the space of sections of $E$  preserves the Dorfmann bracket.
Its  Lie algebra is
the Lie algebra $\mathrm{Der} (E)$  of derivations
of $E$. This   is the  subalgebra 
of $\mathrm{End}\, \Gamma (E)$ 
of  
first order linear differential operators $D: \Gamma (E)\rightarrow \Gamma (E)$  which satisfy, for any 
$s, s_{1}, s_{2}\in \Gamma (E)$,
\begin{align}
\nonumber& D [s_{1}, s_{2} ] = [Ds_{1}, s_{2} ] + [ s_{1}, D s_{2} ]\\
\nonumber& X \langle s_{1}, s_{2} \rangle = \langle D s_{1}, s_{2} \rangle + \langle s_{1}, D  s_{2} \rangle\\
\label{cond-deriv}&\pi  \circ D (s) = {\mathcal L}_{X}\pi  (s), 
\end{align}
where $X\in {\mathfrak X}(M)$ 
is a vector field on $M$, uniquely determined by $D$ (from the second relation
(\ref{cond-deriv}))  and usually denoted by $\pi  (D).$

Let $\mathfrak{g}$ be a Lie algebra acting on $M$ by an infinitesimal
action 
$$
\psi : \mathfrak{g} \rightarrow {\mathfrak X}(M),\ a\mapsto \psi (a) = X_{a}.
$$
We will always assume (without repeating it each time) that all the infinitesimal actions considered are {\cmssl free}, which means that the fundamental vector fields $X_{a}$ are non-vanishing, for all $a\in \mathfrak{g}\setminus \{ 0\}$.

\begin{defn}\label{def-triv-ext} i) An {\cmssl  (infinitesimal) action} of $\mathfrak{g}$ on $E$ which lifts $\psi$ is an
algebra homomorphism $\Psi : \mathfrak{g} \rightarrow \mathrm{Der}(E)$ 
which satisfies $\pi \Psi (a) = X_{a}$ for any $a\in \mathfrak{g}$.\

ii) Let $\Psi :  \mathfrak{g} \rightarrow \mathrm{Der}(E)$ be an  action of $E$ which lifts $\psi .$
An {\cmssl invariant dissection} of $E$ is a dissection $I: E \rightarrow T^{*}M\oplus \mathcal G \oplus TM$
for which the action 
$$
\mathfrak{g}\ni a\rightarrow  I\circ \Psi (a)\circ I^{-1}\in \mathrm{Der} (T^{*}M\oplus \mathcal G\oplus TM)
$$ 
preserves the  summands  $T^{*}M$, $\mathcal G$ and $TM.$
\end{defn}

We will only consider  (without repeating it each time) infinitesimal actions on Courant algebroids for which
 there is  an invariant dissection. The next proposition shows that this is automatically the case if the infinitesimal action is 
 induced from an action of a compact group.
 \begin{prop} Let $\Psi : G \rightarrow \mathrm{Aut}(E)$ be an action of a compact group $G$ by 
 automorphisms of a Courant algebroid $E$ over $M$, hence covering a group action $\psi: G\rightarrow \mathrm{Diff}(M)$. 
 Then $E$ admits a dissection invariant under $\Psi$. 
 \end{prop}
 
 \begin{proof} By compactness of $G$ there exists a $G$-invariant positive definite metric $h$ in $E$. Using the auxiliary metric 
 $h$ we can define a $G$-invariant splitting $\sigma_0 : TM \rightarrow E$ of the anchor map $\pi : E \rightarrow TM$, where 
 $\sigma_0(TM)$ is h-orthogonal complement of $\mathrm{Ker}\, \pi$. The section $\sigma_0$ of $\pi$ can be canonically 
 modified to a $G$-invariant totally isotropic section $\sigma$ defined by 
 \[\langle \sigma (X), v\rangle = \langle \sigma_0(X),v-\frac12 \sigma_0(\pi (v))\rangle\] 
 for all $X\in T_pM$, $v\in E_p$, $p\in M$. 
If we define $\mathcal{G}$ as the $\langle \cdot ,\cdot \rangle$-orthogonal complement of $\pi^*T^*M \oplus \sigma (TM)$, 
then $E= \pi^*T^*M \oplus \mathcal{G}\oplus \sigma (TM)$ is a 
$G$-invariant dissection.     
 \end{proof}

 In the remaining part of this section we assume that 
\begin{equation}\label{decomp:eq}E = T^{*}M \oplus \mathcal G \oplus TM
\end{equation}
is a standard Courant
algebroid,  
defined by a quadratic Lie algebra bundle $(\mathcal G , [\cdot , \cdot]_{\mathcal G}, \langle \cdot , \cdot\rangle_{\mathcal G})$ and   data $(\nabla , R, H)$ as in 
Section \ref{trans-sect-basic} and we  consider  in detail the class of  actions  $\Psi : \mathfrak{g}\rightarrow \mathrm{Der}(E)$   which lift $\psi$ and preserve the factors $T^{*}M$, $\mathcal G$ and $TM$ of $E$. 
From the third condition 
(\ref{cond-deriv}),   the restriction of $\Psi$   to $TM$ is given by 
\begin{equation}\label{act-TN}
\Psi (a) (X) = {\mathcal L}_{X_{a}} X,\  \forall   a\in \mathfrak{g},\ X\in {\mathfrak X}(M).
\end{equation} 
Since  $X_{a}$ (with $a\in \mathfrak{g}\setminus \{ 0\}$) are nowhere vanishing  we can define
$$
\nabla^{\Psi}_{X_{a}(p)} r :=\left( \Psi (a) (r)\right) (p),\ \forall a\in \mathfrak{g},\ r\in \Gamma (\mathcal G),\  p\in M,
$$
which is a partial connection on $\mathcal G .$

\begin{lem}\label{cond-triv-ext} There is a one to one correspondence between actions 
$\Psi : \mathfrak{g} \rightarrow \mathrm{Der}(E)$ which lift
$\psi$ and preserve the factors $T^{*}M$, $\mathcal G$, $TM$ of $E$ 
and partial connections $\nabla^{\Psi}$ on $\mathcal G$ such that the following conditions are satisfied:

i) $\nabla^{\Psi}$ is flat and preserves $[\cdot , \cdot ]_{\mathcal G}$ and $\langle \cdot , \cdot \rangle_{\mathcal G}$;\

ii) $H$ and $R$ are invariant, i.e. for any $a\in \mathfrak{g}$,  
\begin{equation}\label{inv-H-R}
{\mathcal L}_{X_{a}} H =0,\   {\mathcal L}_{\Psi (a)} R=0
\end{equation}
where 
\begin{equation}\label{Lie-R}
 ({\mathcal L}_{\Psi (a)} R) (X, Y)  := \nabla_{X_{a}}^{\Psi}( R(X, Y)) - R({\mathcal L}_{X_{a}}X, Y) - 
R(X, {\mathcal L}_{X_{a}} Y)
\end{equation}
for any $X, Y\in {\mathfrak X}(M)$;

iii) for any $a\in \mathfrak{g}$, the  endomorphism $A_{a}:= \nabla_{X_{a}}^{\Psi} - \nabla_{X_{a}}$  
satisfies 
\begin{equation}\label{nabla-A}
(\nabla_{X} A_{a}) (r) = [ R(X_{a}, X), r]_{\mathcal G},\ \forall r\in \Gamma (\mathcal G ).
\end{equation}
If i), ii) and iii) are satisfied,  then  the corresponding action  $\Psi$  acts naturally (by Lie derivative)  on the subbundle  $T^{*}M \oplus TM$ of $E$, i.e.
\begin{equation}\label{u1}
\Psi (a) (\xi + X ) = {\mathcal L}_{X_{a}} ( \xi + X ),\ X\in {\mathfrak X}(M),\ \xi \in \Omega^{1}(M),
\end{equation}
and on $\mathcal G$ by
\begin{equation}\label{u2}
\Psi (a)(r)= \nabla^{\Psi}_{X_{a}} r,\  r\in \Gamma (\mathcal G).
\end{equation}
Moreover, for any $a\in \mathfrak{g}$, the endomorphism $A_{a}$ 
is a skew-symmetric derivation of $\mathcal G$. 
\end{lem}

\begin{proof} Let $\Psi $ be 
an  action as in the statement of the lemma.  
From (\ref{act-TN}), 
$$
X_{a} \langle X, \eta \rangle = \langle {\mathcal L}_{X_{a}} X, \eta \rangle + \langle X, \Psi (a) (\eta )\rangle ,\
\forall a\in \mathfrak{g},\ X\in {\mathfrak X}(M),\ \eta \in \Omega^{1}(M),
$$
and from  the fact that $\Psi (a)$ preserves $\Omega^{1}(M)\subset \Gamma (E)$, 
we obtain that  $\Psi (a) (\eta ) = {\mathcal L}_{X_{a}} \eta$. 
Relation (\ref{u1}) follows. From our comments above,  $\nabla^{\Psi}$ defined by (\ref{u2})  is a partial 
connection on $\mathcal G$. 
Using (\ref{expr-courant}) 
we obtain that the relations  (\ref{cond-deriv}) satisfied by $\Psi$  are equivalent to the following
conditions:   $R$ and $H$ are invariant,
$\nabla^{\Psi}$ is flat, preserves $[\cdot , \cdot ]_{\mathcal G}$ and $\langle\cdot , \cdot \rangle_{\mathcal G}$, 
and 
\begin{align}
\nonumber&  \nabla^{\Psi}_{X_{a}} \nabla_{X} r- \nabla_{X} \nabla^{\Psi}_{X_{a}} r-\nabla_{{\mathcal L}_{X_{a}}X} r=0\\
\nonumber& \mathcal{L}_{X_{a}} \langle i_{X}R, r\rangle_{\mathcal G} = \langle i_{ {\mathcal L}_{X_{a}} X} R, r\rangle_{\mathcal G} + \langle i_{X}R, \nabla^{\Psi}_{X_{a}} r\rangle_{\mathcal G} \\
\label{invariance}& {\mathcal L}_{X_{a}} \langle \nabla r, \tilde{r}\rangle_{\mathcal G}=
\langle \nabla ( \nabla^{\Psi}_{X_{a}} r),  \tilde{r}\rangle_{\mathcal G} +\langle \nabla r, \nabla^{\Psi}_{X_{a}} \tilde{r}
\rangle_{\mathcal G},
\end{align}
for any $a\in \mathfrak{g}$, $X\in {\mathfrak X}(M)$ and $r, \tilde{r}\in \Gamma (\mathcal G).$ 
The first relation (\ref{invariance}) is equivalent to (\ref{nabla-A}). Since both
$\nabla$ and $\nabla^{\Psi}$ preserve $\langle \cdot , \cdot \rangle_{\mathcal G}$ 
and $[\cdot , \cdot ]_{\mathcal G}$, the endomorphism $A_{a}$ is a skew-symmetric derivation.
The second relation (\ref{invariance}) is equivalent to 
\begin{equation}
\langle A_{a} R(X, Y) , r\rangle_{\mathcal G} + \langle R(X, Y),  A_{a}r\rangle_{\mathcal G} =0  
\end{equation}
and follows from the skew-symmetry of the endomorphism $A_{a}$.
The third relation (\ref{invariance}) follows 
from the fact that $\nabla$ preserves
$\langle\cdot , \cdot \rangle_{\mathcal G}$,
by writing $\nabla^{\Psi}_{X_{a}} = \nabla_{X_{a}} + A_{a}$ and using relation (\ref{nabla-A}) together
with   $R^{\nabla}(X_{a}, X) (r) = [ R(X_a, X), r]_{\mathcal G}.$ 
\end{proof}

\begin{cor}\label{cor-A} The skew-symmetric derivations  $A_{a}$  from Lemma \ref{cond-triv-ext} satisfy 
\begin{equation}\label{parallel-A}
\nabla_{X_{b}}^{\Psi} (A_{a}) = A_{[b, a]},\ \forall a, b\in \mathfrak{g}.
\end{equation}
\end{cor}

\begin{proof}  From $\nabla^{\Psi}_{X_{a}} = \nabla_{X_{a}} + A_{a}$, the flatness of
$\nabla^{\Psi}$  and the expression
(\ref{cond-curv}) 
of $R^{\nabla}$, we obtain, for any $r\in \Gamma (\mathcal G)$, 
\begin{equation}
[ R(X_{a}, X_{b}), r]_{\mathcal G} + (d^{\nabla}A)(X_{a}, X_{b}) (r) + [A_{a}, A_{b} ](r)=0,\ 
\forall a, b\in \mathfrak{g}, 
\end{equation}
where $[A_{a}, A_{b} ] := A_{a} A_{b} - A_{b} A_{a}$ is the commutator of $A_{a}$ and $A_{b}.$ 
But 
\begin{align}
\nonumber  (d^{\nabla} A)(X_{a}, X_{b}) (r)& = \left( \nabla_{X_{a}} A_{b} -\nabla_{X_{b}} A_{a} - A_{[a, b]}\right)  r\\
\nonumber& = 2 [ R(X_{b}, X_{a}), r]_{\mathcal G}- A_{[a, b]} (r)
\end{align}
where we used relation (\ref{nabla-A})
and $[X_{a}, X_{b} ] = X_{[a, b]}$.
We obtain 
\begin{equation}\label{relations-Aa}
[ R(X_{b}, X_{a}), r]_{\mathcal G} - A_{[a, b]} r +  [A_{a}, A_{b}](r) =0,\ \forall a, b\in \mathfrak{g},\ r\in {\mathcal G}.
\end{equation}
On the other hand, 
\begin{align}
\nonumber \nabla^{\Psi}_{X_{b}} (A_{a}) (r)&= \nabla_{X_{b}}^{\Psi} (A_{a} (r)) -A_{a}( \nabla^{\Psi}_{X_{b}} r)\\
\nonumber& = \nabla_{X_{b}} (A_{a}) (r) + [A_{b}, A_{a} ] (r)\\
\label{opA-2}& = [ R(X_{a}, X_{b}) , r]_{\mathcal G} - [A_{a}, A_{b}](r)
\end{align} 
where in the second equality we used 
$\nabla^{\Psi}_{X_{b}} = \nabla_{X_{b}} + A_{b}$
and in the third equality we used again  relation (\ref{nabla-A}).
Combining  (\ref{relations-Aa}) with  (\ref{opA-2}) we obtain  (\ref{parallel-A}).
\end{proof}

\begin{rem}\label{comentarii}{\rm  i) The first relation (\ref{invariance}) implies that  $\nabla$ is {\cmssl  invariant}, i.e.
$$
({\mathcal L}_{\Psi (a)} \nabla )_{X} r:= \Psi (a) (\nabla_{X}r) - \nabla_{\mathcal L_{X_{a}} X} r- \nabla_{X}
(\Psi (a) r) =0,
$$
for any $a\in \mathfrak{g}$, $X\in {\mathfrak X}(M)$ and $r\in \Gamma (\mathcal G).$\

ii)  Like for $R$, we can define the Lie derivative 
\begin{align*}
\nonumber& ({\mathcal L}_{\Psi (a)} \alpha )(X_{1},\cdots , X_{k}):= \Psi (a) ( \alpha (X_{1},\cdots , X_{k})) \\
\nonumber& -\alpha ({\mathcal L}_{X_{a}} X_{1},\cdots , X_{k}) -\cdots - \alpha (X_{1},\cdots ,{\mathcal L}_{X_{a}}X_{k}),
\end{align*}
for any form $\alpha \in \Omega^{k} (M, \mathcal G)$. The Lie derivative so defined can be extended in the usual way to forms   with values in the tensor bundle ${\mathcal T } (\mathcal G )$ of $\mathcal G .$  
In particular,  for $\alpha \in \Omega^{k}(M)$ we simply define ${\mathcal L}_{\Psi (a)}\alpha 
:= {\mathcal L}_{X_{a}}\alpha .$ A  ${\mathcal T}(\mathcal G)$-valued  form $\alpha$ is called {\cmssl  invariant} if ${\mathcal L}_{\Psi (a)} \alpha =0$
for any $a\in \mathfrak{g}.$\

iii) Relation (\ref{parallel-A}) can be written in the equivalent way 
\begin{equation}\label{above-eqn}
{\mathcal L}_{\Psi (b)}(A_{a}) =  A_{[b,a]},\ \forall a, b\in \mathfrak{g}.
\end{equation}
In particular,   the endomorphisms $A_{a}$ are invariant 
when $\mathfrak{g}$ is abelian.}
\end{rem}

Let $E_{i}$  ($i=1,2$) be two transititive Courant algebroids over $M$  
and  $\Psi_i: \mathfrak{g} \rightarrow \mathrm{Der} (E_{i})$   actions which lift $\psi .$
A fiber preserving  Courant algebroid isomorphism $F: E_{1} \rightarrow E_{2}$ is called {\cmssl  invariant} if
\begin{equation}\label{inv-F}
\Psi_{2} (a) (F(u))   = F \Psi_{1}(a) (u),\ \forall a\in \mathfrak{g},\ u\in \Gamma (E_{1}).
\end{equation}

\begin{lem}\label{cond-inv-iso}
Let $E_{i}$ ($i=1,2$) be standard Courant algebroids over $M$ defined by
quadratic Lie algebra bundles $(\mathcal G_{i}, [\cdot , \cdot ]_{\mathcal G_{i}}, \langle \cdot , \cdot \rangle_{\mathcal G_{i}})$
and  the  data 
$(\nabla^{(i)},R_i,H_i)$. 
Assume that $E_{i}$  are
endowed with   actions $\Psi_i: \mathfrak{g} \rightarrow \mathrm{Der} (E_{i})$ which lift $\psi : \mathfrak{g} \rightarrow
{\mathfrak X}(M)$ 
and preserve the factors $T^{*}M$, $\mathcal G_{i}$ and $TM$ of $E_{i}$,  
and that 
a fiber preserving Courant algebroid  isomorphism $F : E_{1} \rightarrow E_{2}$,
defined by $( \beta ,\Phi , K)$,  where $\beta \in \Omega^{2}(M)$, $\Phi \in \Omega^{1}(M,\mathcal G_{2})$ and $K\in \mathrm{Isom} (\mathcal G_{1}, \mathcal G_{2})$, is given,  as in (\ref{concrete-iso}).
Let ${\nabla}^{\Psi_i} := \left( \nabla^{(i)}\right)^{\Psi_i}$   ($i=1,2$) be the partial connections associated with $\nabla^{(i)}$ and $\Psi_{i}.$ 
Then $F$ is invariant if and only if  $K$ maps ${\nabla}^{\Psi_1}$ to ${\nabla}^{\Psi_2}$
and the forms $\beta$ and $\Phi$ are invariant.
\end{lem} 

\begin{proof}
The proof  uses the expression (\ref{expr-courant}) for the Dorfman bracket. 
\end{proof}

\subsubsection{A class of $T^k$-actions}\label{example-t1}

Let $(E=T^*M \oplus \mathcal{G} \oplus TM, \langle \cdot ,\cdot \rangle, [ \cdot ,\cdot ])$ be a standard  Courant algebroid over 
the total space of a principal $T^{k}$-bundle $\pi : M\rightarrow B$, 
where 
$T^k = \mathbb{R}^{k}/\mathbb{Z}^{k}$ denotes the  $k$-dimensional torus.
We assume that $E$  is defined by a bundle of quadratic Lie algebras $(\mathcal G , \langle \cdot ,\cdot \rangle_{\mathcal G},
[\cdot , \cdot]_{\mathcal  G})$   and  data $(\nabla , R, H)$, where $\nabla$ is a connection 
on the vector bundle  $\mathcal G$ compatible with the tensor fields  $\langle \cdot ,\cdot \rangle_\mathcal{G}$ and 
$[\cdot , \cdot ]_\mathcal{G}$,  
$R\in \Omega^2(M,\mathcal{G})$ and $H\in \Omega^3(M)$.  
Recall that these data satisfy 
the compatibility equations
\begin{equation}\label{comp:eq}
dH= \langle R\wedge R\rangle\ ,\quad d^\nabla R=0\ ,\quad R^\nabla = \mathrm{ad}_R\ ,
\end{equation}
where $\langle R\wedge R\rangle_\mathcal{G}$ is abbreviated as $\langle R\wedge R\rangle$ in harmony with 
the fact that $\langle \cdot , \cdot \rangle_\mathcal{G}=\langle \cdot , \cdot \rangle|_{\mathcal{G}\times \mathcal{G}}$.  
The Dorfmann bracket, scalar product and anchor of $E$ are then expressed by the usual formulas in terms of the above data.

We assume that the vertical paralellism of $\pi$  is lifted 
to an action  of $\mathfrak{t}^{k}$ on  $E$,  
\[ \Psi : \mathfrak{t}^k=\mathbb{R}^{k} \rightarrow \mathrm{Der} (E)\ ,\quad a\mapsto \Psi(a)  = ( \xi + r + X \mapsto 
{\mathcal L}_{X_{a}} \xi   + \nabla^{\Psi}_{X_{a}} r  + {\mathcal L}_{X_{a}} X),\] 
where  $X_a$ is the fundamental vector field of $\pi$ determined by $a\in \mathfrak{t}^{k}$ and 
$\nabla^{\Psi}$ is a partial flat connection on $\mathcal G$.  We recall (see Lemma \ref{cond-triv-ext}) that 
$\nabla^{\Psi}$  preserves $[\cdot , \cdot ]_{\mathcal G}$
and $\langle \cdot , \cdot \rangle_{\mathcal G}$ and that
\begin{equation} \label{compatibility:eq}
{\mathcal L}_{X_{a}} R =0,\ {\mathcal L}_{X_{a}}H=0,\  (\nabla_{X} A_{a}) (r) = [ R(X_{a}, X), r]_{\mathcal H},
\end{equation}
where 
$A_{a}:= \nabla^{\Psi}_{X_{a}} - \nabla_{X_{a}} \in \mathrm{End} (\mathcal G )$ is a skew-symmetric derivation,
which is invariant since $\mathfrak{t}^{k}$ is abelian (see Remark \ref{comentarii}).
Recall also  that ${\mathcal L}_{X_{a}} \nabla =0$.\

We consider 
\[ \Omega_b^s (M, \mathcal{G}) := \{ \alpha \in \Omega^s (M, \mathcal{G}) \mid \mathcal{L}_{\Psi (a)} \alpha = 0\ , i_{X_a}\alpha =0,\
\forall a\in \mathfrak{t}^{k}\},\]
the space of {\cmssl basic} $\mathcal{G}$-valued $s$-forms on $M$. 
The space $\Omega_b^s(M)$ of basic scalar valued $s$-forms on $M$ can be defined similarly and coincides  with $\pi^*\Omega^s (B)\cong \Omega^s (B)$. The analogous fact for $\Omega_b^s (M, \mathcal{G})$ is stated in the 
next proposition.

\begin{prop} \label{Omega_b:prop}$\Omega_b^s (M, \mathcal{G}) \cong \pi^*\Omega^s (B)\otimes \Gamma_{t^k}(\mathcal{G})$, where 
$\Gamma_{t^k}(\mathcal{G})$ denotes  the space of $\mathfrak{t}^k$-invariant (i.e.\ $\nabla^\Psi$-parallel) sections.
\end{prop}

\begin{proof}
Let $U\subset B$ be an open set such that $\Lambda^s T^*B|_U$ is trivial. Then 
any horizontal  form $\alpha \in \Omega^{s}(\pi^{-1}(U), {\mathcal G})$  
(i.e. $i_{X_{a}}\alpha =0$ for any $a\in \mathfrak{t}^{k}$) can be written as 
$\alpha =\sum_{i} (\pi^{*}\beta_{i})\otimes s_{i}$ where $( \beta_{i})$ is a basis of
$\Lambda^{s}T^{*}B\vert_{U}$ and $s_{i}\in \Gamma (\mathcal G\vert_{\pi^{-1}(U)}).$
Then ${\mathcal L}_{\Psi (a)}\alpha = \sum_{i} (\pi^{*}\beta_{i})\otimes {\mathcal L}_{\Psi (a)}s_{i}$ 
 from where we deduce  that
$\Omega_b^s (\pi^{-1}(U),\mathcal{G}) = \pi^*\Omega^s (U) \otimes \Gamma_{t^k}( \mathcal{G}\vert_{\pi^{-1}(U)})$.
Using a partition of unity in $B$ one can deduce that the same holds globally for $U=B$. 
\end{proof}

In the following we will always identify $\Lambda^s T^*M\otimes \mathcal{G}$ with 
$\mathcal{G}\otimes \Lambda^s T^*M$, which allows to freely write decomposable elements as $\omega \otimes r$ or as $r\otimes \omega$. 
Let $(e_{i})$ be a basis of $\mathfrak{t}^{k}$, 
$X_{i}:= X_{e_{i}}$ the associated fundamental vector fields 
and $A_{i}:= A_{e_{i}} = \nabla^{\Psi}_{X_{i}} - \nabla_{X_{i}}\in \mathrm{End} (\mathcal G).$ 
We choose a connection  
$\mathcal H$ on 
the principal bundle  $\pi : M \rightarrow B$, with
connection form $\theta = \sum_{i=1}^{k} \theta_{i} e_{i}$.  
We introduce the connection
\begin{equation}\label{nabla-theta} 
\nabla^\theta := \nabla +\sum_{i=1}^{k}  \theta_{i} \otimes A_{i} 
\end{equation}
on the vector bundle  $\mathcal{G}$. 
The  curvature
$R^{\nabla}$ of $\nabla$ and  $R^{\theta}$ of $\nabla^{\theta}$ are related by  
\begin{equation}\label{r-theta}
R^{\nabla}= R^{\theta}  - \sum_{i=1}^{k} (d\theta_{i} ) \otimes A_{i} + 
\sum_{i=1}^{k}\theta_{i}\wedge \mathrm{ad}_{ R(X_{i}, \cdot )}  - \frac{1}{2}
\sum_{i, j} (\theta_{i}\wedge \theta_{j} )\otimes [ A_{i}, A_{j} ],
\end{equation} 
where for any form $\omega \in \Omega^{s}(M, \mathcal G)$
(in particular $\omega :=  R(X_{i}, \cdot )$) 
we define $\mathrm{ad}_{\omega}\in \Omega^{s}
(M, \mathrm{End}\, \mathcal G )$ by
$$
(\mathrm{ad}_{\omega })(Y_{1}, \cdots , Y_{s}) (r):= [\omega (Y_{1}, \cdots , Y_{s}), r]_{\mathcal G },\  
\forall Y_{i}\in {\mathfrak X}(M).
$$

\begin{lem} \label{basictohor:lem}
For any invariant section $r\in \Gamma_{\mathfrak{t}^{k}} (\mathcal G)$, 
the  $\mathcal G$-valued $1$-form $\nabla^{\theta} r$ is basic.
\end{lem}

\begin{proof}
The form $\nabla^\theta r$ is horizontal, since  
\[ \nabla^\theta_{X_i}r= \nabla^\Psi_{X_i}r=
{\mathcal L}_{\Psi (e_{i})} r=0,\  \forall 1\leq i\leq k . 
\] 
On the other hand, as $\nabla$, $\theta$ and, from
Remark \ref{comentarii} iii),   $A_{i}$ are $\mathfrak{t}^{k}$- invariant, so is 
$\nabla^{\theta}$  and 
\begin{equation} \label{aux:eq} \mathcal{L}_{\Psi (a)}(\nabla^\theta r)= 
\mathcal{L}_{\Psi (a)}(\nabla^\theta)r
+ \nabla^\theta \mathcal{L}_{\Psi (a)}r=0, \end{equation}
which implies that $\nabla^{\theta} r$
is  $\mathfrak{t}^{k}$-invariant.
\end{proof}

\begin{ass}{\rm   From now on we will assume that the partial connection $\nabla^\Psi$ has trivial holonomy.
Then  we can define a bundle 
$\mathcal{G}_B\rightarrow B$ whose  fiber over $p\in B$ is
\[ \mathcal{G}_B|_p := \Gamma_{\mathfrak{t}^k}( \mathcal{G}\vert_{\pi^{-1}(p)}),\] 
the vector space of $\nabla^\Psi$-parallel sections of $\mathcal{G}$ over the torus  $\pi^{-1}(p)$. 
We identify $\mathcal{G} = \pi^*\mathcal{G}_B$, $\Gamma_{\mathfrak{t}^k}(\mathcal{G}) = \pi^*\Gamma (\mathcal{G}_B)$ and  
\[ \Omega_b^s (M, \mathcal{G}) = \pi^*\Omega^s (B) \otimes \pi^*\Gamma (\mathcal{G}_B)\cong \Omega^s (B, \mathcal{G}_B).\]
For a basic form $\alpha \in \Omega^{s}_{b}(M,\mathcal G)$ we shall denote by $\alpha^{B}\in \Omega^{s}(B,\mathcal G_{B})$
the corresponding  form  in the above identification. (Note that by working locally in a flow box for the vertical foliation of $M\rightarrow B$, we can always assume that $\nabla^\Psi$ has trivial holonomy. We recall that a flow box is a domain $V\subset M$ such that for all $p\in \pi (V)\subset B$ the manifolds 
$\pi^{-1}(p)\cap V$ are diffeomorphic to $\mathbb{R}^k$).}
\end{ass}

\begin{lem}  i) The bundle $\mathcal{G}_B$ inherits a bracket $[\cdot , \cdot ]_{\mathcal G_{B}}$ and a scalar product
$\langle \cdot , \cdot\rangle_{\mathcal G_{B}}$ which make 
$(\mathcal{G}_B, [ \cdot , \cdot ]_{\mathcal{G}_B}, \langle \cdot , \cdot \rangle_{\mathcal{G}_B})$
into a  quadratic  Lie algebra bundle.\

ii) The connection $\nabla^{\theta}$ induces a connection $\nabla^{\theta , B}$ on $\mathcal G_{B}$,
which  preserves  $[\cdot , \cdot ]_{\mathcal G_{B}}$ and
$\langle \cdot , \cdot\rangle_{\mathcal G_{B}}$. 
The curvature $R^{\theta}$ of $\nabla^{\theta}$  is the pullback of the curvature 
$R^{\theta , B}$ of $\nabla^{\theta , B}$, i.e.  $ ( R^{\theta})^{B}=  R^{\theta , B}.$

\end{lem}

\begin{proof} i) 
The claim follows from
Lemma \ref{cond-triv-ext} i) and the definition of $\mathcal G_{B}.$\

ii) From Lemma \ref{basictohor:lem}, $\nabla^{\theta}$ induces a connection $\nabla^{\theta , B}$ on 
$\mathcal G_{B}$,
defined by
$$
\nabla^{\theta , B} r^{B} =(\nabla^{\theta} \pi^{*} r^{B})^{B},\ \forall  r^{B}\in \Gamma (\mathcal G_{B}).
$$
Since $A_{i}$ are skew-symmetric derivations, 
and $\nabla$ preserves $[\cdot , \cdot ]_{\mathcal G}$
and $\langle \cdot , \cdot \rangle_{\mathcal G}$, we obtain that $\nabla^{\theta}$ preserves these tensor fields  as well.
We deduce that  
$\nabla^{\theta , B}$  preserves $[\cdot , \cdot ]_{\mathcal G_{B}}$ and $\langle \cdot , \cdot \rangle_{\mathcal G_{B}}.$
The last statement is trivial.
\end{proof}

In order to describe the Courant algebroid $E$
together with  the action $\Psi :\mathfrak{t}^{k} \rightarrow \Gamma (E)$  in terms of structures on the base manifold $B$ of the torus bundle,  we need to  interpret equations
(\ref{comp:eq}) and (\ref{compatibility:eq})  on $B$. 
As $H$ and $R$ are invariant, they are of the form  
\begin{align}
\nonumber& H = H_{(3)} + \theta_{i}\wedge H^{i}_{(2)} +\theta_{i}\wedge \theta_{j} \wedge  H_{(1)}^{ij}+ H_{(0)}^{ijs} \theta_{i}\wedge \theta_{j}\wedge \theta_{s}\\
\label{HR:eq}& R= R_{(2)} + \theta_{i} \wedge R_{(1)}^{i}+ R_{(0)}^{ij}\theta_{i}\wedge \theta_{j}\
\end{align}
where 
$H_{(3)}$,  $H_{(2)}^{i}$, $H_{(1)}^{ij}$,  $ H_{(0)}^{ijk}$, $R_{(2)}$, 
$R_{(1)}^{i}$, 
$R_{(0)}^{ij}$  are  basic  
and for simplicity of notation we  omit  the summation signs.
In the next lemma we denote by $d^{\theta , B}$ the exterior covariant derivative defined  
by $\nabla^{\theta , B}$. Let $A_{a}^{B}$ be the section of $\mathrm{End} (\mathcal G_{B})$ 
defined by the invariant section $A_{a}\in \mathrm{End} (\mathcal G )$, where $a\in \mathfrak{t}^{k}.$

\begin{lem} \label{decomp:prop} i) The compatibility equations listed in (\ref{comp:eq}) are satisfied  if and only if
the following conditions hold:
\begin{align}
\label{dHeqs-0}&dH_{(3)}^{B} + H_{(2)}^{i, B}\wedge (d\theta_{i} )^{B} = \langle R_{(2)}^{B} \wedge R_{(2)}^{B}
\rangle_{\mathcal G_{B}}\\
\label{dHeqs-1}&  dH_{(2)}^{p, B}   + 2 H_{(1)}^{pi, B}\wedge (d\theta_{i})^{B}= -2\langle R_{(2)}^{B} \wedge R_{(1)}^{p, B}\rangle_{\mathcal G_{B}}\\
\label{dHeqs-2} & dH_{(1)}^{pq, B}  + 3 H_{(0)}^{ipq, B} (d\theta_{i})^{B} 
= 2 \langle R_{(0)}^{pq, B}, R_{(2)}^{B}\rangle_{\mathcal G_{B}} -  \langle R_{(1)}^{p, B},R_{(1)}^{q, B}\rangle_{\mathcal G_{B}}\\ 
\label{dHeqs-3}& 3 dH_{(0)}^{pqs} + 2 ( \langle R_{(0)}^{pq, B}, R_{(1)}^{s, B}\rangle_{\mathcal G_{B}}
+  \langle R_{(0)}^{sp, B}, R_{(1)}^{q, B}\rangle_{\mathcal G_{B}} 
+ \langle R_{(0)}^{qs, B}, R_{(1)}^{p, B}\rangle_{\mathcal G_{B}})  =0\\
\label{dHeqs-4} & \langle R_{(0)}^{ij, B}, R_{(0)}^{pq, B}\rangle \theta_{i}\wedge\theta_{j}\wedge \theta_{p}\wedge
\theta_{q}=0\\ 
\label{dHprime:eqs}&d^{\theta , B} R^{B}_{(2)} +R_{(1)}^{i, B}\wedge (d\theta_{i} )^{B}= 0\\
 \label{dHprime:eqs-1}& d^{\theta , B} R^{p, B}_{(1)} 
+ A_{p}^{B} \wedge R_{(2)}^{B} +  2 R_{(0)}^{pi, B} (d\theta_{i})^{B}=0\\
\label{eval-1}&  A_{p}^{B} \wedge R_{(1)}^{q, B} -  A_{q}^{B}\wedge R_{(1)}^{p, B}
= 2 \nabla^{\theta , B} R_{(0)}^{pq, B}\\
\label{eval-2}& 
A_{s}^{B} R_{(0)}^{pq, B} +  A_{q}^{B}  R_{(0)}^{sp, B} +
A_{p}^{B}R_{(0)}^{qs, B}=0\\
\label{eval:eqs-1}& R^{\theta , B} = (d\theta_{i} )^{B}\otimes A_{i}^{B}+ 
\mathrm{ad}_{R_{(2)}^{B}}\\
\label{eval:eqs-2} &
\mathrm{ad}_{R^{ij, B}_{(0)} } =\frac{1}{2} [ A^{B}_{i}, A^{B}_{j} ],
\end{align}
where
\begin{equation}\label{adjoint-repres}
\mathrm{ad} : \mathcal G_{B} \rightarrow \mathrm{Der}(\mathcal G_{B}),\ \mathrm{ad}_{u}(v)= [u, v]_{\mathcal G_{B}}
\end{equation}
is the adjoint representation in the Lie algebra 
bundle $(\mathcal G_{B}, [\cdot , \cdot]_{\mathcal G_{B}})$  and 
$1\leq p, q, s\leq k$ are arbitrary.\

ii) If the compatibility relations (\ref{comp:eq}) are satisfied, then  the third relation  (\ref{compatibility:eq})  is satisfied 
as well if and only if
\begin{equation}\label{nabla-A-t}
\nabla^{\theta  , B}_{X} A_{i}^{B}=  [R_{(1)}^{i, B}(X), r]_{\mathcal G_{B}},\ \forall X\in {\mathfrak X}(M).
\end{equation}
\end{lem}

\begin{proof} i) The equations (\ref{dHeqs-0})-(\ref{dHeqs-4})  are obtained from 
the first relation (\ref{comp:eq}), 
by comparing 
\begin{align*}
\nonumber &  dH = d H_{(3)} +  (d\theta_{i})\wedge H_{(2)}^{i} -  \theta_{i}\wedge dH_{(2)}^{i} + 2 
(d\theta_{i})\wedge \theta_{j} \wedge H_{(1)}^{ij} \\
\nonumber&+ \theta_{i}\wedge \theta_{j}\wedge dH_{(1)}^{ij} + 3 H_{(0)}^{ijs} (d\theta_{i})\wedge\theta_{j}\wedge
\theta_{s}+ ( dH_{(0)}^{ijs}) \wedge \theta_{i}\wedge \theta_{j}\wedge \theta_{s}      
\end{align*}
with 
\begin{align*}
& \langle R \wedge R\rangle =\langle R_{(2)} \wedge R_{(2)}\rangle + 2\theta_{i} \wedge \langle R_{(2)} \wedge R^{i}_{(1)}\rangle + 2 \theta_{i}\wedge \theta_{j}\wedge \langle R_{(2)},  R_{(0)}^{ij}\rangle \\
\nonumber&  - \theta_{i}\wedge \theta_{j} \wedge \langle R_{(1)}^{i}\wedge  R_{(1)}^{j}\rangle 
+ 2\theta_{i}\wedge \theta_{j}\wedge \theta_{s}\wedge  \langle R_{(0)}^{ij}  ,  R_{(1)}^{s}\rangle\\
& + \langle R_{(0)}^{ij}, R_{(0)}^{pq}\rangle \theta_{i}\wedge \theta_{j}\wedge\theta_{p}\wedge \theta_{q}, 
\end{align*}
using $d\theta_{i} \in \Omega^2(B)$, that the exterior derivative maps basic forms to basic forms,  that the operation $(\alpha , \beta ) \mapsto 
\langle \alpha \wedge \beta \rangle$ maps a pair of $\mathcal{G}$-valued basic forms to a basic scalar valued form
and then interpreting the resulting relations on $B$.  
The equations
(\ref{dHprime:eqs})-(\ref{eval-2})  are obtained from
the second relation (\ref{comp:eq}), by computing   
\begin{align}
\nonumber & 0= d^\nabla R\\
\nonumber&   =  d^\nabla R_{(2)} +  (d\theta_{i}) \wedge  R_{(1)}^{i}- \theta_{i} \wedge d^\nabla R^{i}_{(1)}
+ (\nabla R_{(0)}^{ij} )\wedge \theta_{i}\wedge \theta_{j}+ 2R_{(0)}^{ij}\otimes  (d\theta_{i})\wedge \theta_{j}\\ 
\nonumber & =d^\theta R_{(2)} -(\theta_{i}\otimes \mathrm{A}_{i})\wedge R_{(2)}   +  R_{(1)}^{i} \wedge d\theta_{i}
- \theta_{j} \wedge  \left( d^\theta R^{j}_{(1)} - ( \theta_{i}\otimes A_{i})\wedge R_{(1)}^{j}\right) \\
\label{d-nabla}& + (\nabla^{\theta} R_{(0)}^{ij}) \wedge \theta_{i}\wedge \theta_{j} - 
A_{i} ( R_{(0)}^{js})\theta_{i} \wedge \theta_{j}\wedge \theta_{s} + 2 R_{(0)}^{ij}\otimes
(d\theta_{i})\wedge
\theta_{j},\
\end{align}
identifying the horizontal and
vertical parts in the last expression of  (\ref{d-nabla})  and interpreting the result on $B$. 
The remaining   equations 
(\ref{eval:eqs-1}) and (\ref{eval:eqs-2})  are obtained by writing 
$R^{\nabla}$  in terms of $R^{\theta}$ as 
in (\ref{r-theta}) and identifying the horizontal and vertical parts in $R^{\nabla}=\mathrm{ad}_{R}$.

ii) The third relation  (\ref{compatibility:eq}) is equivalent to  
relation (\ref{nabla-A-t}), together with relation
(\ref{eval:eqs-2}).
\end{proof}

Since $\nabla^{\theta , B}$ preserves $[\cdot , \cdot ]_{\mathcal G_{B}}$, 
the endomorphism $R^{\theta , B} (X, Y)$  of $\mathcal G_{B}$ is a derivation, for any $X, Y\in {\mathfrak X}(B)$.
Recall that $A^{B}_{i}\in \mathrm{End} (\mathcal G_{B})$ is also a derivation. 
The conditions from Lemma 
\ref{decomp:prop} simplify considerably when the adjoint representation (\ref{adjoint-repres}) of the
Lie algebra bundle $(\mathcal G_{B}, [\cdot ,\cdot ]_{\mathcal G_{B}})$ is an isomorphism. 
Then 
\begin{equation}\label{def-deriv}
A_{i}^{B} =\mathrm{ad}_{r_{i}^{B}},\ R^{\theta , B}(X, Y) =\mathrm{ad}_{{\mathfrak r}^{\theta , B}(X, Y)}
\end{equation}
for $r_{i}^{B}\in \Gamma (\mathcal G_{B})$ and ${\mathfrak r}^{\theta , B}\in \Omega^{2}(B, \mathcal G_{B}).$ 
From the Bianchi identity we obtain that
${\mathfrak r}^{\theta , B}$ is $d^{\theta}$-closed.

\begin{rem} \label{quadraticLA:rem}{\rm i) Consider the class $\mathcal{C}$ of quadratic Lie algebras $(\mathfrak{g}, \langle \cdot ,\cdot \rangle )$ for which the adjoint 
representation is an isomorphism onto the Lie algebra of skew-symmetric derivations of $\mathfrak{g}$. Every semi-simple Lie algebra endowed with its Killing form (or any other invariant 
scalar product) belongs to this class. Since the center of a quadratic Lie algebra coincides with $[\mathfrak{g},\mathfrak{g}]^\perp$, there is 
no non-zero solvable Lie algebra in $\mathcal{C}$. Nevertheless, the class $\mathcal{C}$ is strictly larger than the class 
of  semi-simple quadratic Lie algebras. For instance, the affine Lie algebra $\mathfrak{so}(3) \ltimes \mathfrak{so}(3)^*\cong \mathfrak{so}(3) 
\ltimes \mathbb{R}^3$ can be endowed with the invariant scalar product of neutral signature defined by duality. The adjoint representation is 
faithful and one can easily check that all skew-symmetric derivations are inner. There exist solvable Lie algebras with faithful adjoint representation 
for which all derivations are inner \cite{stitzinger}. However these do not admit any invariant scalar product as we have already remarked.\

ii)  The adjoint representation of the Lie algebra bundle   $(\mathcal G_{B}, [\cdot ,\cdot ]_{\mathcal G_{B}})$ is an isomorphism
if and only if the same is true for the Lie algebra bundle $(\mathcal G , [\cdot ,\cdot ]_{\mathcal G})$ of the Courant
algebroid $E$. Courant algebroids with this property will be described in Proposition \ref{description-heterotic-CA}.
} \end{rem}

\begin{cor}\label{data-extended} 
Let $\pi : M \rightarrow B$ be a principal $T^{k}$-bundle and  $\mathcal H$  a principal connection on $\pi$, 
with connection form $\theta  = \sum_{i=1}^{k} \theta_{i} e_{i}\in \Omega^{1}(M, \mathfrak{t}^{k} )$,
where $(e_{i})$ is a basis of $\mathfrak{t}^{k}.$  
There is a one to one correspondence between

\begin{enumerate}
\item  standard Courant algebroids $E= T^{*}M \oplus \mathcal G \oplus TM$ for which the adjoint action 
of the Lie algebra bundle $(\mathcal G , [\cdot , \cdot ]_{\mathcal G})$ is an isomorphism,   
together with an   action $\Psi : \mathfrak{t}^{k} \rightarrow \mathrm{Der}(E)$
which lifts the vertical parallelism of $\pi$,
preserves the factors $T^{*}M$, $\mathcal G$ and $TM$  of $E$, 
and for which the flat partial 
connection $\nabla^\Psi$ has trivial holonomy 

and

\item  quadratic Lie algebra bundles
$(\mathcal G_{B}, \langle \cdot , \cdot \rangle_{\mathcal G_{B}}, [\cdot , \cdot ]_{\mathcal G_{B}})$
over $B$, whose adjoint action is an isomorphism,  together with 
a connection $\nabla^{B}$ on  the vector bundle $\mathcal G_{B}$ which preserves
$\langle \cdot , \cdot \rangle_{\mathcal G_{B}}$ and  $[\cdot , \cdot ]_{\mathcal G_{B}}$, 
sections $r_{i}^{B}\in \Gamma (\mathcal G_{B})$ ($1\leq i\leq k$),  a $3$-form 
$H_{(3)}^{B}\in \Omega^{3}(B)$, $2$-forms $H_{(2)}^{i, B}\in \Omega^{2}(B)$, 
$1$-forms $H_{(1)}^{ij, B}\in \Omega^{1}(B)$
and constants $c_{ijp}\in \mathbb{R}$ ($1\leq i, j, p\leq k$)  
such that 
\begin{align}
\nonumber  & dH_{(3)}^{B}   = \langle{\mathfrak   r}^{B}\wedge {\mathfrak r}^{B}\rangle_{\mathcal G_{B}}\\
\nonumber& - 
  ( H_{(2)}^{i, B} + 2 \langle{\mathfrak  r}^{B} , r_{i}^{B}\rangle_{\mathcal G_{B}} -
\langle r_{i}^{B}, r_{j}^{B}\rangle_{\mathcal G_{B}}(d\theta_{j})^{B}) \wedge (d\theta_{i} )^{B}, \\ 
\nonumber&  dH_{(2)}^{p, B} = 2( \langle \nabla^{B} r_{p}^{B}, r_{i}^{B}\rangle_{\mathcal G_{B}}
- H_{(1)}^{pi, B}) \wedge  (d\theta_{i})^{B} - 2\langle {\mathfrak  r}^{ B} \wedge \nabla^{B} r_{p}^{B}\rangle_{\mathcal G_{B}} \\
\label{reduced-rel} &  dH_{(1)}^{pq, B} = - 3 c_{ipq} (d\theta_{i})^{B}  + \langle{\mathfrak  r}^{B}, [r_{p}^{B}, r_{q}^{B}
]_{\mathcal G_{B}}\rangle_{\mathcal G_{B}} -\langle \nabla^{B} r_{p}^{B}\wedge \nabla^{B} r_{q}^{B}\rangle_{\mathcal G_{B}},
\end{align}
where ${\mathfrak r}^{B}\in \Omega^{2}(B, \mathcal G_{B})$ is related to the curvature $R^{B}$ of the connection
$\nabla^{B}$ by $R^{B}(X, Y) = \mathrm{ad}_{{\mathfrak r}^{B} (X, Y)}$ for any $X, Y\in {\mathfrak X}(B).$ 
\end{enumerate}
\end{cor}

\begin{proof} The claim follows from 
Lemma \ref{decomp:prop},  by letting $\nabla^{B}:= \nabla^{\theta , B}$  and 
simplifying the relations from this lemma, using in an essential way  that the adjoint
representation of the Lie algebra bundle   $(\mathcal G , [\cdot , \cdot ]_{\mathcal G})$ is an isomorphism.
More precisely,
relations (\ref{eval:eqs-2}), 
(\ref{nabla-A-t})  and 
(\ref{eval:eqs-1})
determine $R_{(0)}^{ij, B}$,  $R_{(1)}^{i, B}$ and $R_{(2)}^{B}$ respectively  by
\begin{equation}\label{expr-R}
R_{(0)}^{ij, B} = \frac{1}{2}[ r_{i}^{B}, r_{j}^{B} ]_{\mathcal G_{B}},\ R_{(1)}^{i, B} =\nabla^{B} r_{i}^{B},\
R_{(2)}^{B} = {\mathfrak r}^{B} - (d\theta_{i})^B\otimes r_{i}^{B}.   
\end{equation}
Relation  (\ref{dHeqs-3}) with $R_{(0)}^{ij, B}$ and $R_{(1)}^{i, B}$ given 
by (\ref{expr-R}) implies that 
\begin{equation}
H_{(0)}^{pqs, B} =-\frac{1}{3} \langle [r_{p}^{B}, r_{q}^{B}]_{\mathcal G_{B}} , r_{s}^{B}\rangle_{\mathcal G_{B}}
+ c_{pqs}
\end{equation}
 for some constants  $c_{pqs}$.   Written in terms of ${\mathfrak r}^{B}$ rather than $R_{(2)}^{B}$, relations
(\ref{dHeqs-0}), (\ref{dHeqs-1}), (\ref{dHeqs-2}) become
relations (\ref{reduced-rel}). The remaining relations from  Lemma \ref{decomp:prop}, with
$R_{(0)}^{ij, B}, R_{(1)}^{i, B}, R_{(2)}^{B}$ and $H_{(0)}^{pqs, B}$ as above and 
$A_{i} =\mathrm{ad}_{r_{i}^{B}}$ are satisfied. 
\end{proof}

\begin{exa}\label{particular-class}{\rm Under the assumptions  of Corollary
\ref{data-extended}, let  $(\mathcal G_{B}, \langle \cdot , \cdot \rangle_{\mathcal G_{B}}, [\cdot , \cdot ]_{\mathcal G_{B}})$ be a quadratic Lie algebra bundle 
over $B$, whose adjoint action is an isomorphism,  together with 
a connection $\nabla^{B}$ on  the vector bundle $\mathcal G_{B}$ which preserves
$\langle \cdot , \cdot \rangle_{\mathcal G_{B}}$ and  $[\cdot , \cdot ]_{\mathcal G_{B}}$. 
Choose arbitrary sections $r_{i}^{B}\in \Gamma (\mathcal G_{B})$ ($1\leq i\leq k$)
and define, for any $i, j, s$, $c_{ijp} :=0$, 
\begin{equation}\label{h-0-ijs}
H_{(0)}^{ijs, B} :=- \frac{1}{3} \langle [r_{i}^{B}, r_{j}^{B}]_{\mathcal G_{B}} , r_{s}^{B}\rangle_{\mathcal G_{B}} 
\end{equation}
and 
\begin{equation}\label{h-1-ij}
H_{(1)}^{ij, B} := \frac{1}{2} \left( \langle \nabla^{B }r_{i}^{B}, r_{j}^{B}\rangle_{\mathcal G_{B}}
-  \langle \nabla^{B }r_{j}^{B}, r_{i}^{B}\rangle_{\mathcal G_{B}}\right) .
\end{equation}
With these choices,  the third relation  (\ref{reduced-rel}) is satisfied. For any  forms $H_{(3)}^{B}$ and $H_{(2)}^{i, B}$,  such that 
\begin{equation}\label{k-i}
{\mathcal K}_{i}:= H_{(2)}^{i, B} + 2\langle \mathfrak{r}^{B}, r_{i}^{B}\rangle_{\mathcal G_{B}} -
 \langle r_{i}^{B}, r_{j}^{B}\rangle_{\mathcal G_{B}} (d\theta_{j})^{B}
\end{equation}
is closed and 
\begin{equation}
d H_{(3)}^{B} = \langle \mathfrak{r}^{B}\wedge \mathfrak{r}^{B}\rangle_{\mathcal G_{B}}  -{\mathcal K}_{i}\wedge 
(d\theta_{i})^{B}
\end{equation}
the relations (\ref{reduced-rel}) are satisfied and we thus obtain a standard Courant algebroid together with an action $\Psi : \mathfrak{t}^k\rightarrow \mathrm{Der}(E)$
lifting the vertical parallelism of the principal torus bundle $\pi : M \rightarrow B$. Note that 
$2$-forms $H_{(2)}^{i, B}$ as required in the above construction do always exist and are unique up to addition of  closed forms whereas $H_{(3)}^{B}$ exists if and only if the closed form $\langle \mathfrak{r}^{B}\wedge \mathfrak{r}^{B}\rangle_{\mathcal G_{B}}  - {\mathcal K_{i}}\wedge 
(d\theta_{i})^{B}$ is exact. It is also unique up to addition of a closed form.
}
\end{exa}

\subsection{Invariant spinors}

Let  $E$ be a transitive Courant algebroid over an oriented  manifold $M$ and 
$\Psi : \mathfrak{g}\rightarrow \mathrm{Der}(E)$   an action on  $E$, which lifts an action 
$\psi : \mathfrak{g}\mapsto {\mathfrak X}(M),\  a\mapsto X_{a}$ of  $\mathfrak{g}$  on $M$. 
Let $\mathbb{S}$ be a   canonical weighted spinor bundle of $E.$ 
 Our aim in this section is to define an action of $\mathfrak{g}$ on $\Gamma (\mathbb{S}).$  
In order to find a proper definition  we assume that 
$\Psi$ integrates to a Lie group action  
$$
G\rightarrow \mathrm{Aut}(E),\  g\mapsto I^{g}_{E} : E\rightarrow E 
$$ 
such that $I^{g}_{E}$ induces a globally defined  isomorphism  $I^{g}_{\mathbb{S}} : \mathbb{S}\rightarrow \mathbb{S}$,
for any $g\in \mathfrak{g}.$ 
Recall that 
$$
I^{g}_{\mathbb{S}}\circ \gamma_{u}= \gamma_{I_{E}^{g} (u)}\circ I_{\mathbb{S}}^{g},\  \forall g\in G,\ u\in E.
$$
Consider a curve $g = g(t)$ of $G$ with $g(0) =e$ and $\dot{g}(0) = a$. 
We choose $I^{g(t)}_{\mathbb{S}}$ depending smoothly on $t$ and such that $I^{g(0)}_{\mathbb{S}}=\mathrm{Id}_{\mathbb{S}}$.
Replacing in the above relation
$g$ by $g(t)$ and taking the derivative at $t=0$ we obtain that
$\Psi^{\mathbb{S}}(a):= \left.\frac{d}{dt}\right\vert_{t=0} I_{\mathbb{S}}^{g(t)}\in \mathrm{End}\,  \Gamma (\mathbb{S})$ satisfies 
\begin{equation}\label{cond-op-spinors}
\Psi^{\mathbb{S}} (a) \gamma_{u}s= \gamma_{\Psi (a) ( u)} s + \gamma_{u}\Psi^{\mathbb{S}}(a)s,\ \forall u\in \Gamma (E),\ s\in \Gamma (\mathbb{S}),\ a\in \mathfrak{g}.
\end{equation}   
In  the following we do not assume that $\Psi$ integrates to an action of $G$.

\begin{prop}\label{action-spinors} i)  There is a unique linear map  
$$
\Psi^{\mathbb{S}} : \mathfrak{g} \rightarrow\mathrm{End}\, \Gamma (\mathbb{S})
$$
which   satisfies   
relation (\ref{cond-op-spinors}), the Leibniz rule
\begin{equation}\label{lie-spinors-2}
\Psi^{\mathbb{S}} (a) ( fs) = f  \Psi^{\mathbb{S}} (a) (  s)
+ X_{a}(f)  s,\ \forall  f\in C^{\infty}(M),\ s\in \Gamma (\mathbb{S}),\ a\in \mathfrak{g},
\end{equation}
and,  for any $U\subset M$ open and sufficiently small,   preserves the canonical $\mathrm{det}\, (T^{*}U)$-valued bilinear pairing $\langle \cdot , \cdot \rangle_{\mathbb{S}\vert_{U}}$
of $\Gamma (\mathbb{S}\vert_{U})$, i.e.
\begin{equation}\label{lie-pairing}
{\mathcal L}_{X_{a}} \langle  s,  \tilde{s} \rangle_{\mathbb{S}\vert_{U}} =
\langle \Psi^{\mathbb{S}} (a) s,  \tilde{s}\rangle_{\mathbb{S}\vert_{U}}
+ \langle s,  \Psi^{\mathbb{S}} (a) \tilde{s}\rangle_{\mathbb{S}\vert_{U}},\ \forall s, \tilde{s}\in \Gamma (\mathbb{S}\vert_{U}),\
a\in \mathfrak{g}.
\end{equation}
ii) The map $\Psi^{\mathbb{S}}:\mathfrak{g}\rightarrow \mathrm{End}\,  \Gamma (\mathbb{S})$ satisfies 
\begin{equation}\label{lie-bracket}
[\Psi^{\mathbb{S}} (a), \Psi^{\mathbb{S}} (b) ] = \Psi^{\mathbb{S}} [ a, b],\  \forall a, b\in \mathfrak{g}.
\end{equation}  
It is called the  {\cmssl action on spinors} induced by $\Psi .$ 
 \end{prop}

The remaining part of this section is devoted to the proof of Proposition~\ref{action-spinors}.  For uniqueness, let $\Psi^{\mathbb{S}}$ and $\tilde{\Psi}^{\mathbb{S}}$  be two maps which satisfy the required
conditions.  Then $F(a):= \Psi^{\mathbb{S}} (a) -\tilde{\Psi}^{\mathbb{S}}(a)$
is $C^{\infty}(M)$-linear and commutes with the Clifford action. Hence $F(a) = \lambda (a)\mathrm{Id}_{\mathbb{S}}$,  for
$\lambda (a)\in  C^{\infty}(M).$  Since
$F(a)$ is skew-symmetric with respect to $\langle \cdot , \cdot \rangle_{\mathbb{S}\vert_{U}}$, we obtain  $\lambda (a) =0.$
The  uniqueness follows. 
For existence we need the following two lemmas. 

\begin{lem}\label{de-adaugat}  
Let $E_{i}$  ($i=1,2$) be two transitive Courant algebroids over  $M$ with  actions
$\Psi_{i} : \mathfrak{g} \rightarrow \mathrm{Der} (E_{i})$, which lift
$\psi .$ 
Let   $I_{E}: E_{1} \rightarrow E_{2}$ be an invariant isomorphism and, for any
$U\subset M$ open and sufficiently small,   
$I_{\mathbb{S}\vert_{U}} : \mathbb{S}_{1} \vert_{U}\rightarrow \mathbb{S}_{2}\vert_{U}$  the induced 
isomorphism between  canonical  weighted spinor bundles of $E_{i}\vert_{U}$. Let  $\Psi^{\mathbb{S}_{i}}: \mathfrak{g}\rightarrow \mathrm{End}\, \Gamma  (\mathbb{S}_{i}\vert_{U} )$   ($i=1,2$) be  two maps related by  
\begin{equation}\label{inv-iso-sp}
\Psi^{\mathbb{S}_{2}}(a)= I_{\mathbb{S}\vert_{U}} \circ \Psi^{\mathbb{S}_{1}}(a)\circ ( I_{\mathbb{S}\vert_{U}} )^{-1}  ,\ \forall
a\in \mathfrak{g}.
\end{equation}
Then $\Psi^{\mathbb{S}_{1}} $ satisfies the conditions from Proposition \ref{action-spinors} if and only if
$\Psi^{\mathbb{S}_{2}}$ does.  
\end{lem}

\begin{proof}  Let  $\Psi^{\mathbb{S}_{1}}:\mathfrak{g}\rightarrow 
\mathrm{End}\, \Gamma  (\mathbb{S}_{1}\vert_{U} )$  be a map which  satisfies the conditions from Proposition \ref{action-spinors}
and $\Psi^{\mathbb{S}_{2}}: \mathfrak{g}\rightarrow \mathrm{End}\, (\mathbb{S}_{2}\vert_{U} )$ be defined by  (\ref{inv-iso-sp}).
The map   $\Psi^{\mathbb{S}_{2}} $  obviously  satisfies (\ref{lie-spinors-2})  and (\ref{lie-bracket})
and,  from  (\ref{pairing-iso-relation}),  
it satisfies  
(\ref{lie-pairing}) as well. 
Using   $I_{\mathbb{S}\vert_{U} }\circ \gamma_{u} = \gamma_{I_{E}(u)}\circ I_{\mathbb{S}\vert_{U} }$ and 
relation (\ref{cond-op-spinors}) satisfied by $\Psi^{\mathbb{S}_{1}}$, we obtain
\begin{equation}
\Psi^{\mathbb{S}_{2}}(a) \gamma_{u} (s)= \gamma_{I_{E} \Psi_{1} (a) I_{E}^{-1}(u)}s +\gamma_{u} \Psi^{\mathbb{S}_{2}}(a)s,\ \forall a\in \mathfrak{g},\  u\in \Gamma (E_{1}\vert_{U}),\ s\in \Gamma (\mathbb{S}_{2}\vert_{U}).
\end{equation}
Since 
$I_{E}$ is invariant, 
$\Psi_{2}(a) = I_{E}\circ \Psi_{1}(a) \circ I_{E}^{-1}$ and we obtain that $\Psi^{\mathbb{S}_{2}}$ satisfies
 (\ref{cond-op-spinors}).
\end{proof}

Let 
 $E_{M}= T^{*}M\oplus \mathcal G\oplus TM$ be a standard Courant algebroid  
defined by a quadratic Lie algebra bundle $(\mathcal G , [\cdot , \cdot ]_{\mathcal G}, \langle\cdot , \cdot\rangle_{\mathcal G})$
and data $(\nabla , R, H)$, with action  
\begin{equation}\label{action-Psi-new}
\Psi : \mathfrak{g} \mapsto \mathrm{Der} (E_{M}),\ \Psi (a) (\xi + r + X):= {\mathcal L}_{X_{a}} \xi  + \nabla^{\Psi}_{X_{a}} r
+ {\mathcal L}_{X_{a}} X
\end{equation}
which  lifts an action 
$$
\psi : \mathfrak{g}\mapsto {\mathfrak X}(M),\  a\mapsto X_{a}
$$
of  $\mathfrak{g}$  on $M$.  
Let $S_{\mathcal G}$ be an irreducible  $\mathrm{Cl} (\mathcal G)$-bundle, $\mathcal S_{\mathcal G}
= S_{\mathcal G}\otimes | \mathrm{det}\, S^{*}|^{1/r}$ the canonical spinor bundle
and 
$\mathbb{S}_{M} := \Lambda (T^{*}M)\hat{ \otimes }\mathcal S_{\mathcal G}$ 
the  canonical weighted spinor bundle of $E$ determined by $S_{\mathcal G}$.

\begin{lem}\label{standard-spinors}
The map 
\begin{equation}\label{lie-spinors}
\Psi^{\mathbb{S}_{M}} : \mathfrak{g} \rightarrow\mathrm{End}\, \Gamma (\mathbb{S}_{M}),\ \Psi^{\mathbb{S}_{M}}(a) (\omega \otimes s):=
({\mathcal L}_{X_{a}} \omega) \otimes s + \omega \otimes {\nabla}^{ \Psi , \mathcal{S}_{\mathcal G}}_{X_{a}} s ,
\end{equation}
for any $a\in \mathfrak{g}$,  $\omega \in \Omega (M)$ and $s\in \Gamma (\mathcal{S}_{\mathcal G})$,
satisfies the conditions from Proposition
\ref{action-spinors}. Above 
${\nabla}^{ \Psi , \mathcal{S}_{\mathcal G}}$ is the partial connection on $\mathcal S_{\mathcal G}$ induced by 
any
partial connection ${\nabla}^{ \Psi ,  S_{\mathcal G}}$ on $S_{\mathcal G}$, compatible with 
the partial connection ${\nabla}^{\Psi }$.
\end{lem}

\begin{proof} Relation (\ref{lie-spinors-2}) is obviously satisfied. To prove  relation (\ref{lie-pairing})  
we recall that  $\langle \cdot , \cdot \rangle_{{\mathbb  S}\vert_{U}}$ is given by
(\ref{can-pairing}), where  $\langle \cdot , \cdot \rangle_{\mathcal S_{\mathcal G}\vert_{U}}$ 
is a canonical bilinear pairing of 
$\Gamma ( \mathcal S_{\mathcal G}\vert_{U})$. 
Relation (\ref{lie-pairing} 
follows from
a computation which uses 
(\ref{can-pairing}), 
\begin{equation}\label{pas2-conseq}
X_{a} \langle s, \tilde{s}\rangle_{\mathcal S_{\mathcal G}\vert_{U}} = \langle  {\nabla}^{ \Psi , \mathcal S_{\mathcal G}}_{X_{a}} s, \tilde{s}\rangle_{\mathcal S_{\mathcal G}\vert_{U}}
+ \langle s,  {\nabla}^{ \Psi , \mathcal S_{\mathcal G}} _{X_{a}} \tilde{s}\rangle_{\mathcal S_{\mathcal G}\vert_{U}}
\end{equation}
and the fact that ${\nabla}^{ \Psi , \mathcal S_{\mathcal{G}}}$ preserves the degree of sections of $\mathcal S_{\mathcal G}.$ 
(Relation (\ref{pas2-conseq}) follows from the fact that $\nabla^{\Psi }$ preserves 
$\langle \cdot , \cdot\rangle_{\mathcal G}$, which is of neutral signature, and $\nabla^{ \Psi , \mathcal S_{\mathcal G}}$
is induced by any partial connection $\nabla^{ \Psi , S_{\mathcal G}}$ on $S_{\mathcal G}$ compatible with $\nabla^{ \Psi }$.  The argument
is similar to the one used in the proof of Lemma \ref{pas2}).  
In order to prove 
(\ref{cond-op-spinors}), decompose $u= \xi + r + X$. Then
$$
\Psi (a)(u) = {\mathcal L}_{X_{a}} (\xi + X ) +{\nabla}^{\Psi}_{X_{a}} r,
$$
from where we deduce that
\begin{align}
\nonumber \gamma_{\Psi (a)(u)} (\omega \otimes s)& = \gamma_{{\mathcal L}_{X_{a}} (\xi+ X )}(\omega \otimes s)
+\gamma_{{\nabla}^{ \Psi}_{X_{a}} r} (\omega \otimes s)\\
\label{r1}&= (i_{{\mathcal L}_{X_{a}}X} \omega +({\mathcal L}_{X_{a}}\xi )\wedge \omega )\otimes s
+(-1)^{| \omega |}\omega \otimes ({\nabla}^{ \Psi}_{X_{a}}r) s.
\end{align} 
Similar computations show that
\begin{align}
\label{r2} \Psi^{\mathbb{S}_{M}} (a) \gamma_{u} (\omega\otimes s)&= {\mathcal L}_{X_{a}} (i_{X}\omega + \xi \wedge \omega )
\otimes s + (i_{X}\omega + \xi \wedge \omega )\otimes {\nabla}^{ \Psi , \mathcal S_{\mathcal G}}_{X_{a}} s\\
\nonumber &+ (-1)^{| \omega |} ( {\mathcal L}_{X_{a}} \omega \otimes (r s) +\omega \otimes 
{\nabla}^{ \Psi , \mathcal S_{\mathcal G}}_{X_{a}} (rs))\\
\label{r3}\gamma_{u} \Psi^{\mathbb{S}_{M}} (a) (\omega \otimes s) & = (i_{X} {\mathcal L}_{X_{a}} \omega +\xi \wedge {\mathcal L}_{X_{a}}\omega )
\otimes s + (-1)^{| \omega |} (\mathcal L_{X_{a}}\omega ) \otimes (rs) \\
\nonumber & + (i_{X}\omega +\xi \wedge \omega )\otimes {\nabla}_{X_{a}}^{ \Psi , \mathcal S_{\mathcal G}}s + (-1)^{ | \omega | } \omega
\otimes (r {\nabla}_{X_{a}}^{{\Psi}, \mathcal S_{\mathcal G}}s).
\end{align}
Combining relations (\ref{r1}), (\ref{r2}) and (\ref{r3}) and using that $\nabla^{\Psi , \mathcal S_{\mathcal G}}$
is compatible with $\nabla^{\Psi}$ we obtain (\ref{cond-op-spinors}).
Relation (\ref{lie-bracket}) follows from the definition of the map $\Psi^{\mathbb{S}_{M}}$ and the flatness of 
 ${\nabla }^{\Psi, \mathcal{S}_\mathcal{G}}$  (which is a consequence of the flatness of 
${\nabla }^{\Psi}$).
\end{proof}

We  conclude the proof of Proposition
\ref{action-spinors} by choosing an invariant dissection $I : E \rightarrow E_{M}$ and
 isomorphisms $I_{\mathbb{S}\vert_{U_{i}}} :
\mathbb{S}\vert_{U_{i}} \rightarrow \mathbb{S}_{M}\vert_{U_{i}}$ compatible with $I\vert_{U_{i}}$, where 
$\mathcal U = \{ U_{i}\}$ is a cover of $M$ with sufficiently small open subsets.
Using Lemmas
\ref{de-adaugat} and \ref{standard-spinors}, 
we obtain that the map $ \Psi^{\mathbb{S}}:\mathfrak{g} \rightarrow \mathrm{End}\, \Gamma (\mathbb{S})$
defined by
$$
\Psi^{\mathbb{S}} (a) (s)\vert_{U_{i}} := 
(I_{\mathbb{S}\vert_{U_{i}}})^{-1}  \circ \Psi^{\mathbb{S}_{M}}(a)\circ  I_{\mathbb{S}\vert_{U_{i}}} (s\vert_{U_{i}}),\
\forall s\in \Gamma (\mathbb{S})
 $$
satisfies the conditions from Proposition \ref{action-spinors}.

\begin{defn}\label{definition-inv-spinors} A section of the canonical weighted  spinor bundle 
$\mathbb{S}$   is an invariant spinor if it is annihilated by the operators
$\Psi^{\mathbb{S}} (a)$, for all $a\in \mathfrak{g}.$
\end{defn}

\begin{notation}{\rm Given an action  
$\Psi : \mathfrak{g}\rightarrow \mathrm{Der}  (E)$
on a transitive Courant algebroid
$E$, we shall  denote by $\Gamma_{\mathfrak{g}}(\mathbb{S})$ the vector space of  invariant spinors.
Similarly, $\Gamma_{\mathfrak{g}} (E)$ will denote the vector space of invariant sections of $E$.}
\end{notation}

\begin{lem}   In the setting of Proposition \ref{action-spinors}, 
\begin{equation}\label{invariance-can-dirac}
\slashed{d} \circ \Psi^{\mathbb{S}} (a) = \Psi^{\mathbb{S}} (a) \circ \slashed{d},\ \forall a\in \mathfrak{g},
\end{equation}
where $\slashed{d}$ is the canonical Dirac generating operator of $E$.
\end{lem}

\begin{proof}  From Proposition \ref{iso-dirac} and Lemma \ref{de-adaugat}, 
it is sufficient to prove the statement for the Courant algebroid $E_{M}$ considered
in  Lemma \ref{standard-spinors} with  action $\Psi^{\mathbb{S}_{M}}$ defined
by (\ref{lie-spinors}).  We need to show that for any  $a\in \mathfrak{g}$, $\omega \in \Omega (M)$ and $s\in \Gamma (\mathcal S_{\mathcal G})$  
\begin{equation}\label{invariance-can-dirac-1}
\slashed{d}_{M}  \Psi^{\mathbb{S}_{M}} (a) (\omega \otimes s)=
 \Psi^{\mathbb{S}_{M}} (a) \slashed{d}_{M}(\omega \otimes s)
\end{equation}
where $\slashed{d}_{M}\in \mathrm{End}\, \Gamma (\mathbb{S}_{M})$ is the Dirac generating operator of $E_{M}.$ 
We consider an invariant local frame $( X_{i} )$ of $TM$. 
Since  ${\nabla}^{\Psi}$ is flat   we may (and will) take the local frame $( r_{k})$ of 
$\mathcal G$ to be ${\nabla}^{\Psi}$-parallel. Since ${\nabla}^{\Psi}$ preserves the scalar product $\langle \cdot , \cdot \rangle_{\mathcal G}$, 
the $\langle \cdot , \cdot\rangle_{\mathcal G}$- dual frame $(\tilde{r}_{k})$ 
is also ${\nabla}^{\Psi}$-parallel. Since ${\nabla}^{\Psi}$ preserves the Lie bracket $[\cdot , \cdot ]_{\mathcal G}$,  
the Cartan form $C_{\mathcal G}$ is ${\nabla}^{\Psi}$-parallel.\

Since $R$, $X_{i}$ and $r_{k}$ are invariant, 
\begin{equation}\label{last-added}
{\mathcal L}_{ X_{a}} \langle R(X_{i}, X_{j}), r_{k}\rangle_{\mathcal G} = 0,\  \forall a\in \mathfrak{g}.
\end{equation}
From (\ref{last-added}),
$\nabla^{\Psi} C_{\mathcal G }=0$, 
the fact that $\nabla^{\Psi , \mathcal S_{\mathcal G}}$ is compatible with $\nabla^{\Psi}$
and the expressions (\ref{can-dirac}),   (\ref{lie-spinors})  for $\slashed{d}$
and $\Psi^{\mathbb{S}}$,  we see that
relation 
(\ref{invariance-can-dirac})  reduces to 
\begin{equation}\label{f1}
{\nabla}_{X_{a}}^{\Psi , {\mathcal S}_{\mathcal G}} \nabla^{{\mathcal S}_{\mathcal G}}_{X_{i}} s=
 \nabla_{X_{i}}^{{\mathcal S}_{\mathcal G}} {\nabla}_{X_{a}}^{\Psi , {\mathcal S}_{\mathcal G}}s,\ 
\forall a\in \mathfrak{g},\  s\in \Gamma (\mathcal{S}_{\mathcal G}),
\end{equation}
where, we recall, $\nabla^{\mathcal S_{\mathcal G}}$ is the connection on $\mathcal S_{\mathcal G}$ induced by
any connection on $S_{\mathcal G}$ compatible with $\nabla $ and similarly
for the partial connections $\nabla^{\Psi , \mathcal S_{\mathcal G}}$ and $\nabla^{\Psi}.$  
For any $a\in \mathfrak{g}$, let $A_{a} := \nabla^{\Psi}_{X_{a}}-\nabla_{X_{a}}.$  
Then 
\begin{equation}\label{rel-conn-spinors}
{\nabla}^{\Psi , {\mathcal S}_{\mathcal G}}_{X_{a}} s = \nabla_{X_{a}}^{{\mathcal S}_{\mathcal G}}s -\frac{1}{2}
A_{a} \cdot s 
\end{equation}
where  $A_{a}\cdot s $ denotes the Clifford action of $A_{a} \in \Gamma (\Lambda^{2}\mathcal G)
\subset \Gamma\, \mathrm{Cl}(\mathcal G)$ on $s\in \Gamma (\mathcal{S}_{\mathcal G})$
(see e.g.\ Proposition 53 of \cite{cortes-david-MMJ} for more details).  
 From (\ref{rel-conn-spinors}), (\ref{nabla-A}) and ${\mathcal L}_{X_{a}}X_{i}=0$ we deduce that (\ref{f1}) is equivalent to
\begin{equation}\label{r-ad}
R^{\nabla^{{\mathcal S}_{\mathcal G}}}(X_{a}, X_{i}) s +\frac{1}{2} (\mathrm{ad}_{R(X_{a}, X_{i})}) s=0,
\end{equation} 
where 
$ (\mathrm{ad}_{R(X_{a}, X_{i})}) s$ means the Clifford action of $ \mathrm{ad}_{R(X_{a}, X_{i})}:
=[ \mathrm{ad}_{R(X_{a}, X_{i})}, \cdot ]_{\mathcal G}\in \Gamma (\Lambda^{2}\mathcal G)
\subset \Gamma\, \mathrm{Cl}(\mathcal G)$ on $s$.
In order to prove (\ref{r-ad}) we remark 
first that both endomorphisms $R^{\nabla^{{\mathcal S}_{\mathcal G}}}(X_{a}, X_{i})$ 
and  $( \mathrm{ad}_{R(X_{a}, X_{i})})$ of $\mathcal{S}_{\mathcal G}$ 
are trace free (the statement for $R^{\nabla^{{\mathcal S}_{\mathcal G}}}(X_{a}, X_{i})$ is a consequence of the fact that 
$\nabla^{{\mathcal S}_{\mathcal G}}$ is induced by a connection $\nabla^{S_{\mathcal G}}$ on $S_{\mathcal G}$). 
On the other hand, since $\nabla^{{\mathcal S}_{\mathcal G}}$ is compatible with $\nabla$, we obtain that 
$T:= R^{\nabla^{{\mathcal S}_{\mathcal G}}}(X_{a}, X_{i})\in \mathrm{End}({\mathcal  S}_{\mathcal G})$
satisfies 
\begin{equation}\label{rel-A}
T(rs) = (R^{\nabla}(X_{a}, X_{i})r) s + r T(s),\forall r\in \mathcal G ,\ s\in \mathcal{S}_{\mathcal G} .
\end{equation}
The same relation is satisfied by $T:= -\frac{1}{2} \mathrm{ad}_{ R(X_{a}, X_{i})}$ acting by Clifford multiplication 
(here we use that $R^{\nabla} (X_{a}, X_{i}) (r)= \mathrm{ad}_{ R(X_{a}, X_{i})} (r)$ and  relation
$\omega (r) = -\frac{1}{2} [ \omega, r]_{\mathrm{Cl}}$,  for any $\omega \in \Lambda^{2}\mathcal G\subset\mathrm{Cl} (\mathcal G)$, where 
$ [ \omega, r]_{\mathrm{Cl}}=\omega r-r \omega $ denotes the commutator of $\omega$ and $r$ in the Clifford algebra 
and $\omega (r)$ the action of $\omega \in \Lambda^{2}\mathcal{G}\cong \mathfrak{so}(\mathcal{G})$ on $r\in\mathcal{G}$).\

To summarize: both $R^{\nabla^{{\mathcal S}_{\mathcal G}}}(X_{a}, X_{i})$ 
and  $-\frac{1}{2}( \mathrm{ad}_{R(X_{a}, X_{i})})$ are trace-free and satisfy (\ref{rel-A}). 
Since $\langle \cdot , \cdot \rangle_{\mathcal G}$ has neutral signature, they coincide.  
 \end{proof}

\subsection{Pullback actions and spinors}\label{sect-pull-back-actions}

 Let $f: M \rightarrow N$ be a submersion  and 
\begin{align*}
\nonumber& \psi^{M}: \mathfrak{g} \rightarrow {\mathfrak X}(M),\  a\mapsto X_{a}^{M}\\
\nonumber& \psi^{N}: \mathfrak{g} \rightarrow {\mathfrak X}(N),\  a\rightarrow X_{a}^{N}\end{align*}
be $f$-related infinitesimal actions, i.e.\ $X_a^N\circ f= dfX_{a}^M$ for all $a\in \mathfrak{g}$.   
Let $E$ be a transitive Courant algebroid over $N$ with anchor $\pi : E \rightarrow TN$ and 
$$
\mathfrak{g} \ni a \mapsto \Psi  (a) \in\mathrm{Der} (E)
$$ 
be an action on $E$ which lifts
$\psi^{N}.$  
Recall that the pullback Courant algebroid $f^{!}E$ is the quotient bundle  $C/C^{\perp}$  over $M$ (identified
with the graph $M_{f}$ of $f$), where,  for any
$p\in M$, 
\begin{align*}
\nonumber& C_{p}:= \{ (u,   \mu + X)\in E_{f(p)}\times\mathbb{ T}_{p}M:\ \pi (u)= (d_{p}f)(X)\}\\
\nonumber& C^{\perp}_{p} := \{ (\frac{1}{2} \pi^{*}(\gamma ), - f^{*}\gamma ),\ \gamma \in T_{f(p)}N\}
\subset C_{p}  
\end{align*}
with the Courant algebroid structure defined  at the beginning of Section~\ref{pull-back-subsection}.
For   $U\subset N$ open,   a  section of $C\vert_{f^{-1} (U)}$ (and the induced section  
of $(f^{!}E)\vert_{f^{-1} (U)}$)  
of the form $(f^{*}u, \mu + X)$ where   $u\in \Gamma (E\vert_{U})$, $X\in {\mathfrak X}( f^{-1} (U))$ is $f$-projectable with  
$f_{*} X = \pi  (u)$ and   $\mu\in \Omega^{1}( f^{-1} (U))$, will be called {\cmssl distinguished}.  
Let $\mathcal U = \{ U_{i}\}$ be an open cover of $N$, with sufficiently small sets $U_{i}.$  Any section of $C\vert_{f^{-1} (U_{i})}$
is a $C^{\infty} ( f^{-1} (U_{i}))$-linear combination of distinguished sections.
For each $U_{i}$ we define
\begin{equation}
\widehat{\Psi}^{U_{i}} :\mathfrak{g} \rightarrow \mathrm{End}\,  \Gamma ( C\vert_{f^{-1}(U_{i})}),
\end{equation}
such that it satisfies  the Leibniz rule
\begin{equation}
\widehat{\Psi}^{U_{i}}(a) ( f s) = X_{a}^{M}(f) s + f \widehat{\Psi}^{U_{i}}(a) (s),
\end{equation} 
for any  $a\in \mathfrak{g}$, 
$f\in C^{\infty}( f^{-1} (U_{i}))$, $s\in \Gamma (C\vert_{f^{-1} (U_{i})})$, 
and on distinguished sections is given by
\begin{equation}
\widehat{\Psi}^{U_{i}}(a)  (f^{*}u, \mu + X ) :=(  f^{*}( \Psi (a) u), {\mathcal L}_{X_{a}^{M}}(\mu + X) ) .
\end{equation}

\begin{lem}\label{def-pull-back-action} The map
$\Psi :\mathfrak{g}\rightarrow \mathrm{End}\, \Gamma (C)$ given by 
\begin{equation}\label{widehat-psi}
\widehat{\Psi} (a) (s)\vert_{f^{-1}(U_{i})}= \widehat{\Psi}^{U_{i}} (a) (s\vert_{f^{-1}(U_{i})})
\end{equation}
is well defined, preserves $\Gamma (C^{\perp})$,  and  induces  an action
\begin{equation}
f^{!}\Psi:\mathfrak{g} \rightarrow \mathrm{Der} (f^{!}E)
\end{equation}
which lifts $\psi^{M}.$ It is called the {\cmssl pullback action of $\Psi$.}  
\end{lem}

\begin{proof} The statement  that $\widehat{\Psi}$ is well defined reduces  to showing that
for any $U_{k}, U_{p}\in \mathcal U$,  if  
$$
\sum_{i} \lambda_{i} ( f^{*}u_{i}, \mu_{i}+ X_{i}) =0
$$ 
where $\lambda_{i}\in C^{\infty}(f^{-1} (U_{k}\cap U_{p}))$ and 
 $( f^{*}u_{i}, \mu_{i}+ X_{i})$ are distinguished sections on $f^{-1} (U_{k}\cap U_{p})$, then 
\begin{equation}
 \sum_{i} \left( X_{a}^{M}(\lambda_{i} ) f^{*}u_{i} + \lambda_{i}f^{*} \Psi (a) (u_{i})\right)  =0. 
\end{equation}
This follows  by writing $u_{i}\in \Gamma (E\vert_{U_{k}\cap U_{p}})$ in terms of a  frame  of $E\vert_{U_{k}\cap U_{p}}$
and using the Leibniz rule for $\Psi (a)$ and 
that $X_{a}^{M}$ projects to $X_{a}^{N}$.  
The map $\widehat{\Psi}(a)$ takes values in $\Gamma (C)$ since 
for any distinguished section 
$(f^{*}u, \mu  + X)$,  we have
$$
\pi  \Psi  (a) (u) =  {\mathcal L}_{X_{a}^{N}} \pi (u) = f_{*} {\mathcal L}_{X_{a}^{M}} X.
$$
It   preserves  $\Gamma (C^{\perp})$ since
$$
\Psi (a) \pi^{*} (\gamma ) = \pi^{*}( {\mathcal L}_{X_{a}^{N}} \gamma ),\ \forall \gamma\in \Omega^{1}(N).
$$
Since $\Psi$ satisfies the relations  (\ref{cond-deriv}), also $f^{!}\Psi$ does (easy check).
\end{proof}

When  $E= E_{N}:= T^{*}N \oplus \mathcal G \oplus TN$ is a
standard  Courant algebroid and $\Psi =\Psi^{N}  :\mathfrak{g}\rightarrow \mathrm{Der}(E_{N})$
preserves the factors $T^{*}N$, $\mathcal G$ and $TN$ of $E_{N}$, 
the pullback action $f^{!}\Psi^{N}$ has a concrete formulation, as follows. Assume that $E_{N}$ 
is  defined 
by a bundle of quadratic Lie algebras $({ \mathcal G} ,[\cdot , \cdot]_{\mathcal G}, \langle\cdot , \cdot \rangle_{\mathcal G})$
and data    $(\nabla , R, H)$. Recall that 
we identify  $f^{!}E$
with  the standard Courant algebroid  
$E_{M}:= T^{*}M \oplus f^{*}\mathcal G\oplus TM$  defined by the bundle of quadratic Lie algebras $(f^{*}{ \mathcal G} ,f^{*}[\cdot , \cdot]_{\mathcal G}, f^{*}\langle\cdot , \cdot \rangle_{\mathcal G})$
and data    $(f^{*}\nabla ,f^{*} R, f^{*}H)$,
using the canonical   isomorphism $F$ defined by (\ref{iso-pull-exact}). Using  this identification, 
we obtain  an action
$$
\Psi^{M}:\mathfrak{g}\rightarrow \mathrm{Der}(E_{M}),\  \Psi^{M}(a):= F^{-1}\circ (f^{!}\Psi^{N} )(a) \circ F
$$
of $\mathfrak{g}$ on  $E_{M}$.

\begin{lem}\label{pull-back-action-prop} In the above setting, assume that  
\begin{equation}\label{rel-exact-1} 
\Psi^{N} (a) ( \xi + r + X) := {\mathcal L}_{X^{N}_{a}} \xi +\nabla^{\Psi}_{X^{N}_{a}} r
+ {\mathcal L}_{X_{a}^{N}} X, 
\end{equation}
where $\xi \in \Omega^1(N)$,  $r\in \Gamma (\mathcal{G})$ and $X\in \mathfrak{X}(N)$. Then 
\begin{equation}\label{rel-exact}
\Psi^{M} ( \xi +  r + X) := {\mathcal L}_{X^{M}_{a}} \xi + (f^{*}\nabla^{\Psi})_{X^{M}_{a}} r
+ {\mathcal L}_{X_{a}^{M}} X, 
\end{equation}
where  $\xi \in \Omega^1(M)$,  $r\in \Gamma (f^{*}\mathcal{G})$ and $X\in \mathfrak{X}(M)$.
\end{lem}

\begin{proof} The isomorphism $F$ given by (\ref{iso-pull-exact}) 
induces an isomorphism $F:\Gamma (E_{M}) \rightarrow \Gamma ( f^{!}E_{N})$ which satisfies 
\begin{equation}\label{definition-f}
F (\xi +  f^{*}r+ X) = [ ( f^{*} (r+f_{*}X), \xi + X ) ]
\end{equation}
where $r\in \Gamma (\mathcal G)$, $X\in {\mathfrak X}(M)$ is $f$-projectable 
and $\xi \in \Omega^{1}(M)$.
(In the right hand side of (\ref{definition-f})  $r+ f_{*}X\in \Gamma (\mathcal G\oplus TN)\subset \Gamma (E_{N})$). Then 
$$
(f^{!}\Psi^{N})(a) \circ F (\xi +  f^{*}r+ X) =  [ ( f^{*} (\nabla^{\Psi} _{X_{a}^{N}}r+{\mathcal L}_{X_{a}^{N}}f_{*}X), 
{\mathcal L}_{X_{a}^{N}}( \xi + X ) ) ],
$$
and, applying $F^{-1}$, we obtain (\ref{rel-exact}).
\end{proof}

The next proposition  states several compatibilities between pullback actions, isomorphisms, 
pullback and pushforward on spinors.

\begin{prop}\label{p-back-forward-inv}
i) Let $(E_{i}, \Psi_{i})$ ($i=1,2$)  be transitive Courant algebroids over $N$ with actions $\Psi_{i}:\mathfrak{g}\rightarrow
\mathrm{Der} (E_{i})$ which lift $\psi^{N}.$ If $I: (E_{1}, \Psi_{1})\rightarrow (E_{2}, \Psi_{2})$ is  invariant
with respect to $\Psi_{i}$, 
then $I^{f}: ( f^{!} E_{1}, f^{!}\Psi_{1})\rightarrow ( f^{!}E_{2}, f^{!} \Psi_{2})$ is  invariant
with respect to $f^{!}\Psi_{i}$.\

ii)  In the setting of 
Lemma \ref{def-pull-back-action}, assume that $M$ and $N$ are oriented and  let
$\Psi^{\mathbb{S}}:\mathfrak{g}\rightarrow \mathrm{End}\, \Gamma (\mathbb{S}_{E})$ and 
$(f^{!}\Psi)^{\mathbb{S}}:\mathfrak{g}\rightarrow \mathrm{End}\, \Gamma (\mathbb{S}_{f^{!}E})$ be the actions on  canonical weighted spinor bundles,  induced  by 
the actions $\Psi :\mathfrak{g}\rightarrow \mathrm{Der}(E)$ and $f^{!}\Psi :\mathfrak{g}\rightarrow \mathrm{Der} ( f^{!}E)$.  
Assume that the pullback $f^{!}: \Gamma (\mathbb{S}_{E})\rightarrow\Gamma (\mathbb{S}_{f^{!}E})$ is
defined and there is an admissible pair  $(I : E \rightarrow T^{*}N\oplus \mathcal G\oplus TN, S_{\mathcal G})$
for $\mathbb{S}_{E}$ and $\mathbb{S}_{f^{!}E}$ 
such that $I$ is invariant, cf.\ Section \ref{pull-back-subsection}.  
Then
\begin{equation}
f^{!}\circ \Psi^{\mathbb{S}}(a)   = ( f^{!}\Psi )^{\mathbb{S}}(a)\circ f^{!},\ \forall a\in \mathfrak{g}.
\end{equation}   
If, in addition, $f : M \rightarrow N$ has  compact fibers  then also the pushforward 
$f_{!}: \Gamma  (\mathbb{S}_{f^{!}E})\rightarrow  \Gamma (\mathbb{S}_{E})$
is defined and 
\begin{equation}
f_{!}\circ ( f^{!}\Psi )^{\mathbb{S}}(a)   =(-1)^{ r |s| + nr +\frac{r(r-1)}{2}} \Psi^{\mathbb{S}}(a)\circ f_{!},\ \forall a\in \mathfrak{g},
\end{equation}   
where $m$, $n$ and $r$ are the dimension of $M$, $N$ and the fibers of $f$.\  
\end{prop}

\begin{proof}
i)  We need to check that 
$$ 
I^{f} \circ f^{!}\Psi_{1}(a) [ ( f^{*} u, \eta +X ) ]= f^{!}\Psi_{2}(a) \circ I^{f}   [ ( f^{*} u, \eta +X ) ]
$$
for any distinguished section $[ ( f^{*}u, \eta +X) ]$ of $f^{!} E_{1}$,  which follows by applying the definitions of 
$I^{f}$ and $f^{!}\Psi_{i}$ and using that $I$ is invariant.\

 ii)  Using the admissible pair $(I, S_{\mathcal G})$  and   Lemma \ref{de-adaugat}, 
we can assume, without loss
of generality, that 
$$
E= E_{N}= T^{*}N\oplus \mathcal G\oplus TN,\ f^{!}E = E_{M}= T^{*}M\oplus f^{*}\mathcal G\oplus TM
$$ 
and $\Psi = \Psi^{N}$, $f^{!}\Psi = \Psi^{M}$ 
are given by (\ref{rel-exact-1}) and (\ref{rel-exact}) respectively. 
Let    $\mathbb{S}_{N}:= \Lambda (T^{*} N)\hat{\otimes} {\mathcal S}_{\mathcal G}$ 
and $\mathbb{S}_{M}:= \Lambda (T^{*}M)\hat{\otimes } f^{*} \mathcal S_{\mathcal G}$ be the
canonical weighted spinor bundles of   $E_{N}$ and $E_{M}$,
determined by  $S_{\mathcal G}$ and its pullback
$S_{f^{*}\mathcal G}:= f^{*} S_{\mathcal G}$ respectively. 
From the definition of $f^{!}$,  we need to show that
\begin{equation}\label{n-p}
f^{*} \Psi^{\mathbb{S}_{N}}(a)(\omega\otimes s) = \Psi ^{\mathbb{S}_{M}} (a) (f^{*}\omega \otimes f^{*}s) 
\end{equation}
for any $\omega \otimes s\in \Gamma (\mathbb{S}_{N})$, where  $\Psi^{\mathbb{S}_{N}}:\mathfrak{g}\rightarrow
\mathrm{End}\, \Gamma (\mathbb{S}_{N})$ and 
$\Psi^{\mathbb{S}_{M}}:\mathfrak{g}\rightarrow
\mathrm{End}\, \Gamma (\mathbb{S}_{M})$ are the induced actions on spinors
(given by Lemma \ref{standard-spinors})
and
$f^{*} $ is the map (\ref{pull-back-forms}).
Relation (\ref{n-p}) follows from (\ref{lie-spinors}), 
$f^{*} {\mathcal L}_{X_{a}^{N}} \omega = {\mathcal L}_{X_{a}^{M}} (f^{*}\omega )$ and $f^{*} 
( {\nabla}^{\Psi  , \mathcal S_{\mathcal G}}_{X_{a}^{N}} s )  =
({\nabla}^{f^{!}\Psi  , \mathcal S_{f^{*}\mathcal G}})_{X_{a}^{M}} (f^{*} s)$
for any $s\in \Gamma (\mathcal S_{\mathcal G})$
(the latter being a consequence of $\nabla^{f^{!}\Psi} = f^{*}\nabla^{\Psi}$).
The statement for the pushforward can be proved by a similar argument, 
which uses that $f_{*} {\mathcal L}_{X_{a}^{M}} \omega = {\mathcal L}_{X_{a}^{N}} f_{*}\omega$ for any form $\omega $ on $M$ and Lemma (\ref{push:lem}).
\end{proof}

\section{$T$-duality}

\subsection{Definition of $T$-duality}\label{T-dual-subsect}

Let $\pi : M \rightarrow B$ and $\tilde{\pi} : \tilde{M}\rightarrow B$ be principal  bundles over the same manifold $B$ with structure group  the $k$-dimensional torus $T^{k}$. For notational convenience, we will denote the structure group of $\tilde{\pi}$ 
by $\tilde{T}^k$ and its Lie algebra by $\tilde{t}^k$. We assume that $M$, $\tilde{M}$ and $B$ are oriented. 
Let 
$$
\mathrm{Lie}\, (T^{k}) = \mathfrak{t}^{k} \ni  a \mapsto \psi^{M} (a) := X_{a}^{M},\ 
\tilde{\mathfrak{t}}^{k} \ni a  \mapsto  \psi^{\tilde{M}} (a):= X_{a}^{\tilde{M}}, 
$$
be the vertical paralellism of $\pi$ and $\tilde{\pi}.$  We denote by
$$
N:= M\times_{B} \tilde{M} := \{ (m, \tilde{m})\in M\times \tilde{M}\mid  \pi (m) = \tilde{\pi}(\tilde{m})\} 
$$  
the fiber product of $M$ and $\tilde{M}$  and by $\pi_{N}: N \rightarrow M$ and $\tilde{\pi}_{N}: N \rightarrow \tilde{M}$ 
the natural projections.  The   actions of $T^{k}$ on $M$ and $\tilde{T}^k$ on $\tilde{M}$ induce naturally
an action of $T^{2k}=T^k\times \tilde{T}^k$ on $N$,  with infinitesimal action
$$
\mathfrak{t}^{2k} \ni a\rightarrow \psi^{N} (a)= X_{a}^{N}, 
$$
where,  for any $a\in \mathfrak{t}^{k}:= \mathfrak{t}^{k} \oplus 0 \subset \mathfrak{t}^{2k}$, 
\[ (\pi_{N})_{*}  X_{a}^{N}= X_{a}^{M},\quad (\tilde{\pi}_{N})_{*}  X_{a}^{N}= 0,\] 
and for any $a\in \tilde{\mathfrak{t}}^{k} :=  0 \oplus \mathfrak{t}^{k}\subset \mathfrak{t}^{2k}$, 
\[(\pi_{N})_{*}  X_{a}^{N}= 0,\quad (\tilde{\pi}_{N})_{*}  X_{a}^{N}=  X_{a}^{\tilde{M}}.\]

Let $E$ and $\tilde{E}$ be transitive Courant algebroids over $M$ and $\tilde{M}$, and assume they come with actions
$$
\Psi : \mathfrak{t}^{k} \rightarrow \mathrm{Der} (E),\  \tilde{\Psi} :\tilde{  \mathfrak{t}}^{k} \rightarrow \mathrm{Der} (\tilde{E}),
$$
which lift $\psi^{M}$ and $\psi^{\tilde{M}}$, such that there are invariant dissections $I: E \rightarrow T^{*}M\oplus \mathcal G\oplus TM$ 
and $\tilde{I}:\tilde{ E} \rightarrow T^{*}\tilde{M}\oplus\tilde{ \mathcal G}\oplus T\tilde{M}$ with the property that the 
partial connections $\nabla^{\Psi}$ and $\nabla^{\tilde{\Psi}}$  on $\mathcal G$ and $\tilde{\mathcal G}$, induced
by the actions,  have trivial holonomy
(it is easy to see that this condition is independent of the choice of invariant dissections). 
The pullback Courant algebroids $\pi_{N}^{!}E$ and $\tilde{\pi}_{N}^{!} \tilde{E}$ inherit  the  pullback actions 
(see Lemma \ref{def-pull-back-action})
\begin{equation}
\pi_{N}^{!} \Psi  : \mathfrak{t}^{k} \rightarrow \mathrm{Der} (\pi_{N}^{!} E),\ 
\tilde{\pi}_{N}^{!} \Psi  : \tilde{\mathfrak{t}}^{k} \rightarrow \mathrm{Der} (\tilde{\pi}_{N}^{!} \tilde{E})
\end{equation}
which lift the infinitesimal actions $\mathfrak{t}^{k} \ni a\rightarrow X_{a}^{N}$ and  $\tilde{\mathfrak{t}}^{k} \ni a\rightarrow X_{a}^{N}$ respectively. 
The situation is summarized in the following commutative diagram, in which 
the arrows pointing down are quotient maps with respect to principal $T^k$-actions: $B=M/T^k=\tilde{M}/\tilde{T}^k=N/T^{2k}$, $M=N/\tilde{T}^k$, $\tilde{M}=N/T^k$ ($T^{2k}=T^k\times \tilde{T}^k$). 
\[
 \begin{xy}
    \xymatrix{ 
 &\mathfrak{t}^k\curvearrowright\pi_N^{!}E \ar@{->}[r]  &  \ar@{->}_{\pi_N}[dl]  N 
 \ar@{->}[dr]^{\tilde{\pi}_N} & \tilde{\pi}_N^{!}\tilde{E}\ar@{->}[l]\curvearrowleft \tilde{\mathfrak{t}}^k&\\
\mathfrak{t}^k\curvearrowright     E \ar@{->}[r]  &M\ar@{->}[dr]& & \tilde{M}\ar@{->}[dl] & \tilde{E}\ar@{->}[l]\curvearrowleft \tilde{\mathfrak{t}}^k\\
 &&B&& }
 \end{xy}
\]
The next lemma  extends  the action $\pi_{N}^{!}\Psi$ to a $\mathfrak{t}^{2k}$-action and states some
of the  properties of this $\mathfrak{t}^{2k}$-action.

\begin{lem}\label{extensie-action} i) The map
$$
\Psi^{\pi_{N}^{!}E}: \mathfrak{t}^{2k}\rightarrow \mathrm{Der} (\pi_{N}^{!}E)
$$
which on $\mathfrak{t}^{k}$ coincides with  $\pi_{N}^{!}\Psi$ 
and the evaluation of which on any  $b\in \tilde{t}^{k}$
satisfies the Leibniz rule
$$
\Psi^{\pi_{N}^{!}E}(b)( f s) = X_{b}^{N} (f) s + f  \Psi^{\pi_{N}^{!}E}(b)( s) ,\ \forall  f\in C^{\infty}(N),\ 
s\in \Gamma (\pi_{N}^{!}E)
$$
and on distinguished sections $[ ( \pi_{N}^{*} (u), \xi +X ) ] $ of  $\pi_{N}^{!}E$
is given by  
\begin{equation}\label{extindere-actiuni}
\Psi^{\pi_{N}^{!}E}(b)[(  \pi_{N}^{*} (u),  \xi +X ) ]  = [ ( 0, {\mathcal L}_{X_{b}^{N}} ( \xi +X ) ) ]
\end{equation}
is a well defined action on $\pi_{N}^{!}E.$\

ii) Let $(E_{1}, \Psi_{1})$ be another  transitive Courant algebroid over $M$ with an action
$\Psi_{1} :\mathfrak{t}^{k}\rightarrow \mathrm{Der} (E_{1})$ which lifts $\psi^{M}$. If $I : E \rightarrow E_{1}$
is  an isomorphism invariant with respect to $\Psi$ and $\Psi_{1}$, then the pullback isomorphism $I^{\pi_{N}}:
\pi_{N}^{!} E \rightarrow \pi_{N}^{!} E_{1}$  is invariant with respect to
$\Psi^{\pi_{N}^{!}E}$ and $\Psi^{\pi_{N}^{!}E_{1}}$ (the latter defined as $\Psi^{\pi_{N}^{!}E}$, using
$\Psi_{1}$ instead of $\Psi$).
\end{lem}

\begin{proof}   Claim i)  follows  from arguments similar to the proof of Lemma \ref{def-pull-back-action}.
To prove  claim ii), we need to show that
\begin{equation}\label{pull-back-ext-iso}
I^{\pi_{N}}\circ  \Psi^{\pi_{N}^{!}E} (a) (s) = \Psi_{1}^{\pi_{N}^{!} E_{1}} (a) \circ I^{\pi_{N}} (s),\ 
\forall a\in \mathfrak{t}^{2k},\ s\in \Gamma ( \pi_{N}^{!} E).
\end{equation}
Relation (\ref{pull-back-ext-iso}) with $a\in \mathfrak{t}^{k}$, 
follows from  Proposition \ref{p-back-forward-inv} i).  
Relation (\ref{pull-back-ext-iso}) with $a\in \tilde{\mathfrak{t}}^{k}$
follows by assuming that $s$ is a distinguished section and using
(\ref{extindere-actiuni}) together with the definition of $I^{\pi_{N}}$, cf.\ equation (\ref{I^f:eq}).
 \end{proof}

\begin{lem}\label{standard-properties} i) If  $E=  T^{*}{M}\oplus\mathcal G\oplus TM$ is a standard Courant algebroid and 
$\Psi $ preserves the summands
$T^{*}M$, $\mathcal G$, $TM$ of $E$,  then 
$\Psi^{\pi_{N}^{!}E}$  preserves the summands  $T^{*}N$, $\pi_{N}^{*}\mathcal G$, $TN$
of $\pi_{N}^{!}E = T^{*}N\oplus \pi_{N}^{*}\mathcal G\oplus TN$. The  
partial connection
$\nabla^{\Psi^{\pi_{N}^{!}E}}$ on $\pi_{N}^{*}\mathcal G$ 
associated to  $\Psi^{\pi_{N}^{!}E}$ 
is the pullback of the partial connection 
$\nabla^{\Psi}$ on $\mathcal G$ associated to $\Psi$: 
\begin{equation}\label{pull-back-extension}
\nabla^{\Psi^{\pi_{N}^{!}E}}_{X_{a}^{N}}(\pi_{N}^{*}r) =\pi_{N}^{*}\nabla^{\Psi}_{X_{a}^{M}}r,\
\nabla^{\Psi^{\pi_{N}^{!}E}}_{X_{b}^{N}}(\pi_{N}^{*}r)=0,\  \forall r\in \Gamma (\mathcal G),\ a\in \mathfrak{t}^{k},\  b\in \tilde{\mathfrak{t}}^{k}.
\end{equation}

ii) A section of $\pi_{N}^{*}\mathcal G$ is $\nabla^{\Psi^{\pi_{N}^{!}E}}$-parallel if and only if 
it is  the pullback by $\pi_{N}$ of
a $\nabla^{\Psi}$-parallel section of $\mathcal G$
(or the pullback by $\Pi := \pi\circ \pi_{N}$ of a section of $\mathcal G_{B}$).
\end{lem}

\begin{proof} Claim i) follows from an argument similar to the proof of Lemma 
\ref{pull-back-action-prop}. Claim ii) follows immediately from claim i) 
(recall the definition of the bundle $\mathcal G_{B}\rightarrow B$ from Section
\ref{example-t1}).
\end{proof}

When $E$ is a standard Courant algebroid like in Lemma \ref{extensie-action} ii),  the  partial connection $\nabla^{\Psi^{\pi_{N}^{!}E}}$  
will be denoted by $\nabla^{\Psi , \pi_{N}^{!}E}$. 
Let $S_{\mathcal G}$ be an irreducible  $\mathrm{Cl} (\mathcal G)$-bundle,
with canonical spinor bundle $\mathcal S_{\mathcal G}.$  
Then $\mathcal S_{\pi_{N}^{*}\mathcal G}=\pi_{N}^{*}\mathcal S_{\mathcal G}$ 
is the canonical spinor bundle of the irreducible  $\mathrm{Cl} ( \pi_{N}^{*} \mathcal G)$-bundle
$S_{\pi^{*}_{N}\mathcal G}:= \pi_{N}^{*}S_{\mathcal G}$ 
and the partial connection  $\nabla^{ \Psi^{\pi_{N}^{!}E}, {\mathcal S}_{\pi_{N}^{*} \mathcal G}}$
on $\mathcal S_{\pi_{N}^{*}\mathcal G}$ 
induced by any partial connection on $S_{\pi^{*}_{N}\mathcal G}$ 
compatible with $\nabla^{\Psi , \pi_{N}^{!}E}$ 
is the pullback of 
the partial connection $\nabla^{\Psi, \mathcal S_{\mathcal G}}$
on $\mathcal S_{\mathcal G}$ induced  by any partial connection  on $S_{\mathcal G}$ compatible with $\nabla^{\Psi}$, that is, 
\begin{equation}\label{more-precisely}
\nabla^{ \Psi , \pi_{N}^{*} \mathcal S_{\mathcal G}}_{X_{a}^{N}} = (\pi_{N}^{*} \nabla^{\Psi , \mathcal S_{\mathcal G}})_{X_{a}^{N}},\ \forall a\in {\mathfrak{t}}^{2k}. 
\end{equation}  
In a similar way,   we construct an  action 
$\tilde{\Psi}^{\tilde{\pi}_{N}^{!}\tilde{ E}}: \mathfrak{t}^{2k} \rightarrow \mathrm{Der} (\tilde{\pi}_{N}^{!} \tilde{E})$
which extends $\tilde{\pi}_{N}^{!}\tilde{\Psi}$.
When  $\tilde{E} = T^{*}\tilde{M}\oplus\tilde{ \mathcal G}\oplus T\tilde{M}$
is a standard Courant algebroid  
and $\tilde{\Psi}$ preserves the factors $T^{*}\tilde{M}$, $\tilde{\mathcal G}$ and $T\tilde{M}$
of $\tilde{E}$, 
we use the notation  $\nabla^{\tilde{\Psi }, \tilde{\pi}_{N}^{!}\tilde{E}}$ for the 
partial connection $\nabla^{\tilde{\Psi}^{\tilde{\pi}_{N}^{!}\tilde{E}}}$  
on $\tilde{\pi}_{N}^{*}\tilde{\mathcal G}$.
It  is related to $\nabla^{\tilde{\Psi}}$ by relations analogous to
(\ref{more-precisely}).\

From now on  the Courant algebroids    $\pi_{N}^{!} E$ and  $\tilde{\pi}_{N}^{!} \tilde{E}$
will be considered with the $\mathfrak{t}^{2k}$-actions ${\Psi}^{{\pi}_{N}^{!}{ E}}$
and $\tilde{\Psi}^{\tilde{\pi}_{N}^{!}\tilde{ E}}$.

\begin{defn}\label{def-T-duality} The Courant algebroids $E$ and $\tilde{E}$ are called {\cmssl $T$-dual} if there is an invariant
fiber preserving Courant algebroid isomorphism
$$
F: \pi_{N}^{!} E \rightarrow \tilde{\pi}_{N}^{!} \tilde{E}
$$
such that the following non-degeneracy condition is satisfied. 
Let
$$
I: E \rightarrow T^{*}M\oplus \mathcal G\oplus TM,\ \tilde{I}: \tilde{E} \rightarrow T^{*}\tilde{M}\oplus \tilde{\mathcal G}
\oplus T\tilde{M},
$$
be    dissections of $E$ and $\tilde{E}$, and 
$$
I^{\pi_{N}}:\pi_{N}^{!} E \rightarrow T^{*}N\oplus \pi_{N}^{*}\mathcal G\oplus TN,\ \tilde{I}^{\tilde{\pi}_{N}}: \tilde{\pi}_{N}^{!}\tilde{E} \rightarrow T^{*}N \oplus \tilde{\pi}_{N}^{*}\tilde{\mathcal G}
\oplus {TN}
$$
the induced dissections of  $\pi_{N}^{!} E$ and $\tilde{\pi}_{N}^{!}\tilde{ E}$.
Let   $(\beta  ,  \Phi , K)$, where  
$ \beta  \in \Omega^{2}(N)$,
$\Phi \in \Omega^{1}(N,  \tilde{\pi}_{N}^{*}\tilde{\mathcal G})$ and $K\in \mathrm{Isom} (\pi_{N}^{*} \mathcal G ,
\tilde{\pi}_{N}^{*} \tilde{\mathcal G})$, be the data which defines 
the isomorphism
$$
\tilde{I}^{\tilde{\pi}_{N}}\circ F\circ( I^{\pi_{N}})^{-1}:  T^{*}N\oplus \pi_{N}^{*}\mathcal G\oplus TN\rightarrow
 T^{*}N\oplus\tilde{ \pi}_{N}^{*}\tilde{\mathcal G}\oplus TN
$$
(according to relation (\ref{concrete-iso}) from Section \ref{trans-sect-basic}). 
Then 
\begin{equation}\label{pairing-non-deg}
\beta  - \Phi^{*} \Phi : \mathrm{Ker}\ (d \pi_{N})  \times \mathrm{Ker}\  ( d \tilde{\pi}_{N}) \rightarrow \mathbb{R}
\end{equation}
is non-degenerate.
\end{defn}

\begin{defn} The above  definition is independent of the choice of  dissections.
\end{defn} 

\begin{proof} Let $I_{i}: E \rightarrow T^{*}M\oplus \mathcal G_{i}\oplus TM$
($i=1,2$) be two dissections 
of $E$. Then 
$$
\hat{F}_{i}:= \tilde{I}^{\tilde{\pi}_{N}}\circ F\circ( I_{i}^{\pi_{N}})^{-1}:  T^{*}N\oplus \pi_{N}^{*}\mathcal G_{i}\oplus TN\rightarrow
 T^{*}N\oplus\tilde{ \pi}_{N}^{*}\tilde{\mathcal G}\oplus TN
$$
satisfy 
\begin{equation}
\hat{F}_{2} = \hat{F}_{1} \circ (I_{1}\circ I_{2}^{-1})^{\pi_{N}}.
\end{equation}
Assume that the dissections $I_{1}$ and $I_{2}$ are  related by $(\beta , K, \Phi )$.
Then from Lemma  \ref{pull-back-bdle} iii)  
the induced dissections of $\pi_{N}^{!}E$ are related by  $(\pi_{N}^{*}\beta ,\pi_{N}^{*} K,\pi_{N}^{*} \Phi )$, i.e.\  
the Courant algebroid  isomorphism 
$$
(I_{1}\circ I_{2}^{-1})^{\pi_{N}}:  T^{*}N\oplus \pi_{N}^{*}\mathcal G_{2}\oplus TN \rightarrow 
 T^{*}N\oplus \pi_{N}^{*}\mathcal G_{1}\oplus TN 
$$
is given by (\ref{concrete-iso}), with $(\beta , K, \Phi )$ replaced by $(\pi_{N}^{*}\beta ,\pi_{N}^{*} K, \pi_{N}^{*}\Phi )$.
The independence of the non-degeneracy  condition
(\ref{pairing-non-deg}) 
on the  dissection of $E$ follows from (\ref{iso-standard-3}).
The independence on the  dissection of $\tilde{E}$
can be proved similarly.
\end{proof}

\begin{rem}{\rm 
Unlike the $T$-duality for exact Courant algebroids, the
definition of $T$-dual transitive Courant algebroids $E$ and $\tilde{E}$ is not symmetric
with respect to $E$ and $\tilde{E}$, 
in general. 
This follows from the lack of symmetry in  the non-degeneracy condition from Definition \ref{def-T-duality}
}
\end{rem}

\begin{lem}\label{T-duality-iso}  Let $(E_{1}, \Psi_{1})$ 
and $(\tilde{E}_{1}, \tilde{\Psi}_{1})$  be  transitive Courant algebroids over $M$ and $\tilde{M}$, 
together with actions 
$$
\Psi_{1}:\mathfrak{t}^{k}\rightarrow \mathrm{Der}(E_{1}),\ \tilde{\Psi}_{1}:\tilde{\mathfrak{t}}^{k}\rightarrow \mathrm{Der}(\tilde{E}_{1})
$$
which lift $\psi^{M}$
and $\psi^{\tilde{M}}$ respectively.   Assume that
\begin{equation}\label{iso-I}
G: E_{1}\rightarrow E,\ \tilde{G}: \tilde{E}_{1}\rightarrow \tilde{E}
\end{equation}
are invariant,  fiber preserving  Courant algebroid isomorphisms. 
If $E$ and $\tilde{E}$ are $T$-dual, then also $E_{1}$ and $\tilde{E}_{1}$ are $T$-dual.
\end{lem}

\begin{proof} Let $F: \pi_{N}^{!} E\rightarrow \tilde{\pi}_{N}^{!}\tilde{E}$ be an isomorphism
which satisfies the conditions from Definition \ref{def-T-duality}.
Then the isomorphism 
\begin{equation}\label{iso-1}
F_{1}:= (\tilde{G}^{\tilde{\pi}_{N}})^{-1}\circ F\circ G^{\pi_{N}}:
\pi_{N}^{!} E_{1} \rightarrow \tilde{\pi}_{N}^{!}\tilde{E}_{1}.
\end{equation}
satisfies  the same conditions. (For the invariance 
of  $\tilde{G}^{\tilde{\pi}_{N}}$ and   ${G}^{{\pi}_{N}}$ we use Lemma \ref{extensie-action} ii)).
\end{proof}

The next lemma   states  the conditions that two standard Courant algebroids are
$T$-dual.  Let $E = T^{*}M\oplus \mathcal G \oplus TM$ and $\tilde{E}= T^{*}\tilde{M}\oplus \tilde{\mathcal G}\oplus  T\tilde{M}$ be standard Courant algebroids over $M$ and $\tilde{M}$,
defined by a quadratic Lie algebra bundle $(\mathcal G , [\cdot , \cdot ]_{\mathcal G}, \langle\cdot , \cdot \rangle_{\mathcal G})$ and  data $(\nabla^{E}, R, H)$ and, respectively,  
a quadratic Lie algebra bundle $(\tilde{\mathcal G} , [\cdot , \cdot ]_{\tilde{\mathcal G}}, \langle\cdot , \cdot \rangle_{\tilde{\mathcal G}})$ and data  $(\nabla^{\tilde{E}}, \tilde{R}, \tilde{H})$ 
(as there are various connections involved, we choose to 
use the notation  $\nabla^{E}$ rather than $\nabla$ for the connection which is part of the data which defines  $E$;  a similar convention is used for
$\tilde{E}$). Let $\Psi : \mathfrak{t}^{k}\rightarrow \mathrm{Der}\, (E)$  and 
$\tilde{\Psi }: \tilde{\mathfrak{t}}^{k}\rightarrow \mathrm{Der}\, (\tilde{E})$ be actions which lift $\psi^{M}$ and
$\psi^{\tilde{M}}$ and preserve the factors of $E$ and $\tilde{E}.$

\begin{lem}\label{iso-upstairs} 
The standard  Courant algebroids $E$ and $\tilde{E}$ are $T$-dual if and only if
there are invariant forms $\beta\in \Omega^{2}(N)$ and $\Phi \in \Omega^{1}(N, \tilde{\pi}_{N}^{*}\tilde{\mathcal{G}})$
and  a quadratic Lie algebra bundle isomorphism $K\in \mathrm{Isom} (\pi_{N}^{*} \mathcal G ,
\tilde{\pi}_{N}^{*} \tilde{\mathcal G})$ which maps
$\nabla^{\Psi , \pi_{N}^{!} E}$ to $\nabla^{\tilde{\Psi}, \tilde{\pi}_{N}^{!} \tilde{E}}$ 
such that the non-degeneracy condition (\ref{pairing-non-deg}) and  
the following relations hold:
\begin{align}
\label{con-T}& (\tilde{\pi}_{N}^{*} \nabla^{\tilde{E}} )_{X} r = K (\pi_{N}^{*} \nabla^{E})_{X} (K^{-1} r) + [ r, \Phi (X) ]_{\tilde{\pi}_{N}^{*} 
\tilde{\mathcal G}},\\
 \label{tilde-T}&  K \pi_{N}^{*} R- \tilde{\pi}_{N}^{*}\tilde{R} = 
d^{\tilde{\pi}^{*}_{N}\nabla^{\tilde{E}}} \Phi  + c_{2},\\
\label{H-tH}&  \pi_{N}^{*} H -  \tilde{\pi}_{N}^{*} \tilde{H} =
d\beta +  \langle (K\pi_{N}^{*} R  +\tilde{\pi}_{N}^{*}\tilde{ R})\wedge \Phi \rangle_{\tilde{\pi}_{N}^{*}\tilde{\mathcal{G}}} - c_{3}, 
\end{align}
where $c_{2} \in \Omega^{2} (N, \tilde{\pi}_{N}^{*}\tilde{ \mathcal G})$ and
$c_{3}\in \Omega^{3}(N)$ 
are defined by
\begin{align}
\nonumber& c_{2} (X, Y) :=  [ \Phi (X ), \Phi (Y) ]_{\tilde{\pi}_{N}^{*} \tilde{\mathcal G}},\\
\nonumber&  c_{3}(X, Y, Z) := 
 \langle \Phi (X), [ \Phi (Y), \Phi (Z) )]_{\tilde{\pi}_{N}^{*} \tilde{\mathcal G}}\rangle_{\tilde{\pi}_{N}^{*} \tilde{\mathcal G}},
 \end{align}
for any $X , Y, Z\in {\mathfrak X}(N).$
\end{lem}

\begin{proof} 
Since $E$ and $\tilde{E}$ are standard Courant algebroids, 
so are $\pi_{N}^{!}E$ and $\tilde{\pi}_{N}^{!}\tilde{E}$ 
and an invariant isomorphism 
$F: \pi_{N}^{!}E\rightarrow  \tilde{\pi}_{N}^{!}\tilde{E}$ is defined by invariant  data
$(\beta , K, \Phi )$, where
$\beta \in \Omega^{2}(N)$, $K\in \mathrm{Isom} ( \pi_{N}^{*}\mathcal G , \tilde{\pi}_{N}^{*}\tilde{G})$
and $\Phi \in \Omega^{1}(N, \tilde{\pi}_{N}^{*}\tilde{\mathcal G}).$  
The above relations coincide with relations (\ref{iso-chen}), applied 
to  $F$ and 
$X, Y, Z\in {\mathfrak X} (N).$
\end{proof}

We end this section with a simple lemma on the existence of preferred dissections of $T$-dual transitive Courant
algebroids.

\begin{lem}\label{schimbare-disectie}  Let $E$ and $\tilde{E}$ be $T$-dual transitive Courant algebroids (not necessarily in the standard
form). Then $E$ and $\tilde{E}$ admit invariant dissections of the form
\begin{equation}\label{tilde-ii}
I: E\rightarrow T^{*}M\oplus \pi^{*} {\mathcal G}_{B}\oplus TM,\
\tilde{I}: \tilde{E}\rightarrow T^{*}\tilde{M}\oplus \tilde{\pi}^{*}\mathcal G_{B} \oplus T\tilde{M}, 
\end{equation}
where $(\mathcal G_{B}, \langle \cdot , \cdot \rangle_{\mathcal G_{B}}, [\cdot , \cdot ]_{\mathcal G_{B}})$
is a quadratic Lie algebra bundle on $B.$
\end{lem}

\begin{proof}
Let $I : E \rightarrow E_{M} =  T^{*}{M}\oplus \mathcal G\oplus TM$ and 
$\tilde{I} : \tilde{E} \rightarrow  E_{\tilde{M}}= T^{*}\tilde{M}\oplus \tilde{\mathcal G}\oplus T\tilde{M}$ be invariant dissections
of $E$ and $\tilde{E}.$ From 
Lemma \ref{T-duality-iso}, $E_{M}$ and $E_{\tilde{M}}$ are $T$-dual. 
Let   $F: \pi_{N}^{!}E_{M}\rightarrow \pi_{N}^{!}{E}_{\tilde{M}}$ be an invariant isomorphism, and assume it is defined by  
data  $( \beta , K, \Phi ).$ 
Since  $F$  is invariant, the quadratic Lie algebra
bundle isomorphism $K : \pi_{N}^{*}\mathcal G \rightarrow \tilde{\pi}_{N}^{*}\tilde{\mathcal{G}}$ maps  $\nabla^{\Psi , \pi_{N}^{!}E}$ to $\nabla^{\tilde{\Psi }, \tilde{\pi}_{N}^{!}\tilde{E}}$.
From Lemma \ref{standard-properties} ii), $K= \Pi^{*} K_{B}$
where  $K_{B} : \mathcal G_{B}\rightarrow \tilde{\mathcal G}_{B}$ is a quadratic Lie algebra bundle
isomorphism  and  $ \mathcal G= \pi^{*} \mathcal G_{B}$, $\tilde{\mathcal G}= \tilde{\pi}^{*} \tilde{\mathcal G}_{B}$ (see Section
\ref{example-t1}). Using Lemma \ref{c-d-k}, we  change the dissection $\tilde{I}$ to obtain a new invariant dissection 
of $\tilde{E}$  with $\tilde{\mathcal G}= \tilde{\pi}^{*} \tilde{\mathcal G}_{B}\cong \tilde{\pi}^*\mathcal{G}_B$ replaced by $\tilde{\pi}^*\mathcal{G}_B$.  
\end{proof}

\subsection{$T$-duality and spinors}
Assume  that $E$ and $\tilde{E}$ are $T$-dual transitive Courant
algebroids and let 
$$
F: \pi_{N}^{!} E\rightarrow \tilde{\pi}_{N}^{!}\tilde{E}
$$ 
be an invariant isomorphism as in  Definition
\ref{def-T-duality}. Let   $\mathbb{S}_{E}$,  $\mathbb{S}_{\tilde{E}}$,
$\mathbb{S}_{\pi_{N}^{!} E}$ and $\mathbb{S}_{\tilde{\pi}_{N}^{!}\tilde{ E}}$
be   canonical weighted  spinor bundles of $E$, 
$\tilde{E}$, $\pi_{N}^{!}E$ and $\tilde{\pi}_{N}^{!}\tilde{E}$ respectively,
such that the pullbacks $\pi_{N}^{ !}: \Gamma (\mathbb{S}_{E}) 
\rightarrow \Gamma (\mathbb{S}_{\pi_{N}^{!} E})$ and 
 $\tilde{\pi}_{N}^{ !}: \Gamma (\mathbb{S}_{\tilde{E}}) 
\rightarrow \Gamma (\mathbb{S}_{\tilde{\pi}_{N}^{!}\tilde{ E}})$
are defined. We  consider an  admissible pair  
$(I, S_{\mathcal G})$  for
$\mathbb{S}_{E}$ and $\mathbb{S}_{\pi_{N}^{!}{E}}$,  and 
an admissible pair $(\tilde{I}, S_{\tilde{\mathcal G}})$ for
$\mathbb{S}_{\tilde{E}}$ and $\mathbb{S}_{\tilde{\pi}_{N}^{!}\tilde {E}}$, 
with invariant dissections 
$$
I: E \rightarrow E_{M} = T^{*}M\oplus\pi^{*} \mathcal G_{B}\oplus TM,\ \tilde{I}: \tilde{E}
\rightarrow  \tilde{E}_{\tilde{M}}= T^{*}\tilde{M}\oplus \tilde{\pi}^{*}\tilde{\mathcal G}_{B} \oplus T\tilde{M}
$$
such that   $S_{\mathcal G}= \pi^{*}S_{B}$
and $S_{\tilde{\mathcal G}} = \tilde{\pi}^{*}\tilde{S}_{B}$, where $S_{B}$  is an  irreducible
$\mathrm{Cl } ({\mathcal G}_{B})$-bundle and $\tilde{S}_{B}$  is an irreducible 
$\mathrm{Cl} (\tilde{\mathcal G}_{B})$-bundle. We assume that    
the isomorphism  $F_{\mathbb{S}}: \mathbb{S}_{\pi^{!}_{N}  E} \rightarrow \mathbb{S}_{\tilde{\pi}_{N} ^{!}\tilde{E}}$ induced by $F$ is  globally defined. 
This   is equivalent  to assuming  that  
the isomorphism   $(F_{1})_{\mathbb{S}}: \mathbb{S}_N \rightarrow \tilde{\mathbb{S}}_N$ compatible with $F_{1}:= \tilde{I}^{\tilde{\pi}_{N}}\circ F\circ  (I^{\pi_{N}})^{-1}: \pi_{N}^{!} E_{M}\rightarrow \tilde{\pi}_{N}^{!}\tilde{E}_{\tilde{M}}$
is globally defined,  where $\mathbb{S}_N = \Lambda  (T^{*}N)\hat{\otimes}\Pi^{*}\mathcal  S_{B}$ and 
$\tilde{\mathbb{S}}_N = \Lambda  (T^{*}N)\hat{\otimes}\Pi^{*}\mathcal{\tilde{S}}_{B}$ are spinor bundles
of  $\pi_{N}^{!} E_{M} = T^{*}N \oplus \Pi^{*} \mathcal G_{B} \oplus TN$ 
and $\tilde{\pi}_{N}^{!}\tilde{E}_{\tilde{M}}= T^{*}N \oplus \Pi^{*} \tilde{\mathcal G}_{B} \oplus TN$ respectively 
(and  $\Pi = \pi \circ\pi_{N} = \tilde{\pi}\circ \tilde{\pi}_{N}$).

\begin{rem}{\rm \label{F1:rem}
When the dissections are chosen  such that  
$\mathcal G_{B} =\tilde{\mathcal G}_{B}$  as quadratic Lie algebra bundles (see Lemma \ref{schimbare-disectie})   and ${S}_{B} =\tilde{ S}_{B}$ as $\mathrm{Cl}(\mathcal G_{B})$-bundles, 
$F_{1}$ is an automorphism 
of the vector bundle 
$$ 
\pi_{N}^{!} E_{M}=\tilde{\pi}_{N}^{!}E_{\tilde{M}} =  T^{*}N\oplus \Pi^{*}\mathcal G_{B}\oplus TN
$$
with scalar product
\begin{equation}\label{scalar-products-pull}
\langle \xi + \Pi^{*} (r_{1}) +X, \eta + \Pi^{*} (r_{2}) + Y\rangle =
\frac{1}{2} ( \xi (Y) +\eta (X)) +\Pi^{*} \langle r_{1},r_{2}\rangle_{\mathcal G_{B}}
\end{equation}
and  $(F_{1})_{\mathbb{S}}$ is an automorphism of  the spinor bundle 
$\mathbb{S}_N =  \Lambda  (T^{*}N)\hat{\otimes}\Pi^{*}\mathcal  S_{B}$ 
of   $ T^{*}N\oplus \Pi^{*}\mathcal G_{B}\oplus TN$.
If   $F_{1}$ belongs to the image of the exponential map 
$\mathrm{exp} :\mathfrak{so}(T^{*} N\oplus \Pi^{*}\mathcal G_{B}\oplus TN)\rightarrow
\mathrm{SO}_{0} (T^{*} N\oplus \Pi^{*}\mathcal G_{B}\oplus TN)$, then 
$(F_{1})_{\mathbb{S}}$ is automatically globally defined (cf.\ Remark \ref{global-remark}).
In fact, in that case $F_{1}$ can be lifted to $(F_{1})_{\mathbb{S}}$ using the exponential map for 
$\mathfrak{spin}(T^{*} N\oplus \Pi^{*}\mathcal G_{B}\oplus TN)$.} 
\end{rem}

\begin{thm}\label{thm-T-duality}  i) The map
\begin{equation}\label{def-tau}
\tau : =  (\tilde{\pi}_{N})_{!} \circ  F_{\mathbb{S}} \circ \pi_{N}^{!} : \Gamma ( \mathbb{S}_{E})\rightarrow \Gamma (\mathbb{S}_{\tilde{E}}) 
\end{equation}
intertwines the canonical Dirac generating operators  of $E$ and $\tilde{E}$ and 
maps invariant spinors to invariant spinors.
In particular,   there is the following commutative diagram
\begin{equation*} \begin{xy}
    \xymatrix{ 
 \Gamma_{\mathfrak{t}^{2k}}(\mathbb{S}_{\pi_N^{!}E}) \ar@{->}[r]^{F_\mathbb{S}}&  
 \Gamma_{\mathfrak{t}^{2k}}(\mathbb{S}_{\tilde{\pi}_N^{!}\tilde{E}})\ar@{->}[d]^{(\tilde{\pi}_N)_!}\\
 \Gamma_{\mathfrak{t}^{k}}(\mathbb{S}_{E})\ar@{->}[u]^{\pi_N^!} \ar@{->}[r]^\tau & \Gamma_{\tilde{\mathfrak{t}}^{k}}(\mathbb{S}_{\tilde{E}})}
    \end{xy}
\end{equation*}

ii) There is an isomorphism $\rho : \Gamma_{\mathfrak{t}^{k} }(E) \rightarrow \Gamma_{ \tilde{\mathfrak{t}}^{k}} (\tilde{E})$ 
of $C^{\infty}(B)$-modules  
which preserves Courant brackets, scalar products and is compatible with $\tau$, i.e. 
\begin{equation}\label{properties-rho}
\tau (\gamma_{u} s) = \gamma_{\rho (u)} \tau (s),\ 
 [\rho (u), \rho (v) ]_{\tilde{E}}   = \rho [u, v]_{E},\ 
\langle \rho (u), \rho (v) 
\rangle_{ \tilde{E}}= \langle u, v\rangle_{ E},
\end{equation}
for any $u, v\in \Gamma_{\mathfrak{t}^{k}} (E)$ and $s\in \Gamma_{\mathfrak{t}^{k}}( \mathbb{S}_{E}).$
\end{thm}

The claim  from Theorem \ref{thm-T-duality}  i) concerning the canonical Dirac generating operators follows 
from  Propositions  \ref{iso-dirac},  \ref{lem-pull} and  \ref{lem-push}.  The remaining claims  from Theorem 
\ref{thm-T-duality} 
will be proved in
the next two lemmas.

\begin{lem} 
In the above setting, $E_{M}$ and $\tilde{E}_{\tilde{M}}$ are $T$-dual.
The assertions of Theorem \ref{thm-T-duality}
hold for the pair $(E, \tilde{E})$
and canonical weighted spinor bundles $\mathbb{S}_{E}$ and $\mathbb{S}_{\tilde{E}}$ 
if and only if they hold  for the pair $(E_{M}, \tilde{E}_{\tilde{M} })$ and  canonical
weighted spinor bundles
$\mathbb{S}_{M}: = \Lambda (T^{*}M)\hat{\otimes}\pi^{*} \mathcal S_{B}$ and 
$\mathbb{S}_{\tilde{M}}: = \Lambda (T^{*}\tilde{M})\hat{\otimes}\tilde{\pi}^{*} \tilde{\mathcal S}_{B}$
of $E_{M}$ and $\tilde{E}_{\tilde{M}}$ respectively. 
\end{lem}

\begin{proof} 
The fact that $E_{M}$ and $\tilde{E}_{\tilde{M}}$  are $T$-dual follows
from  Lemma \ref{T-duality-iso}.
We will assume that the assertions of Theorem \ref{thm-T-duality}
hold for $(E, \tilde{E})$ and show that they hold for $(E_{M}, \tilde{E}_{\tilde{M} })$ as well.  The same arguments prove also the converse statement. Using  relation (\ref{iso-1}), we obtain   that the map
$\tau_{1}$ defined for the pair $(E_{M}, \tilde{E}_{\tilde{M}})$ 
as is defined $\tau$ for the pair $(E, \tilde{E})$, 
that is, 
\begin{equation}\label{def-tau-1}
\tau_{1} : =  (\tilde{\pi}_{N})_{*} \circ  (F_{1})_{\mathbb{S}} \circ \pi_{N}^{*} : \Gamma ( \mathbb{S}_{M})\rightarrow \Gamma (\mathbb{S}_{ \tilde{M} }),  
\end{equation}
where
\begin{align*}
\nonumber& \pi_{N}^{*} : \Gamma (\mathbb{S}_{M})\rightarrow \Gamma  (\mathbb{S}_{N}) =
\Omega (N,  \Pi^{*}\mathcal S_{B})\\
\nonumber& (\tilde{\pi}_{N})_{*} :\Gamma (\tilde{\mathbb{S}}_{{N}})= \Omega (N,  \Pi^{*}\tilde{\mathcal S}_{B})\rightarrow \Gamma (\mathbb{S}_{\tilde{M}})
\end{align*}
are the pullback and pushforward maps
(\ref{pull-back-forms})  and (\ref{push-forward-spinors}), 
is related to $\tau$  by $\tau_{1}= \epsilon  \tilde{I}_{\mathbb{S}} \circ \tau\circ ( I_{\mathbb{S}})^{-1}$. 
Here $\epsilon \in \{ \pm 1\} $ and 
$I_{\mathbb{S}} : \mathbb{S}_{E} \rightarrow  \mathbb{S}_{M}$, 
$\tilde{I}_{\mathbb{S}} : \mathbb{S}_{\tilde{E}} \rightarrow  \mathbb{S}_{\tilde{M}}$ 
are induced  by $I$, $\tilde{I}$.
As $I$ and $\tilde{I}$ are invariant, $I_{\mathbb{S}}$ and $\tilde{I}_{\mathbb{S}}$ 
map invariant spinors to invariant spinors. This is true also for $\tau .$
We obtain that  $\tau_{1}$ maps invariant spinors to invariant spinors. 
Define
$$
\rho_{1}: \Gamma_{\mathfrak{t}^{k}} (E_{M}) \rightarrow \Gamma_{\mathfrak{t}^{k}} (\tilde{E}_{\tilde{M} }),\
\rho_{1}:= \tilde{I} \circ \rho\circ I^{-1}.
 $$
Since  $(\tau , \rho )$ satisfy  (\ref{properties-rho}), so do  $(\tau_{1}, \rho_{1})$. 
\end{proof}

\begin{lem} The statements from Theorem \ref{thm-T-duality} hold for 
the pair $(E_{M}, \tilde{E}_{\tilde{M}})$ and 
canonical weighted spinor bundles $\mathbb{S}_{M}$ and $\mathbb{S}_{\tilde{M}}$.
\end{lem}

\begin{proof} 
i) We prove that the map $\tau_{1}$ 
defined by   (\ref{def-tau-1}) maps invariant spinors to invariant spinors. 
We  denote by $\nabla^{\Psi}$ 
and $\nabla^{\tilde{\Psi }}$ the partial connections defined by the actions 
on the standard Courant algebroids $E_{M}$ and  $\tilde{E}_{\tilde{M}}$. For the various other  partial connections
(like  $\nabla^{\Psi , \pi_{N}^{!} E_{M}}$,  $\nabla^{\Psi ,\pi^{*}  \mathcal S_{B}}$, 
$\nabla^{ \Psi ,\Pi^{*} {\mathcal S}_{B}}$ etc) 
we use the definition   explained 
after Lemma \ref{extensie-action}
 (keeping in mind that $S_{\mathcal G} = \pi^{*}S_{B}$ and
$\pi_{N}^{*} S_{\mathcal G}  = \Pi^{*} S_{B}$).
We prove  that if  $\omega \otimes s\in \Gamma (\mathbb{S}_{M})$ 
is $\mathfrak{t}^{k}$-invariant then 
$\pi_{N}^{*} ( \omega \otimes s)\in \Gamma ( \mathbb{S}_{N})$ 
is $\mathfrak{t}^{2k}$-invariant. The $\mathfrak{t}^{k}$-invariance
of $\pi_{N}^{*} ( \omega \otimes s)$  follows from 
Proposition \ref{p-back-forward-inv} ii).  In order to prove that $\pi_{N}^{*} ( \omega \otimes s)$ 
is $\tilde{\mathfrak{t}}^{k}$-invariant, we apply the formula  
\begin{equation}
\Psi^{\mathbb{S}_{ N}} (a) \pi_{N}^{*} (\omega \otimes s) = {\mathcal L}_{X_{a}^{N}}( \pi_{N}^{*}\omega )\otimes
\pi_{N}^{*} \omega + (\pi_{N}^{*}\omega )\otimes  \nabla^{ \Psi ,\Pi^{*} {\mathcal S}_{B}}(\pi_{N}^{*}s).
\end{equation}
If $a\in \tilde{\mathfrak{t}}^{k}$ then 
\begin{equation}\label{i1}
{\mathcal L}_{X_{a}^{N}}( \pi_{N}^{*} \omega )= \pi_{N}^{*} {\mathcal L}_{ (\pi_{N})_{*} X_{a}^{N}}\omega=0, 
\end{equation}
since  $(\pi_{N})_{*} X_{a}^{N}=0$.  On the other hand, from 
(\ref{more-precisely}), 
we deduce that 
$$
\nabla^{\Psi , \Pi^{*}  \mathcal S_{B}}_{X_{a}^{N}}( \pi_{N}^{*}s)= \pi_{N}^{*} (\nabla^{\Psi , \pi^{*}  \mathcal S_{B}}_{(\pi_{N})_{*} X_{a}^{N}} r) =0 
$$
by using again   $(\pi_{N})_{*} X_{a}^{N}=0$.
It follows  that $\pi_{N}^{*} (\omega\otimes s)$ is $\tilde{\mathfrak{t}}^{k}$-invariant.
We have proven that  $\pi_{N}^{*}(\omega \otimes s)$ is 
$\mathfrak{t}^{2k}$-invariant. 
From Lemma  \ref{de-adaugat},  $(F_{1})_{\mathbb{S}} \pi_{N}^{*} (\omega \otimes s)$ 
is   $\mathfrak{t}^{2k}$-invariant (in particular, $\tilde{\mathfrak{t}}^{k}$-invariant) and  from 
Proposition  \ref{p-back-forward-inv} ii) we obtain that  $\tau_{1} (\omega \otimes s)$ 
is $\tilde{\mathfrak{t}}^{k}$-invariant, as needed.\

ii)  Let $u= \xi + r + X\in \Gamma_{\mathfrak{t}^{k}} (E_{M})$.  Then 
$X$ and $\xi$ are invariant with respect to the standard (by Lie derivatives) action of $\mathfrak{t}^{k}$
on $M$  and  $r$ is $\nabla^{\Psi }$-parallel.
We obtain 
\begin{equation}\label{invariance-pull-part}
{\mathcal L}_{X_{a}^{N}}( \pi_{N}^{*} \xi )=0,\ \forall a\in \mathfrak{t}^{2k},\ \nabla^{\Psi ,  \pi_{N}^{!} E_{M}} (\pi_{N}^{*} r)=
0,
\end{equation}
where in the second relation (\ref{invariance-pull-part}) we used (\ref{pull-back-extension}).
We claim that there is a unique $\mathfrak{t}^{2k}$-invariant  lift $\widehat{X}_{0}\in {\mathfrak X}_{\mathfrak{t}^{2k}} (N)$ 
of $X$ with the property that  
\begin{equation}\label{nn}
\mathrm{pr}_{T^{*}N} F_{1}( \pi_{N}^{*}( \xi + r) + \widehat{X}_{0}) = \tilde{\pi}_{N}^{*} (\tilde{\xi})
\end{equation}
for  an (invariant)  $1$-form $\tilde{\xi }\in \Omega^{1}( \tilde{M}).$ To prove the claim we assume that 
the isomorphism $F_{1}$ is defined by data $( \beta  , K, \Phi )$ as in Section \ref{trans-sect-basic},
where $\beta  \in \Omega^{2}(N)$, $K\in \mathrm{Isom} (\Pi^{*}\mathcal G_{B}, \Pi^{*}\tilde{\mathcal G}_{B} )$
and $\Phi \in \Omega^{1}( N,  \Pi^{*}\tilde{\mathcal G}_{B}  ).$ 
Let $\widehat{X}$ be an arbitrary  $\mathfrak{t}^{2k}$-invariant lift of $X$.
Then 
\begin{align}
\nonumber& \mathrm{pr}_{TN} F_{1}( \pi_{N}^{*} ( \xi + r) +\widehat{X}) = \widehat{X},\
\mathrm{pr}_{\tilde{\pi}_{N}^{*} \tilde{\mathcal G }} F_{1}( \pi_{N}^{*} ( \xi + r) +\widehat{X}) =  K (\pi^*_Nr) +\Phi (\widehat{X})\\
\label{expr-F-proj}&  \mathrm{pr}_{T^{*}N} F_{1}( \pi_{N}^{*} ( \xi + r) +\widehat{X}) = \pi_{N}^{*} (\xi ) - 2 \Phi^{*} K( \pi_{N}^{*} r) + i_{\widehat{X}} \beta - \Phi^{*} \Phi (\widehat{X}).
\end{align}
From (\ref{invariance-pull-part}) $\pi_{N}^{*} (\xi + r)$ is $\mathfrak{t}^{2k}$-invariant
and we obtain that $F_{1}( \pi_{N}^{*} (\xi + r) +\widehat{X})$ is also $\mathfrak{t}^{2k}$-invariant.
The non-degeneracy of (\ref{pairing-non-deg})
together with the $\mathfrak{t}^{2k}$-invariance of  $\mathrm{pr}_{T^{*}N} F_{1}( \pi_{N}^{*} ( \xi + r) +\widehat{X}) $
and the last formula (\ref{expr-F-proj}) 
imply that there is a unique $\mathfrak{t}^{2k}$- invariant lift $\widehat{X}_{0}$ of $X$ such that 
$\mathrm{pr}_{T^{*}N} F_{1}( \pi_{N}^{*} ( \xi + r) +\widehat{X}_{0}) $
is horizontal with respect to $\tilde{\pi}_{N}$ and invariant, hence basic.  For this  lift
relation (\ref{nn}) holds.

On the other hand, since  $F_{1}( \pi_{N}^{*} (\xi + r) +\widehat{X}_{0})$ is  $\mathfrak{t}^{2k}$-invariant,
its projection to $\Pi^{*}\tilde{\mathcal G}_{B}$ is $\nabla^{\tilde{\Psi}, \tilde{\pi}_{N}^{!}\tilde{E}_{\tilde{M}}}$-parallel
and is therefore the pullback of a $\nabla^{\tilde{\Psi}}$-parallel section $\tilde{r}$ of
$\tilde{\pi}^{*} {\mathcal G}_{B} .$ To summarize,
\begin{equation}\label{FS}
F_{1}(  \pi_{N}^{*} (\xi + r) +\widehat{X}_{0})=\tilde{\pi}_{N}^{*} (\tilde{\xi} + \tilde{r}) + \widehat{X}_{0}
\end{equation} 
where $\tilde{\xi}\in \Omega^{1}(\tilde{M})$ is $\tilde{\mathfrak{t}}^{k}$-invariant
and $\nabla^{\tilde{\Psi}}( \tilde{r})=0.$
We define
\begin{equation}
\rho_{1} ( u):= \tilde{\xi} + \tilde{r} + (\tilde{\pi}_{N})_{*} \widehat{X}_{0}.
\end{equation}
Obviously, $\rho_{1} (u)$ is $\tilde{\mathfrak{t}}^{k}$-invariant  and the resulting map 
$\rho_{1} : \Gamma_{\mathfrak{t}^{k}} (E_{M}) \rightarrow  \Gamma_{\tilde{t}^{k}}(\tilde{E}_{\tilde{M}})$ is
$C^{\infty}(B)$-linear. 
It remains to prove 
that $(\rho_{1}, \tau_{1})$  satisfy relations   (\ref{properties-rho}).
In order to prove the first relation (\ref{properties-rho}), let
$u:= \xi + r + X\in \Gamma_{\mathfrak{t}^{k}} (E_{M})$, $\rho_{1} (u)= \tilde{\xi} + \tilde{r} +(\tilde{\pi}_{N})_{*} \widehat{X}_{0}$ constructed as above, 
$s\in \Gamma_{\mathfrak{t}^{k}} (\mathbb{S}_{M})$ and $\sigma \in \Gamma (\mathbb{S}_{\tilde{N}})$. 
Using (\ref{pull-compat-cliff}) and Remark \ref{later}, 
\begin{equation}
\pi_{N}^{*} \gamma_{u} (s) = \gamma_{\pi_{N}^{*} (\xi + r) + \widehat{X}_{0}} \pi_{N}^{*} s,\ 
(\tilde{\pi}_{N})_{*} \gamma_{\tilde{\pi}_{N}^{*}(\tilde{\xi} + \tilde{r})+ \widehat{X}_0}\sigma  = \gamma_{\rho_{1} (u)} (\tilde{\pi}_{N})_{*} \sigma ,    
\end{equation}
and  we write
\begin{eqnarray}
 \nonumber \tau_{1} \gamma_{u} (s) &=& (\tilde{\pi}_{N})_{*}( F_{1})_{\mathbb{S}} ( \pi_{N})^{*}\gamma_{u}(s) =
(\tilde{\pi}_{N})_{*} (F_{1})_{\mathbb{S}}  \gamma_{\pi_{N}^{*} (\xi + r) + \widehat{X}_{0}} \pi_{N}^{*} s\\
\nonumber  &=&
(\tilde{\pi}_{N})_{*}  \gamma_{\tilde{\pi}_{N}^{*} (\tilde{\xi} + \tilde{r})+ \widehat{X}_0} (F_{1})_{\mathbb{S}} \pi_{N}^{*} (s) 
= \gamma_{\rho_{1} (u)} (\tilde{\pi}_{N})_{*} (F_{1})_{\mathbb{S}}  (\pi_{N})^{*} (s)\\
\nonumber&  =& \gamma_{\rho_{1} (u)} \tau_{1} (s),
\end{eqnarray}
where in the third equality we used  Lemma \ref{iso-bdle}
and relation (\ref{FS}). 
The first relation of (\ref{properties-rho}) is proved.  The second relation of (\ref{properties-rho}) follows from the next computation, which uses the first relation of (\ref{properties-rho}) 
together with $\tau_{1}\circ \slashed{d}_{M} = \slashed{d}_{\tilde{M}}\circ \tau_{1}$
 proved in part i) of Theorem \ref{thm-T-duality}
(where $\slashed{d}_{M}$ and $\slashed{d}_{\tilde{M}}$ are the Dirac generating operators
of $E_{M}$ and $E_{\tilde{M}}$ acting on $\Gamma (\mathbb{S}_{M})$ and $\Gamma (\mathbb{S}_{\tilde{M}})$ respectively).
For any $u, v\in \Gamma_{\mathfrak{t}^{k}} (E_{M})$ and
$s\in \Gamma_{\mathfrak{t}^{k}} (\mathbb{S}_M)$, we have 
\begin{align*}
\nonumber\gamma_{\rho_{1} [u, v]_{E_{M}}} \tau_{1} (s)& = \tau_{1} \gamma_{[u, v]_{E_{M}}}( s) = \tau_{1} [[\slashed{d}_{M}, \gamma_{u}], \gamma_{v} ] (s)
= [[\slashed{d}_{\tilde{M}} ,\gamma_{\rho_{1} (u) }], \gamma_{\rho_{1} (v)}] \tau_{1} (s)\\
\nonumber& =\gamma_{[\rho_{1} (u), \rho_{1} (v) ]_{\tilde{E}_{\tilde{M}}}}\tau_{1} (s) .
\end{align*}
In order to prove the third relation of (\ref{properties-rho}), we remark that for any  $u\in \Gamma_{{\mathfrak{t}^k}} (E_{M})$,
$\langle u, u\rangle_{E_{M}}$ is  $\mathfrak{t}^{k}$-invariant and hence is the pullback of a function $g\in C^{\infty}(B).$ 
Similarly, $\langle \rho_{1} (u), \rho_{1} (u)\rangle_{\tilde{E}_{\tilde{M}}}$ is the pullback of a function 
$\tilde{g}\in C^{\infty}(B).$ We need to show that $g = \tilde{g}.$ This follows from the next computation which uses the first relation of (\ref{properties-rho}):
\begin{align*}
\nonumber \tilde{\pi}^{*} (g)  \tau_{1} (s)& = \tau_{1} (\pi^{*} (g) s) =\tau_{1} ( \langle u, u\rangle_{E_{M}} s)=
\tau_{1} \gamma_{u}^{2} (s) = \gamma_{\rho_{1} (u)}^{2} \tau_{1} (s)\\
& = \langle \rho_{1} (u), \rho_{1} (u)\rangle_{\tilde{E}_{\tilde{M}}} \tau_{1} (s) = \tilde{\pi}^{*} (\tilde{g}) \tau_{1} (s).
\end{align*}  
From the third relation of (\ref{properties-rho}) we obtain that $\rho_{1}$ is an isomorphism (of vector spaces and even of $C^\infty (B)$-modules).
The proof of the theorem is completed.
\end{proof}

\begin{cor} The  map 
\begin{equation}\label{def-tau-invariant}
\tau : =  (\tilde{\pi}_{N})_{!} \circ  F_{\mathbb{S}} \circ \pi_{N}^{!} : \Gamma_{\mathfrak{t}^{k}} ( \mathbb{S}_{E})\rightarrow
 \Gamma_{\tilde{\mathfrak{t}}^{k}} (\mathbb{S}_{\tilde{E}})
\end{equation}
is an isomorphism of $C^\infty (B)$-modules.
\end{cor}

\begin{proof} This follows from the irreducibility of the spinor bundles together with the fact that $\tau$
is $C^{\infty}(B)$-linear,    is not the zero map and  intertwines
the Clifford multiplications in the commutative diagram 
 \begin{equation*} \begin{xy}
    \xymatrix{ 
 \Gamma_{\mathfrak{t}^{k}}(E) \times \Gamma_{\mathfrak{t}^{k}}(\mathbb{S}_E)\ar@{->}[r] \ar@{->}[d]_{\rho \times \tau}
 &  
 \Gamma_{\mathfrak{t}^{k}}(\mathbb{S}_{\tilde{E}})\ar@{->}[d]^{\tau}
 \\
 \Gamma_{\tilde{\mathfrak{t}}^{k}}(\tilde{E})\times  \Gamma_{\tilde{\mathfrak{t}}^{k}}(\mathbb{S}_{\tilde{E}})\ar@{->}[r]& \Gamma_{\tilde{\mathfrak{t}}^{k}}(\mathbb{S}_{\tilde{E}}).}
    \end{xy}
\end{equation*}
\end{proof}

\begin{rem}{\rm 
As in the $T$-duality for exact or heterotic Courant algebroids, the map $\rho$ 
constructed in Theorem 
\ref{thm-T-duality}
can be interpreted as
an isomorphism between Courant algebroids over $B$ (see  \cite{baraglia} and \cite{T-duality-exact}).
}
\end{rem}

In the next remark we  discuss Theorem
\ref{thm-T-duality} without the assumption that $F_{\mathbb{S}}$ is globally defined.

\begin{rem}\label{global-remark2} {\rm i) 
We claim  that 
the isomorphism $(F_{1})_{\mathbb{S}}$  introduced before Remark \ref{F1:rem}  (hence also $F_{\mathbb{S}}$) 
is always defined on  subsets of $N$ of the form
$\Pi^{-1}(V)$, where $V\subset B$ is open  and sufficiently small.  Indeed,
$(F_{1})_{\mathbb{S}\vert_{ U}}\in \mathrm{Isom}
(\mathbb{S}_{N}\vert_{U}, \tilde{\mathbb{S}}_{N}\vert_{U})$
 is defined, whenever  $U\subset N$ is open and sufficiently small (see Lemma \ref{iso-bdle}).
Letting  $V: = \Pi (U)$, we can find (reducing $V$ is necessary)
invariant frames $(s_{i})$ and $(\tilde{s}_{i})$
of  $\mathbb{S}_{N}$ and  $\mathbb{S}_{\tilde{N}}$ 
on $\Pi^{-1} (V).$ 
With respect to these frames, $(F_{1})_{\mathbb{S}\vert_{U}}$ is given by
\begin{equation}
(F_{1})_{\mathbb{S}\vert_{U}} (s_{i}) =\sum_{j}  C_{ji} \tilde{s}_{j}
\end{equation}
for some functions $C_{ji} \in C^{\infty}( U)$.  From  the invariance of  $(F_{1})_{\mathbb{S}\vert_{U}}$, we deduce that 
${\mathcal L}_{X_{a}^{N}} C_{ji} =0$ for any $a\in \mathfrak{t}^{2k}$, i.e.\   
$C_{ji} = \Pi_{0}^{*} (c_{ji})$ where $c_{ji}\in C^{\infty}(V)$ and $\Pi_{0} : U \rightarrow V$
is the restriction of $\Pi .$ Using that  all $\pi_{N}^{!} E_{M}$,  $\tilde{\pi}_{N}^{!} \tilde{E}_{\tilde{M}}$, 
 $\mathbb{S}_{N}$ and  $ \mathbb{S}_{\tilde{N}}$ 
admit invariant frames on $\Pi^{-1}(V)$, 
we deduce from the compatibility of $(F_{1})_{\mathbb{S}\vert_{U}}$  with $F_{1}\vert_{U}$ that 
\begin{equation}
(F_{1})_{\mathbb{S}\vert_{\Pi^{-1} (V)}} (s_{i}):= \sum_{j}  \Pi^{*} (c_{ji}) \tilde{s}_{j}, 
\end{equation} 
defined on $\Pi^{-1} (V)$,  is compatible with $F_{1}\vert_{\Pi^{-1} (V)}.$\

ii) From the above, the map 
$$
\tau_{V}: = (\tilde{\pi}_{N})_{!}\circ F_{\mathbb{S}}\circ 
\pi_{N}^{!}  : \Gamma ( \mathbb{S}_{E}\vert_{\pi^{-1}  (V)})
\rightarrow \Gamma ( \mathbb{S}_{\tilde{E}}\vert_{\tilde{\pi}^{-1} (V)})
$$
is defined.  Theorem 
\ref{thm-T-duality} still holds,   
the only difference being that
$\tau$ is replaced by the locally defined maps $\tau_{V}$,    for any $V\subset B$  open and sufficiently small. (The isomorphism $\rho$ remains defined
globally.)
}
\end{rem}

\subsection{Existence of a $T$-dual}

Let $\pi : M \rightarrow B$ be a principal $T^{k}$-bundle and $\mathcal H$ a principal
connection on $\pi$, with connection form $\theta = \sum_{i=1}^{k}\theta_{i} e_{i}\in \Omega^{1} (M, \mathfrak{t}^{k})$,
where $(e_{i})$ is a basis of $\mathfrak{t}^{k}$
 such that $T^{k} = \mathfrak{t}^{k}/\mathrm{span}_{\mathbb{Z}} \{ e_{i}\} .$ 
Let $(E, \Psi) $ be a standard  Courant algebroid with an action 
$\Psi :\mathfrak{t}^{k}\rightarrow \mathrm{Der} (E)$ 
which lifts the vertical paralellism of $\pi$,  defined by a quadratic Lie algebra bundle $(\mathcal G_{B}, [\cdot , \cdot ]_{\mathcal G_{B}},\langle\cdot ,
\cdot \rangle_{\mathcal G_{B}})$ whose adjoint representation is an isomorphism, a 
connection $\nabla^{B}$ on $\mathcal G_{B}$ which preserves $[\cdot , \cdot ]_{\mathcal G_{B}}$
and $\langle \cdot , \cdot \rangle_{\mathcal G_{B}}$, 
a $3$-form
$H_{(3)}^{B}$, $2$-forms $H_{(2)}^{i, B}$ and sections $r_{i}^{B}\in \Gamma (\mathcal G_{B})$ 
($1\leq i\leq k$)  as 
in  Example \ref{particular-class}
(see also Corollary \ref{data-extended}).  We denote by $(e^{i})$ the dual basis of $(e_{i})$
and by $\tilde{T}^{k}=(\mathfrak{t}^{k})^{*}/ \mathrm{span}_{\mathbb{Z}}\{ e^{i}\}$
the dual torus.

\begin{thm}\label{statement-T-dual} Assume that the closed forms
$\mathcal K_{i}$ defined by 
 (\ref{k-i}) represent   integral cohomology classes 
in  $H^{2} (B, \mathbb{R})$ and let   $\tilde{\pi} : \tilde{M} \rightarrow B$ be a principal
$\tilde{T}^{k}$-bundle with connection form $\tilde{\theta} =\sum_{i=1}^{k} \tilde{\theta}_{i} e^{i}$,  such
that $(d\tilde{\theta}_{i})^{B}=\mathcal K_{i}$ for  any $i$.  
Then $E$ admits a $T$-dual $\tilde{E}$,  defined on 
$\tilde{M}$ and  
\begin{equation}\label{chern-classes}
\left[\sum_{i=1}^{k}(d\theta_{i})^{B}\wedge (d \tilde{\theta}_{i} )^{B}\right]= [ \langle {\mathfrak r}^{B}\wedge \mathfrak{r}^{B} \rangle_{\mathcal G_{B}} ]
\in H^{4} (B, \mathbb{R}).
\end{equation} 

\end{thm}

\begin{proof}
From the expression (\ref{k-i}) of ${\mathcal K}_{i}=(d\tilde{\theta}_{i})^{B}$, we have 
\begin{equation}
(d\tilde{\theta}_{i})^{B} =   H_{(2)}^{i, B} + 2\langle {\mathfrak r}^{ B}, r_{i}^{B}\rangle_{\mathcal G_{B}} -
 \langle r_{i}^{B}, r_{j}^{B}\rangle_{\mathcal G_{B}} (d\theta_{j})^{B}.
\end{equation}
We consider the data formed by the quadratic Lie algebra  bundle   
$$
 (\tilde{\mathcal G}_{B}, [\cdot , \cdot ]^{\tilde{} }_{\mathcal G_{B}},\langle \cdot , \cdot \rangle_{\mathcal G_{B}}^{\tilde{}}):=
(\mathcal G_{B}, [\cdot , \cdot ]_{\mathcal G_{B}},\langle \cdot , \cdot \rangle_{\mathcal G_{B}}),
$$ 
connection $\tilde{\nabla}^{ B}  := \nabla^{B}$, 
sections $\tilde{r}_{i}^{B}\in \Gamma ({\mathcal G}_{B})$ (arbitrarily chosen),  $3$-form $\tilde{H}_{(3)}^{B}:= H_{(3)}^{B}$
and $2$-forms 
\begin{equation}\label{tilde-h2} 
\tilde{H}_{(2)}^{i, B}:= (d\theta_{i} )^{B}- 2\langle {\mathfrak r}^{B}, \tilde{r}^{B}_{i}\rangle_{\mathcal G_{B}} 
+\langle \tilde{r}_{i}^{B}, \tilde{r}_{j}^{B} \rangle_{\mathcal G_{B}} (d\tilde{\theta}_{j})^{B}.
\end{equation}
From (\ref{tilde-h2}),
the $2$-form 
$$
\tilde{\mathcal{K}}_i := \tilde{H}_{(2)}^{i, B}+  2\langle {\mathfrak r}^{B}, \tilde{r}^{B}_{i}\rangle_{\mathcal G_{B}} 
-\langle \tilde{r}_{i}^{B}, \tilde{r}_{j}^{B} \rangle_{\mathcal G_{B}} (d\tilde{\theta}_{j})^{B} = (d\theta_{i})^{B}
$$
is closed. Since 
$$
d\tilde{H}_{(3)}= dH_{(3)}   = \langle \mathfrak{r}^{B}\wedge  \mathfrak{r}^{B}\rangle_{\mathcal G_{B}}
- \mathcal{K}_i \wedge (d\theta_{i})^{B}= \langle \mathfrak{r}^{B}\wedge  \mathfrak{r}^{B}\rangle_{\mathcal G_{B}}
- \tilde{\mathcal{K}}_i \wedge (d\tilde{\theta}_{i})^{B},
$$
we obtain from Example \ref{particular-class}  a standard  Courant algebroid  $(\tilde{E}, \tilde{\Psi} )$ together with an action which lifts the vertical paralellism of $\tilde{\pi}$, such that   
\begin{align}
\nonumber& \tilde{H}_{(0)}^{pqs, B} :=- \frac{1}{3} \langle [\tilde{r}_{p}^{B}, \tilde{r}_{q}^{B}]_{\mathcal G_{B}} , \tilde{r}_{s}^{B}\rangle_{\mathcal G_{B}}\\
\nonumber&  \tilde{H}_{(1)}^{ij, B} := \frac{1}{2} \left( \langle \nabla^{B }\tilde{r}_{i}^{B},\tilde{r}_{j}^{B}\rangle_{\mathcal G_{B}}
- \langle \nabla^{B }\tilde{r}_{j}^{B}, \tilde{r}_{i}^{B}\rangle_{\mathcal G_{B}}\right) 
\end{align} 
and the $\mathcal G_{B}$-valued forms $\tilde{R}_{(0)}^{ij, B}$,  $\tilde{R}_{(1)}^{i, B}$ and $\tilde{R}_{(2)}^{B}$  given by the tilde analogue of  (\ref{expr-R}). The quadratic Lie algebra bundles of $E$ and $\tilde{E}$ are the
pullbacks of $(\mathcal G_{B},  [\cdot , \cdot ]_{\mathcal G_{B}}, \langle \cdot , \cdot \rangle_{\mathcal G_{B}})$
and,  as vector bundles with scalar products, 
 $$
\pi_{N}^{!} E = \tilde{\pi}_{N}^{!}\tilde{E} = T^{*}N \oplus \Pi^{*}\mathcal G_{B} \oplus TN,
$$
where the scalar products are given by (\ref{scalar-products-pull}) and  $\Pi = \pi \circ\pi_{N} = 
\tilde{\pi}\circ \tilde{\pi}_{N}$.

We claim  that  $E$ and $\tilde{E}$ are $T$-dual, i.e.\ not only that 
$\pi_{N}^{!} E$ and $\tilde{\pi}_{N}^{!} \tilde{E}$ are isomorphic 
as Courant algebroids but  that  one can choose the isomorphism to be invariant and 
such that the non-degeneracy condition (\ref{pairing-non-deg})  is satisfied.
Such an  isomorphism $F: \pi_{N}^{!} E \rightarrow \tilde{\pi}_{N}^{!} \tilde{E}$ (if it exists)  is given by a triple  $(\beta , K, \Phi )$, where $\beta \in \Omega^{2} (N)$,  $\Phi \in  \Omega^{1}(N,  \Pi^{*}\mathcal G_{B})$ and
$K = \Pi^{*} K_{B}$ where $K_{B}\in \mathrm{Aut} (\mathcal G_{B})$ is a quadratic Lie algebra bundle automorphism
(see the proof of Lemma  \ref{schimbare-disectie}). 
Let 
\begin{align*}
\nabla^{\theta} := \nabla^{E}  + \theta_{i}\otimes \mathrm{ad}_{r_{i}} =\pi^{*}\nabla^{B}\\
\nabla^{\tilde{\theta}}:= {\nabla}^{\tilde{E}} +\tilde{\theta}_{i}\otimes \mathrm{ad}_{\tilde{r}_{i}}=\tilde{\pi}^{*}\nabla^{B}
\end{align*}
be the connections on $E$ and $\tilde{E}$ defined  before Lemma  \ref{basictohor:lem},
where $r_{i} := \pi^{*}(r_{i}^{B})$, $\tilde{r}_{i}:= \tilde{\pi}^{*}(\tilde{ r}_{i}^{B})$ and to simplify notation
we continue to  omit the summation sign and we denote by the same symbol  `$\mathrm{ad}$' the adjoint action in the  Lie algebra bundles
$\mathcal G_{B}$, $\mathcal G$, 
$\tilde{\mathcal G}$ or their pullbacks to $N$. Then
\begin{align}
\nonumber& \tilde{\pi}_{N}^{*} \nabla^{\tilde{E}} = \Pi^{*} \nabla^{B} -  (\tilde{\pi}_{N}^{*} \tilde{\theta}_{i})
\otimes\Pi^{*} ( \mathrm{ad}_{\tilde{r}_{i}^{B}})\\
\label{B-con} & {\pi}_{N}^{*} \nabla^{{E}} = \Pi^{*} \nabla^{B} - ({\pi}_{N}^{*} {\theta}_{i})
\otimes\Pi^{*} ( \mathrm{ad}_{{r}_{i}^{B}}).
\end{align}
With these preliminary remarks, we now consider separately the relations from Lemma 
\ref{iso-upstairs} and we look for $(\beta , K = \Pi^{*} K_{B}, \Phi )$ such that these relations are satisfied.
Relation 
(\ref{con-T})
can be written in the equivalent way
\begin{equation}\label{con-T-alt}
\Pi^{*}\left(  K_{B}  (\nabla^{B}) K_{B}^{-1} - \nabla^{B} \right) =
(\pi_{N}^{*} \theta_{i} )\otimes \Pi^{*} ( \mathrm{ad}_{K_{B}(r_{i}^{B})}) - (\tilde{\pi}_{N}^{*}\tilde{\theta}_{i})\otimes  (\Pi^{*} \mathrm{ad}_{\tilde{r}_{i}^{B}})+\mathrm{ad}\circ \Phi .
\end{equation}
Letting
\begin{equation}\label{K-P}
K_{B}:=\mathrm{Id}_{\mathcal G_{B}},\ \Phi :=  (\tilde{\pi}_{N}^{*} \tilde{\theta}_{i})
\otimes \Pi^{*} (\tilde{r}_{i}^{B} ) - (\pi_{N}^{*} \theta_{i}) \otimes  \Pi^{*} (r_{i}^{B} ),
\end{equation}
relation  (\ref{con-T-alt}) is obviously satisfied. 
Relation (\ref{tilde-T}) is 
automatically satisfied from Lemma \ref{simplificare-second} and our hypothesis that the adjoint representation of
$(\mathcal G_{B}, [\cdot , \cdot ]_{\mathcal G_{B}})$ is an isomorphism. 
It remains to 
find an invariant $2$-form $\beta \in \Omega^{2}(N)$ such that relation
(\ref{H-tH}) is satisfied.
Now, a straightforward computation which uses the definition of  $\Phi$ 
shows that the $3$-form 
$$
c_{3}(X, Y, Z) :=\langle \Phi (X), [\Phi (Y), \Phi (Z) ]_{\Pi^{*}\mathcal G_{B}}\rangle_{\Pi^{*}\mathcal G_{B}},\
\forall X, Y, Z\in {\mathfrak X}(N)
$$
is given by 
\begin{align}
\nonumber&c_{3}=  \frac{1}{6}\left(\langle [\tilde{ r}_{s}^{B}, \tilde{r}_{i}^{B}]_{\mathcal G_{B}},\tilde{r}_{j}^{B}\rangle
\tilde{\theta}_{s}\wedge \tilde{\theta}_{i}\wedge\tilde{\theta}_{j}
-\langle [{ r}_{s}^{B}, {r}_{i}^{B}]_{\mathcal G_{B}},{r}_{j}^{B} \rangle{\theta}_{s}\wedge {\theta}_{i}\wedge{\theta}_{j}  \right)\\
\label{tilde-h3}&+ \frac{1}{2} \left( \langle [ r_{i}^{B}, r_{j}^{B}]_{\mathcal G_{B}}, \tilde{r}^{B}_{s}\rangle
\tilde{\theta}_{s}\wedge \theta_{i}\wedge \theta_{j} -\langle [ \tilde{r}_{j}^{B},\tilde{ r}_{s}^{B}]_{\mathcal G_{B}}, {r}_{i}^{B}\rangle 
{\theta}_{i} \wedge\tilde{\theta}_{j}\wedge \tilde{\theta}_{s} \right) ,
\end{align}
where we identify forms on $M$, $\tilde{M}$ or $B$ with their pullback to $N$   (we omit the pullback signs) 
and we denote $\langle \cdot , \cdot\rangle_{\mathcal G_{B}}$ by $\langle \cdot , \cdot \rangle$ for simplicity. 
On the other hand, 
\begin{align}
\nonumber& \pi_{N}^{*} R = R_{(2)}^{B} + \theta_{i}  \wedge  R_{(1)}^{i, B}
+ R_{(0)}^{ij,B}\otimes  (\theta_{i}\wedge \theta_{j}),\\
\label{expr-r-r} & \tilde{\pi}_{N}^{*}\tilde{R} = \tilde{R}_{(2)}^{B} +
 \tilde{\theta}_{i}  \wedge \tilde{R}_{(1)}^{i, B}+ \tilde{R}_{(0)}^{ ij, B}
\otimes  (\tilde{\theta}_{i}\wedge \tilde{\theta}_{j}),
\end{align}
where we recall that
\begin{equation}\label{r-ijk}
R_{(0)}^{ij, B} = \frac{1}{2} [r_{i}^{B}, r_{j}^{B}]_{\mathcal G_{B}},\   R_{(1)}^{i, B} 
=\nabla^{B}r_{i}^{B},\ R_{(2)}^{B} ={\mathfrak r}^{B} - d\theta_{i}\otimes r_{i}^{B} 
\end{equation}
and  similarly for  $\tilde{R}_{(0)}^{ij, B}$,   $\tilde{R}_{(1)}^{i, B}$ and  $\tilde{R}_{(2)}^{B}$,
with  $r_{i}^{B}$ replaced by $\tilde{r}_{i}^{B}$ and $\theta_{i}$ replaced by $\tilde{\theta}_{i}.$
From (\ref{expr-r-r}) and (\ref{r-ijk}) we obtain that
\begin{align}
\nonumber& \langle ( \pi_{N}^{*} R + \tilde{\pi}_{N}^{*} \tilde{R})\wedge \Phi  \rangle_{\Pi^*\mathcal{G}_B} =
 - \theta_{i}\wedge \tilde{\theta}_{j} \wedge d \langle r_{i}^{B}, \tilde{r}_{j}^{B}\rangle\\
\nonumber&  + \langle ( 2 {\mathfrak r}^{B} - (d\theta_{i} )\otimes r_{i}^{B} - (d\tilde{\theta}_{i}) \otimes
\tilde{r}_{i}^{B}) \wedge \tilde{r}_{j}\rangle \wedge \tilde{\theta}_{j}\\   
\nonumber& -   \langle ( 2 {\mathfrak r}^{B} - (d\theta_{i} )\otimes r_{i}^{B} - (d\tilde{\theta}_{i}) \otimes
\tilde{r}_{i}^{B}) \wedge {r}_{j}\rangle \wedge {\theta}_{j}\\   
\nonumber&+ \frac{1}{2}\left(  \langle [r_{i}^{B}, r_{j}^{B}], \tilde{r}^{B}_{p}\rangle \theta_{i}\wedge \theta_{j}
\wedge \tilde{\theta}_{p}-   \langle [\tilde{r}_{i}^{B},\tilde{ r}_{j}^{B}], {r}^{B}_{p}\rangle \tilde{\theta}_{i}\wedge\tilde{ \theta}_{j}
\wedge{\theta}_{p}\right)\\ 
\nonumber& + \frac{1}{2}\left(  \langle [\tilde{r}_{i}^{B}, \tilde{r}_{j}^{B}], \tilde{r}^{B}_{p}\rangle \tilde{\theta}_{i}\wedge 
\tilde{\theta}_{j}\wedge \tilde{\theta}_{p}-   \langle [{r}_{i}^{B},{ r}_{j}^{B}], {r}^{B}_{p}\rangle {\theta}_{i}\wedge{ \theta}_{j}
\wedge {\theta}_{p}\right)\\
\label{R-wedge}& +\langle \nabla^{B} r_{i}^{B}, r_{j}^{B}\rangle \wedge \theta_{i}\wedge \theta_{j}
- \langle \nabla^{B} \tilde{r}_{i}^{B}, \tilde{r}_{j}^{B}\rangle \wedge \tilde{\theta}_{i}\wedge \tilde{\theta}_{j}.
\end{align}
We write the $2$-form $\beta$ as
$$
\beta =\beta_{(2)} +\theta_{i}\wedge \beta_{(1)}^{i} +\tilde{\theta}_{i}\wedge \tilde{\beta}_{(1)}^{i}
+ f_{ij} \theta_{i}\wedge \tilde{\theta}_{j}
$$
where  $\beta_{(2)}$, $\beta_{(1)}^{i}$, $ \tilde{\beta}_{(1)}^{i}$ and $f_{ij}$ are defined on $B$, so that
\begin{align}
\nonumber& d\beta =d\beta_{(2)} +d\theta_{i}\wedge \beta_{(1)}^{i} +  d\tilde{\theta}_{i}\wedge \tilde{\beta}_{(1)}^{i}\\
\nonumber& - \theta_{i}\wedge ( d\beta_{(1)}^{i}  + f_{ij} d\tilde{\theta}_{j})+\tilde{\theta}_{i}\wedge ( - d\tilde{\beta}_{(1)}^{i} + f_{ji} d\theta_{j} )\\
\label{d-gamma} &- df_{ij} \wedge \tilde{\theta}_{j} \wedge \theta_{i}.
\end{align}
Finally, we write,  as in Section \ref{example-t1},
\begin{align}
\nonumber&\pi_{N}^{*} H = H_{(3)}^{B} + \theta_{i}\wedge H_{(2)}^{i, B} +\theta_{i}\wedge \theta_{j}\wedge H_{(1)}^{ij, B}
+ H_{(0)}^{ijs, B} \theta_{i}\wedge \theta_{j}\wedge\theta_{s}\\ 
\label{H-E}&\tilde{\pi}_{N}^{*} \tilde{H} = \tilde{H}_{(3)}^{B} + \tilde{\theta}_{i}\wedge \tilde{H}_{(2)}^{i} +
\tilde{\theta}_{i}\wedge \tilde{\theta}_{j}\wedge \tilde{H}_{(1)}^{ij, B}
+ \tilde{H}_{(0)}^{ijs, B} \tilde{\theta}_{i}\wedge \tilde{\theta}_{j}\wedge\tilde{\theta}_{s}.
\end{align}
Using  the expressions   of 
${H}_{(0)}^{ijs, B}$ and 
 $\tilde{H}_{(0)}^{ijs, B}$,   and 
(\ref{tilde-h3}), 
  (\ref{R-wedge}), (\ref{d-gamma}) and (\ref{H-E}), we obtain that
relation (\ref{H-tH}) reduces to the following relations:
\begin{align}
\nonumber& 
H_{(3)}^{B} -\tilde{H}_{(3)}^{B} -d\beta_{(2)}
- (d\theta_{i})^{B} \wedge \beta_{(1)}^{i} -( d\tilde{\theta}_{i} )^{B}\wedge \tilde{\beta}_{(1)}^{i}=0\\
\nonumber& d\beta_{(1)}^{i}+ f_{ij} (d\tilde{\theta}_{j})^{B}  
= - H_{(2)}^{i, B} +\langle r_{i}^{B}, \tilde{r}_{j}^{B}\rangle (d\tilde{\theta}_{j})^{B} - 2 
\langle {\mathfrak r}^{B}, r_{i}^{B}\rangle +\langle r_{i}^{B}, r_{j}^{B}\rangle  (d\theta_{j} )^{B}\\
\nonumber&  d\tilde{\beta}_{(1)}^{i}- f_{ji}( d\theta_{j})^{B}  
= \tilde{H}_{(2)}^{i, B} - \langle \tilde{r}_{i}^{B}, \tilde{r}_{j}^{B}\rangle  (d\tilde{\theta}_{j})^{B} + 2 
\langle {\mathfrak r}^{B},\tilde{ r}_{i}^{B}\rangle - \langle \tilde{r}_{i}^{B}, r_{j}^{B}\rangle ( d\theta_{j})^{B} \\
\label{final-rel} & d f_{ij} = d\langle \tilde{r}_{j}^{B}, r_{i}^{B}\rangle .
\end{align}
Recall now that $H_{(3)}^{B} = \tilde{H}_{(3)}^{B}$ and 
\begin{align*}
& (d\tilde{\theta}_{i})^{B} = H_{(2)}^{i, B} + 2\langle \mathfrak{r}^{B}, r_{i}^{B}\rangle - \langle r_{i}^{B}, r_{j}^{B}\rangle  (d\theta_{j})^{B}\\
& (d{\theta}_{i} )^{B}= \tilde{H}_{(2)}^{i, B} + 2\langle \mathfrak{r}^{B},\tilde{ r}_{i}^{B}\rangle - \langle\tilde{ r}_{i}^{B}, \tilde{r}_{j}^{B}\rangle ( d\tilde{\theta}_{j})^{B} .
\end{align*}

It follows that $\beta_{(2)}:=0$, $\beta_{(1)}^{i} :=0$, $\tilde{\beta}_{(1)}^{i} :=0$ and
$f_{ij} := \langle r_{i}^{B}, \tilde{r}^{B}_{j} \rangle -\delta_{ij}$ satisfy relations (\ref{final-rel}), 
We obtain that  the $2$-form 
$$
\beta := (\langle r_{i}^{B}, \tilde{r}^{B}_{j} \rangle -\delta_{ij} )\theta_{i}\wedge \tilde{\theta}_{j}
$$
satisfies (\ref{H-tH}).  The existence of $F$ is proved.  It is clear that it is invariant.  The non-degeneracy condition
 (\ref{pairing-non-deg})  is satisfied, since 
\[
\beta ( X_{\tilde{a}}, X_{b}) = -  \langle \tilde{r}^{B}_{a}, r^{B}_{b}\rangle + \delta_{ab},\ 
(\Phi^{*}\Phi )(X_{\tilde{a}}, X_{b}) = -  \langle \tilde{r}_{a}, r_{b}\rangle . \qedhere
\]
\end{proof}

\subsection{Examples of $T$-duality}

In this section we  apply Theorem \ref{statement-T-dual}  to various classes of transitive Courant algebroids. In particular, we recover,
in our setting,  the $T$-duality for exact Courant algebroids \cite{T-duality-exact}
and  the $T$-duality for heterotic Courant algebroids \cite{baraglia}.

 \subsubsection{$T$-duality for exact Courant algebroids}
 \label{exact:sec}

Let  $E= T^{*}M \oplus TM$ 
be an exact Courant algebroid over the total space of a principal $T^{k}$-bundle  
$\pi : M \rightarrow B$,   with 
Dorfmann bracket $[\cdot , \cdot ]_{H}$  twisted by an invariant, closed,  $3$-form $H\in \Omega^{3}(M)$, that is,
\begin{equation}
[\xi + X, \eta + Y]_{H}:=  {\mathcal L}_{X}(Y+\eta )- i_{Y} d\xi + i_{Y}i_{X}H, 
\end{equation}
for any $X, Y\in {\mathfrak X}(M)$, $\xi , \eta\in \Omega^{1}(M)$, 
scalar product  $\langle \xi + X, \eta + Y \rangle := \frac{1}{2} ( \xi (Y) +\eta (X))$
and anchor the natural projection from $E$ to $TM.$ The  action  of $T^{k}$ on $M$ lifts naturally to an action on $E$.   
Assuming that $H\vert_{\Lambda^{2}(\mathrm{Ker}\, \pi )} =0$, we obtain a Courant algebroid  
 of the type described
in  Example \ref{particular-class}. 
Choose  a connection $\mathcal H$ on $\pi$, with connection form $\theta =\sum_{i=1}^{k}\theta_{i}e_{i}$
(where $(e_{i})$ is a basis of $\mathfrak{t}^{k}$)
and write
$$
H = H_{(3)} +\sum_{i=1}^{k} \theta_{i}\wedge H_{(2)}^{i},
$$
where $H_{(3)}$ and $H_{(2)}^{i}$ are basic. 
If $[H]\in H^{3}(M, \mathbb{R})$ is an integral cohomology class, then so is
$[H_{(2)}^{i, B}]\in H^{2}(B, \mathbb{R})$ (for any $i$) and  Theorem \ref{statement-T-dual}  
can be applied. 
We recover the existence of a $T$-dual for exact Courant algebroids, which  was proved in
\cite{BHM}  (see also  Proposition 2.1 of  \cite{T-duality-exact}).

\subsubsection{Heterotic $T$-duality}

Let $G$ be  a compact semi-simple Lie group,  with a fixed  invariant scalar product  of neutral signature  $\langle\cdot , \cdot \rangle_{\mathfrak{g}}$ on $\mathfrak{g}= \mathrm{Lie}(G)$. 
Let $\sigma :P\rightarrow  M$ be a principal $G$-bundle and  $\mathcal H$ a connection on $\sigma .$ 
By definition, the heterotic Courant algebroid  defined by the principal $G$-bundle
$\sigma : P \rightarrow M$, connection $\mathcal H$
and a $3$-form $H\in \Omega^{3}(M)$   is the standard Courant algebroid $E = T^{*}M \oplus \mathcal G \oplus   TM$ 
with  the following properties:

\begin{enumerate}
  \item  $(\mathcal G , [\cdot , \cdot ]_{\mathcal G}, \langle\cdot ,\cdot
\rangle_{\mathcal  G})$, as a quadratic  Lie algebra bundle, is given by the adjoint bundle  $\mathfrak{g}_{P}:= P\times_{\mathrm{Ad}} \mathfrak{g}$. Recall that sections 
$r\in \Gamma (\mathfrak{g}_{P})$ are invariant vertical vector fields on $P$ and can be identified
with functions  $f: P\rightarrow \mathfrak{g}$ which satisfy the equivariance  condition
$f(p g) = \mathrm{Ad}_{g^{-1}} f(p)$ for any $p\in P$ and $g\in G$. 
We shall use the notation $r\equiv f$ to denote this identification.
Since  the Lie bracket
$[\cdot , \cdot ]_{\mathfrak{g}}$ and scalar product 
$\langle \cdot , \cdot \rangle_{\mathfrak{g}}$ of $\mathfrak{g}$ are $\mathrm{Ad}$-invariant, they induce a  Lie bracket 
and a scalar product on
$\mathfrak{g}_{P}$, which make $\mathfrak{g}_{P}$ a quadratic Lie algebra bundle.
The Lie bracket of $\mathfrak{g}_{P}$ so defined coincides with the usual Lie bracket of invariant, vertical vector fields on $P$. 

\item  The connection $\nabla$ which is part of the data $(\nabla , R, H)$ which defines  the standard Courant
algebroid  $E$ is induced 
 by $\mathcal H$,   $R= R^{\mathcal H}\in \Omega^{2}(M ,{\mathfrak g}_{P})$ is the curvature of $\mathcal H$
and  $dH = \langle R^{\mathcal H}\wedge R^{\mathcal H}\rangle_{\mathfrak{g}}$. 
\end{enumerate}

The following proposition describes all invariant scalar products on compact 
semi-simple Lie algebras. A similar description can be given for arbitrary reductive Lie algebras. 
Recall that a semi-simple Lie algebra is called {\cmssl compact}
it it is the Lie algebra of a compact group.
\begin{prop} Let $\mathfrak{g} = \bigoplus_{i=1}^s k_i\mathfrak{g}_i$ be the decomposition of a real 
semi-simple Lie algebra into its simple ideals $\mathfrak{g}_i$, of multiplicity $k_i\ge 1$. Assume that 
$\mathfrak{g}$ is compact (or, more generally, that none of the $\mathfrak{g}_i$ has an invariant complex structure).
Then every invariant scalar product on $\mathfrak{g}$ is of the form 
\begin{equation} \label{scp:eq}\sum B_i \otimes b_i,\end{equation}
where $B_i$ is the Killing form of $\mathfrak{g}_i$ and $b_i$ is a scalar product on $\mathbb{R}^{k_i}$. 
The scalar product (\ref{scp:eq}) is of neutral signature if and only if $\sum (\dim \mathfrak{g}_i)p_i=\sum (\dim \mathfrak{g}_i)q_i$, 
where $(p_i,q_i)$ is the signature of $b_i$. 
\end{prop}
\begin{proof} We compute the space of invariant symmetric bilinear forms on $\mathfrak{g}$ as 
\[ (\mathrm{Sym}^2\mathfrak{g}^*)^\mathfrak{g} = \bigoplus (\mathrm{Sym}^2\mathfrak{g_i}^*)^{\mathfrak{g}_i} \otimes \mathrm{Sym}^2 (\mathbb{R}^{k_i})^* \oplus \bigoplus  (\mathrm{\Lambda}^2\mathfrak{g_i}^*)^{\mathfrak{g}_i} \otimes \mathrm{\Lambda}^2 (\mathbb{R}^{k_i})^*.\]
Since $\mathfrak{g}_i$ is simple and not complex, every invariant bilinear form on $\mathfrak{g}_i$ is a multiple of $B_i$ and the right-hand side reduces to $\bigoplus B_i \otimes \mathrm{Sym}^2 (\mathbb{R}^{k_i})^*$. This implies the first claim. 
The second claim follows from observing that the signature $(p,q)$ of (\ref{scp:eq}) is given by $p= -\sum (\dim \mathfrak{g}_i)p_i$, $q=-\sum (\dim \mathfrak{g}_i)q_i$. 
\end{proof}

Assume that $M$ is  the total space of a principal $T^{k}$-bundle
$\pi : M \rightarrow B$ and  that   $\sigma : P \rightarrow M$ is the pullback of a principal $G$-bundle
$\sigma_{0} : P_{0}\rightarrow B$.  Then, 
for any $m\in M$ and $g\in T^{k}$  there is a natural identification between the fibers $P_{m} := \sigma^{-1}(m)$, 
$P_{mg}:=\sigma^{-1} (mg)$
and $(P_{0})_{\pi (m)}:=\sigma_{0}^{-1} (\pi (m))$,  and the $T^{k}$-action on $M$ lifts naturally 
to an action on $P$ (such that $g$ acts as the identity map between $P_{m}$ and $P_{mg}$ in the above identification).   
We deduce that on  any heterotic Courant algebroid $E = T^{*}M \oplus \mathfrak{g}_{P}\oplus TM$ 
there is an induced action
\begin{equation}\label{Psi-het}
\Psi : \mathfrak{t}^{k}\rightarrow \mathrm{Der} (E),\Psi (a) (\xi + r + X) := {\mathcal L}_{X_{a}^{M}}\xi + 
{\mathcal L}_{X_{a}^{P}} r + {\mathcal L}_{X_{a}^{M}} X
\end{equation}
where $X_{a}^{M}$ and $X_{a}^{P}$ denote the fundamental vector fields of the $T^{k}$-action on $M$ and $P$
defined by $a\in \mathfrak{t}^{k}$ and  $r\in \Gamma ({\mathfrak  g}_{P})$ is viewed as an invariant
vertical vector field on $P$.     
If $r\equiv f$, then  ${\mathcal L}_{X_{a}^{P}} r\equiv X_{a}^{P}(f).$

Following \cite{baraglia}, we shall be interested in heterotic Courant algebroids defined by $\sigma$ and a particular class of connections $\mathcal H := \mathcal H^{\sigma}$ on $\sigma $. More precisely, we
consider a connection $\mathcal{H}^{\pi}$  on the principal $T^{k}$-bundle $\pi : M \rightarrow B$, with connection form $\theta =\sum_{i=1}^{k}\theta_{i}e_{i}\in \Omega^{1}(M, \mathfrak{t}^{k})$
(where $(e_{i})$ is a basis of $\mathfrak{t}^{k}$), 
 a connection ${\mathcal H}^{\sigma_{0}}$ on the principal $G$-bundle $\sigma_{0} : P_{0}\rightarrow B$,
with connection form $A_{0}\in \Omega^{1}(P_{0},
\mathfrak{g})$ 
and a $G\times T^{k}$-equivariant function  $\hat{v} : P \rightarrow (\mathfrak{t}^{k})^{*} \otimes
\mathfrak{g}$.  They define a connection $\mathcal H^{\sigma}$ on $\sigma$,  with connection form
\begin{equation}\label{conn-form-A}
A := \pi_{0}^{*} A_{0} - \langle  \sigma^{*} \theta , \hat{v} \rangle =
 \pi_{0}^{*} A_{0} - \sum_{i=1}^{k}\sigma^{*} \theta_{i} \otimes  \hat{v}_{i},  
\end{equation}
where $\pi_{0} : P \rightarrow P_{0}$ is the natural projection,  $\langle \cdot , \cdot \rangle$ denotes the natural contraction between $\mathfrak{t}^{k}$ and $(\mathfrak{t}^{k})^{*}$,   and $\hat{v}_{i} = \langle \hat{v}, e_{i}\rangle : P \rightarrow \mathfrak{g}$. 
From the equivariance of $\hat{v}$, the functions $\hat{v}_{i}$ define sections of $\mathfrak{g}_{P}
=\pi^{*} \mathfrak{g}_{P_{0}}$ which are pullback of sections of $\mathfrak{g}_{P_{0}}$, i.e.
$\hat{v}_{i} =\pi^{*} \hat{v}_{i}^{B}$ for $\hat{v}_{i}^{B} \in \Gamma (\mathfrak{g}_{P_{0}})$
(we use the same notation for the 
functions  $\hat{v}_{i}$ , $\hat{v}_{i}^{B}$ and the corresponding sections of $\mathfrak{g}_{P}$ and
$\mathfrak{g}_{P_{0}}$ respectively).  
We shall denote by $\widetilde{X}^{A_{0}}$,  $\widetilde{Y}^{\pi^{*}_{0} A_{0}}$,  $\widetilde{Y}^{A}$,
the horizontal lifts of 
$X\in {\mathfrak X}(B)$ and $Y\in {\mathfrak X}(M)$ with respect to $\mathcal H^{\sigma_{0}}$,
$\pi_0^{*} \mathcal H^{\sigma_{0}}$ and 
${\mathcal H}^{\sigma}$ respectively. Here $\pi_0^{*} {\mathcal H}^{\sigma_{0}}\subset TP$ denotes 
the $G$-invariant horizontal distribution in $\sigma : P \rightarrow M$ defined by $(\pi_0^{*} {\mathcal H}^{\sigma_{0}})_p= (d_p\pi_0)^{-1}\mathcal{H}_{\pi_0(p)}^{\sigma_0}$, $p\in P$. It coincides with the kernel of the connection form 
$\pi^{*}_{0} A_{0}$.

\begin{lem} 
Let $(E = T^{*}M \oplus \mathfrak{g}_{P}\oplus TM, \Psi )$ be the heterotic Courant algebroid defined by $\sigma : P \rightarrow M$,  the connection 
$\mathcal H^{\sigma}$ with connection form (\ref{conn-form-A})  and a $3$-form $H\in \Omega^{3}(M)$
such that  $dH = \langle R^{{\mathcal H}^{\sigma}}\wedge R^{{\mathcal H}^{\sigma}}\rangle_{\mathfrak{g}}$,
together  
with  the  $\mathfrak{t}^{k}$-action (\ref{Psi-het}).  
Then  the connection  $\nabla^{\theta}$ 
on  $\mathfrak{g}_{P} = \pi^{*}(\mathfrak{g}_{P_{0}})$  defined in (\ref{nabla-theta})   is the pullback 
of the connection  $\nabla^{A_{0}}$ on  $\mathfrak{g}_{P_{0}}$  induced by $\mathcal H^{\sigma_{0}}$
(in the notation of Corollary \ref{data-extended}, $\mathcal G_{B} = \mathfrak{g}_{P_{0}}$
and $\nabla^{\theta , B} =\nabla^{A_{0}}$). 
\end{lem}

\begin{proof} We claim  that the horizontal  lift
${\widetilde{X_{a}^{M}}}^{\pi_{0}^{*} A_{0}}\in {\mathfrak X}(P)$ of the  $\pi$-vertical vector field
$X_{a}^{M}$ determined by $a\in \mathfrak{t}^{k}$ coincides with the fundamental vector field
$X_{a}^{P}$ of the $T^{k}$-action on $P$, i.e.
\begin{equation}\label{lifts}
{\widetilde{X_{a}^{M}}}^{\pi_{0}^{*} A_{0}}= X_{a}^{P},\ \forall a\in \mathfrak{t}^{k}.
\end{equation}
In order to prove (\ref{lifts}), 
let $U\subset B$  be open and sufficiently small such that, over $U$,
$\sigma_{0}$  is the trivial $G$-bundle and   
$$
\pi_{0} : P|_{\pi^{-1}(U)}= \pi^{-1} (U)\times G\rightarrow P_{0}|_U= U\times G,\ \pi_{0} (p, g)= (\pi (p), g).
$$
For any  $X\in {\mathfrak X} (\pi^{-1} (U))$, 
\begin{equation}\label{lift-2}
\widetilde{X}^{\pi_0^{*}A_{0}} = X -\langle  (\pi^{*}_{0} A_{0})(X), f_{i}^{*}\rangle X_{f_{i}}^{P},
\end{equation}
where $(f_{i})$ is a basis of $\mathfrak{g}$ with dual basis $(f_{i}^{*})$ and for $f\in \mathfrak{g}$, 
$X_{f}^{P}$ is the left invariant vector field on $G$ determined by $f$ (viewed as a vector field on
$\pi^{-1}(U)\times G$). On the other hand, $X_{a}^{M}\in {\mathfrak X}(\pi^{-1}(U))$, viewed as a vector field on
$P= \pi^{-1} (U)\times G$, satisfies $(\pi_{0})_{*} X_{a}^{M} =0$, since $\pi_{*} X_{a}^{M}=0$ and $\pi_0 = \pi \times \mathrm{Id}$ in our trivializations.
 Applying (\ref{lift-2})  to $X:= X_{a}^{M}$  
and using $(\pi_{0})_{*} X_{a}^{M}=0$  we obtain  $\widetilde{X_{a}^{M}}^{\pi_{0}^{*}A_{0}} =
X_{a}^{M}$. 
On the other hand, the action of $T^{k}$ on $P=\pi^{-1} (U)\times G$ is given by
$R_{g}(m ,\tilde{g}) = (mg, \tilde{g})$ 
which implies that $X_{a}^{P} = X_{a}^{M}. $ Relation (\ref{lifts}) follows.

Let $\nabla$ be the connection on $\mathfrak{g}_{P}$ induced by ${\mathcal H}^{\sigma}.$  
Its covariant derivative is given by
\begin{equation}\label{nabla-rr}
\nabla_{X} r\equiv\widetilde{X}^{A} (f) = \widetilde{X}^{\pi_{0}^{*} A_{0}}(f) -\theta_{i}(X) \mathrm{ad}_{\hat{v}_{i}}\circ f ,\
X\in{ \mathfrak X}(M), 
\end{equation}
where  $r\in \Gamma (\mathfrak{g}_{P})$ and $r\equiv f.$   Here we have used that 
\[ \widetilde{X}^{A} = \widetilde{X}^{\pi_{0}^{*} A_{0}} + \theta_i(X)X_{\hat{v}_i}^P\]
and the $G$-equivariance of $f$, which implies $X_{v}^P(f) = - \mathrm{ad}_{v}\circ f$ for all $v\in \mathfrak{g}$.
Applying relation  (\ref{nabla-rr}) to $X:= X_{a}^{M}$  and using (\ref{lifts}) we obtain
\begin{equation}\label{nabla-fct}
\nabla_{X_{a}^{M}} r  \equiv X_{a}^{P}(f)- \langle a, e_{i}^{*}\rangle \mathrm{ad}_{\hat{v}_{i}} \circ f,\ \forall 
a\in \mathfrak{t}^{k},
\end{equation}
which implies that the skew-symmetric derivation  $A_{a}$ of $\mathfrak{g}_{P}$,  from Lemma  \ref{cond-triv-ext},  is  given by
\begin{equation}\label{aad}
A_{a} (r)=({\mathcal L}_{X_{a}^{P}} - \nabla_{X_{a}^{M}}) r\equiv  \mathrm{ad}_{\langle \hat{v}, a\rangle }\circ f,\
\forall a\in  \mathfrak{t}^{k}.
\end{equation}
From its definition  (\ref{nabla-theta}) and relations (\ref{nabla-rr}),  (\ref{aad}),   $\nabla^{\theta}$ is given by
\begin{equation}\label{THETA}
\nabla^{\theta}_{X} r= \nabla_{X} r +\sum_{i=1}^{k} \theta_{i}(X) A_{i} (r)\equiv
\widetilde{X}^{\pi_{0}^{*}A_{0}}(f), 
\end{equation}
which implies that $\nabla^{\theta} = \pi^{*} \nabla^{A_{0}}$ as needed.
\end{proof}

Since 
$\nabla^{A_{0}}$ preserves the Lie bracket and scalar product of $\mathfrak{g}_{P_{0}}$, 
its curvature takes values in the bundle of skew-symmetric derivations 
of $\mathfrak{g}_{P_{0}}$ and is of the form $\mathrm{ad}_{\mathfrak{r}^{A_{0}}}$
where  $\mathfrak{r}^{A_{0}}\in \Omega^{2}(B, \mathfrak{g}_{P_{0}})$
(since $\mathfrak{g}$ is semi-simple).   
Like in  (\ref{HR:eq}), we decompose $H\in \Omega^{3}(M)$ using the connection $\theta .$

\begin{prop}\label{T-dual-het-prop} In the above setting, assume that 
\begin{align}
\nonumber& H_{(0)}^{ijs, B} =- \frac{1}{3} \langle [\hat{v}^{B}_{i},\hat{v}^{B}_{j}]_{\mathfrak{g} } , \hat{v}_{s}^{B}\rangle_{\mathfrak{g}}\\
\label{het-H}& H_{(1)}^{ij, B} := \frac{1}{2} \left( \langle \nabla^{A_{0}}\hat{v}_{i}^{B}, \hat{v}_{j}^{B}
\rangle_{\mathfrak{g} }
- \langle \nabla^{A_{0} }\hat{v}_{j}^{B}, \hat{v}_{i}^{B}\rangle_{\mathfrak{g}}\right) .
\end{align} 
and  that
the (closed) forms
\begin{equation}
H_{(2)}^{i, B} + 2\langle {\mathfrak r}^{ A_{0}}, \hat{v}_{i}^{B}\rangle_{\mathfrak{g}} -
 \langle\hat{v}_{i}^{B}, \hat{v}_{j}^{B}\rangle_{\mathfrak{g}} (d\theta_{j})^{B}.
\end{equation}
represent integral cohomology classes.  
Then $(E, \Psi )$ admits a $T$-dual which is a heterotic Courant algebroid. 
\end{prop}

\begin{proof} The conditions (\ref{het-H}) mean that  $(E, \Psi )$ belongs to the class 
of standard  Courant algebroids with $\mathfrak{t}^{k}$-action described in Example \ref{particular-class}
(in the notation of that example, $r_{i} = \hat{v}_{i}$ and $r_{i}^{B} =\hat{v}_{i}^{B}$). 
Let $(\tilde{E}, \tilde{\Psi })$ be a $T$-dual of $(E, \Psi )$,  
provided by  Theorem  \ref{statement-T-dual}.
Then $(\tilde{E}, \tilde{\Psi} )$  is defined 
on the total space of a principal
$T^{k}$-bundle $\tilde{\pi} : \tilde{M} \rightarrow B$, with connection form $\tilde{\theta} =\sum_{i=1}^{k} \tilde{\theta}_{i} e^{i}$, in terms of arbitrarily chosen sections
$\tilde{r}_{i}^{B}\in \Gamma (\mathfrak{g}_{P_{0}})$.
We define the pullback bundle $\tilde{\sigma } : \tilde{P}\rightarrow \tilde{M}$ of $\sigma_{0}: P_{0} \rightarrow B$ by  the map $\tilde{\pi}.$
The arguments from Theorem \ref{statement-T-dual} and the above lemma show that $\tilde{E}$ is the heterotic
Courant algebroid defined by the principal $G$-bundle $\tilde{\sigma}$, connection ${\mathcal H}^{\tilde{\sigma}}$
with connection form 
\begin{equation}
\tilde{A} = \tilde{\pi}_{0}^{*} A_{0} -\sum_{i=1}^{k}\tilde{\sigma}^{*}\tilde{\theta}_{i}\otimes \tilde{r}_{i}
\end{equation}
where $\tilde{\pi}_{0} :\tilde{P}\rightarrow P_{0}$ is the natural projection,  
$\tilde{r}_{i}= \tilde{\pi}^{*} (\tilde{r}_{i}^{B})\in \Gamma (\mathfrak{g}_{\tilde{P}})$  
and $3$-form 
$\tilde{H}$ is constructed as in Theorem \ref{statement-T-dual}
(in particular,  $\tilde{H}_{(0)}^{ijs, B}$ and  $\tilde{ H}_{(1)}^{ij, B}$ are given by (\ref{het-H}) with 
$\hat{v}_{i}^{B}$ replaced by $\tilde{r}_{i}^{B}$). 
\end{proof}

\begin{rem}{\rm The above treatment  provides an alternative view-point   for the heterotic  $T$-duality 
developed
in \cite{baraglia}. Heterotic Courant algebroids can be obtained 
from exact Courant algebroids  by a reduction procedure
described in \cite{baraglia} and the heterotic $T$-duality from \cite{baraglia}  was obtained as a reduction of the $T$-duality for exact Courant algebroids \cite{T-duality-exact}.  Our approach   is more direct and makes no reference to
exact Courant algebroids.}
\end{rem}

In our setting it is natural to relax the definition of a heterotic Courant algebroid \cite{baraglia} by 
allowing as structure groups of the principal bundle not only compact semi-simple Lie groups but any connected Lie group $G$ such that 
\begin{itemize}
\item[$P_1$)] 
$\mathrm{Ad} : G \rightarrow \mathrm{Aut}(\mathfrak{g},\langle \cdot ,\cdot \rangle_\mathfrak{g})_0$ is a covering for some 
invariant scalar product $\langle \cdot ,\cdot \rangle_\mathfrak{g}$ on $\mathfrak{g}= \mathrm{Lie}\, G$.  
 (Equivalently,
$\mathrm{ad} : \mathfrak{g} \rightarrow \mathrm{Der}(\mathfrak{g}, \langle \cdot ,\cdot \rangle_\mathfrak{g})$ is an isomorphism 
onto the Lie algebra of skew-symmetric derivations, cf.\ Remark \ref{quadraticLA:rem}.) As before, we restrict to scalar products of neutral signature.
\end{itemize}
The resulting 
Courant algebroids are transitive and the corresponding bundles of quadratic Lie algebras $\mathcal G$ have the property 
\begin{itemize}
\item[$P_2$)] 
$\mathrm{ad} : \mathcal{G} \rightarrow \mathrm{Der}(\mathcal{G})$ is an isomorphism. 
\end{itemize}

\begin{prop}\label{description-heterotic-CA}
The class of transitive Courant algebroids $E\rightarrow M$ over simply connected manifolds 
for which the bundle of quadratic Lie algebras $\mathcal{G}$ has the property $P_2$ coincides with the above 
(relaxed) class of heterotic Courant algebroids. 
\end{prop}
\begin{proof}
We first remark that the fibers  
$(\mathcal{G}, [\cdot , \cdot]_{\mathcal G}, \langle \cdot ,\cdot \rangle_\mathcal{G})|_p$, $p\in M$,  are all isomorphic to a fixed quadratic Lie algebra $(\mathfrak{g}, \langle \cdot ,\cdot \rangle_\mathfrak{g})$. In fact, for a transitive Courant algebroid $E$, any two fibers of $\mathcal G$ are related by parallel transport,  which 
preserves the tensor fields  $[\cdot , \cdot ]_{\mathcal G}$ and $\langle \cdot , \cdot \rangle_{\mathcal G}.$ 
Note that $\mathcal G$ satisfies $P_2$ if and only if $\mathfrak{g}$ satisfies $P_1$. 

Let us fix a basis a basis in $\mathfrak{g}$. 
The connection $\nabla$ in 
the bundle $\mathcal G$ induces a connection in the bundle $P$ of standard frames of $\mathcal{G}$. A frame is called  
{\cmssl standard} if its structure constants and the Gram matrix of the scalar product coincide with those of the underlying quadratic Lie algebra $(\mathfrak{g},\langle \cdot ,\cdot \rangle_\mathfrak{g})$ with respect to the fixed basis in $\mathfrak{g}$. 
The structure group of $P$ is $\mathrm{Aut}(\mathfrak{g},\langle \cdot ,\cdot \rangle_\mathfrak{g})$ and 
can be always reduced to the connected group $\mathrm{Aut}(\mathfrak{g},\langle \cdot ,\cdot \rangle_\mathfrak{g})_0$ by holonomy reduction if $M$ is simply connected. The property $P_2$ implies $\mathrm{Lie}\,  \mathrm{Aut}(\mathfrak{g},\langle \cdot ,\cdot \rangle_\mathfrak{g})_0\cong \mathfrak{g}$ and then $G := \mathrm{Aut}(\mathfrak{g},\langle \cdot ,\cdot \rangle_\mathfrak{g})_0$ satisfies $P_1$. In that case we can rewrite the bundle $\mathcal{G}$ as the adjoint bundle 
with connection induced from the connection in the principal $G$-bundle $P$. This shows that $E$ belongs to the (relaxed) class of heterotic Courant algebroids.
\end{proof}

V.\ Cort\'es: vicente.cortes@uni-hamburg.de\

Department of Mathematics and Center for Mathematical Physics, University of Hamburg,  Bundesstrasse 55, D-20146, Hamburg, Germany.\\

\noindent
L.\ David: liana.david@imar.ro\

Institute of Mathematics  `Simion Stoilow' of the Romanian Academy,   Calea Grivitei no. 21,  Sector 1, 010702, Bucharest, Romania.
\end{document}